\documentclass[a4, reqno]{amsart}
\usepackage{amssymb}

\makeatletter
\def\@tocline#1#2#3#4#5#6#7{\relax
  \ifnum #1>\c@tocdepth 
  \else
    \par \addpenalty\@secpenalty\addvspace{#2}%
    \begingroup \hyphenpenalty\@M
    \@ifempty{#4}{%
      \@tempdima\csname r@tocindent\number#1\endcsname\relax
    }{%
      \@tempdima#4\relax
    }%
    \parindent\z@ \leftskip#3\relax \advance\leftskip\@tempdima\relax
    \rightskip\@pnumwidth plus4em \parfillskip-\@pnumwidth
    #5\leavevmode\hskip-\@tempdima
      \ifcase #1
       \or\or \hskip 1em \or \hskip 2em \else \hskip 3em \fi%
      #6\nobreak\relax
    \dotfill\hbox to\@pnumwidth{\@tocpagenum{#7}}\par
    \nobreak
    \endgroup
  \fi}
\makeatother

\usepackage[left=3cm,right=3cm,top=2.9cm,bottom=2.9cm]{geometry}
\usepackage{enumitem}
\usepackage{amsmath}
\usepackage{mathtools}
\usepackage{amsfonts}
\usepackage{color}
\usepackage{graphicx,tikz}
\usepackage[bookmarksnumbered,colorlinks, linkcolor=blue, citecolor=red, pagebackref, bookmarks, breaklinks]{hyperref}
\usepackage{dsfont}

\newtheorem{thm}{Theorem}[section]
\newtheorem{cor}[thm]{Corollary}
\newtheorem{lema}[thm]{Lemma}
\newtheorem{prop}[thm]{Proposition}
\theoremstyle{definition}
\newtheorem{defn}[thm]{Definition}
\theoremstyle{remark}
\newtheorem{exam}[thm]{Example}

\newtheorem{rem}[thm]{Remark}
\numberwithin{equation}{section}
\newcommand{\R}{\mathbb R}
\newcommand{\N}{\mathbb N}
\newcommand{\Z}{\mathbb Z}

\newcommand{\J}{\mathcal{L}}
\newcommand{\E}{\mathcal{E}_{p,\Omega}}

\DeclareMathOperator{\supp}{supp}
\DeclareMathOperator{\osc}{osc}
\DeclareMathOperator{\sign}{sign}
\DeclareMathOperator{\Lip}{Lip}

\def\LL{\mathrm{L}}

\title[{Sharp regularity estimates  for $0$-order $p$-Laplacian evolution problems}]{Sharp regularity estimates \\ for $0$-order $p$-Laplacian evolution problems}

\author{Matteo Bonforte}
\address{Matteo Bonforte: Departamento de Matemáticas, Universidad Autónoma de Madrid, Campus de
Cantoblanco, 28049 Madrid, Spain}
\email{mailto:matteo.bonforte@uam.es}
\urladdr{http://verso.mat.uam.es/~matteo.bonforte}

\author{Ariel Salort}
\address{Ariel  Salort: Departamento de Matemáticas, FCEN - Universidad de Buenos Aires, and\\
$0+\infty$ Building, Ciudad Universitaria (1428), Buenos Aires, Argentina.}
\email{asalort@dm.uba.ar}
\urladdr{http://mate.dm.uba.ar/~asalort}

\begin{document}



\begin{abstract}
We study regularity properties of solutions to nonlinear and nonlocal evolution problems driven by the so-called \emph{$0$-order fractional $p-$Laplacian} type operators:
$$
\partial_t u(x,t)=\J_p u(x,t):=\int_{\R^n} J(x-y)|u(y,t)-u(x,t)|^{p-2}(u(y,t)-u(x,t))\,dy\,,
$$
where  $n\ge 1$, $p>1$, $J\colon\R^n\to\R$ is a bounded nonnegative function with compact support, $J(0)>0$ and normalized such that $\|J\|_{\LL^1(\R^n)}=1$, but not necessarily smooth. We deal with Cauchy problems on the whole space, and with Dirichlet and Neumann problems on bounded domains. Beside complementing the existing results about existence and uniqueness theory, we focus on sharp regularity results in the whole range $p\in (1,\infty)$. When $p>2$, we find an unexpected $\LL^q-\LL^\infty$ regularization: the surprise comes from the fact that this result is false in the linear case $p=2$.

We show next that bounded solutions automatically gain higher time regularity, more precisely that $u(x,\cdot)\in C^p_t$. We finally show that solutions preserve the regularity of the initial datum up to certain order, that we conjecture to be optimal ($p$-derivatives in space). When $p>1$ is integer we can reach $C^\infty$ regularity (gained in time, preserved in space) and even analyticity in time.  The regularity estimates that we obtain are quantitative and constructive (all computable constants), and have a local character, allowing us to show further properties of the solutions: for instance, initial singularities do not move with time.  We also study the asymptotic behavior for large times of solutions to Dirichlet and Neumann problems. Our results are new also in the linear case and are sharp when $p$ is integer. We expect them to be optimal for all $p>1$, supporting this claim with some numerical simulations.
\end{abstract}

\vspace*{-1cm}
\maketitle
\vspace{-1cm}
\setlength{\parskip}{0.06em}\small
\tableofcontents
\setlength{\parskip}{0.5em}
\normalsize
\vspace{-1.5cm}

\newpage

\section{Introduction}

In this paper we study nonlinear and nonlocal evolution problems whose diffusion part is governed by the so-called \emph{$0$-order fractional $p-$Laplacian} type operators, which take the form
$$
\J_p u(x,t):=\int_{\R^n} J(x-y)|u(y,t)-u(x,t)|^{p-2}(u(y,t)-u(x,t))\,dy\,,
$$
where $p>1$, $J\colon\R^n\to\R$,  with $n\ge 1$ is a nonnegative function, which we assume to be bounded and with compact support, $J(0)>0$ and normalized such that $\|J\|_{\LL^1(\R^n)}=1$, but not necessarily smooth.

When $p=2$, $\J_p$ corresponds to a subclass of zero order L\'evy operators, which are particularly nasty to treat due to their lack of regularization properties, which L\'evy operator of positive order possess. When $p>1$, such operators model a special family of jump processes and have many applications: modeling diffusion processes and phase transitions, image processing, populations dynamics, etc. see for instance \cite{BC99-2, BC99,  BHZ06,FW98, KOJ05}. A more detailed discussion can be found in Section \ref{ssec.model} below.  

The evolutionary equation $u_t=\J_p(u)$  has been studied by many authors in the last years, see the monograph  \cite{AMRT} where the state of the art up to 2010 is presented in an excellent way. The authors of \cite{AMRT}   consider also problems with different boundary conditions, see also \cite{AMRT08a, FR15, GMR09,  K19,MRT17}, and they study the asymptotic behavior of solutions, see also \cite{CCR,  CEQW16, CEQW16a,CEQW12, IR08}.

In this paper, we shall study different evolution problems\footnote{We shall focus our main attention to the Cauchy problem, but we also provide a quite complete set of results also for the Dirichlet and Neumann problems on bounded domains, for which we also analyze the asymptotic behavior and long time decay. See the end of this section and Section \eqref{sec.nd.intro} for more details. } driven by the above $0$-order fractional $p-$Laplacian type operators, for all $p>1$. Although we provide sharp results in the whole range $p\in (1,\infty)$, including the linear case, we shall pay more attention to the nonlinear range of parameters $p>2$, where a surprising $\LL^q-\LL^\infty$ smoothing effect happens to be true: the surprise comes from the fact that analogous smoothing effect are false when $p\in (1,2]$, especially in the linear case $p=2$, as we shall discuss below. We also show that bounded solutions automatically gain higher time regularity and preserve the regularity of the initial datum up to certain order, that we conjecture to be optimal. When $p>1$ is integer (with $p=2$ as a particular case) we can reach $C^\infty$ regularity (gained in time, preserved in space) and even analyticity (in time). All of our regularity results are supported by quantitative and constructive estimates (i.e. all the constants are computable).  

We first consider solutions of the following \emph{Cauchy problem}
\begin{align} \label{eqC} \tag{C}
\begin{cases}
u_t(x,t)=\J_p u(x,t), &x\in \R^n, t\in (0,\infty)\\
u(x,0)=u_0(x),&x\in \R^n,
\end{cases}
\end{align}
starting from initial data $u_0\in \LL^q(\R^n)$, with $q\in [1,\infty]$. The definition of solution is particularly simple in this framework: \textsl{we say that $u$ is a solution to the Cauchy problem \eqref{eqC}, if it satisfies the equation almost everywhere in space and time. }We shall be more precise in Definition \ref{Def.Soln}, but in order to fix ideas, we say that $u\in W^{1,1}((0,\infty); \LL^p(\R^n))$ is a solution to the Cauchy problem \eqref{eqC} if $u_t(x,t)=\J_p u(x,t)$ for almost every $x\in \R^n$ and all $t> 0$,  and takes the initial datum in the strong $\LL^q$ topology.

\noindent\textit{Existence and uniqueness of solutions, }as well as some other useful basic properties, have been first obtained via nonlinear semigroup theory based on a version of the celebrated Crandall-Liggett Theorem, when $u_0\in \LL^q$  for some, but not all, $q\in[1,\infty)$, see \cite{AMRT} and Section \ref{section3} for further details.  We shall complement those result with a quite complete $\LL^q$ theory, based on approximations by means of suitable $\LL^2$ solutions constructed using the simpler gradient flow on Hilbert spaces approach, that is, via Brezis-Komura Theorem in $\LL^2$. This has been done in Theorem \ref{existencia}. 

\noindent Once a basic theory of existence and uniqueness is established, the natural question to be addressed is
\begin{center}
  \textit{Which initial data produce bounded solutions? }
\end{center}
After answering to this question,  we shall address the next natural one, that is space-time regularity: 
\begin{center}
  \textit{In which case bounded solutions are also continuous or even smoother?\\ Are there classical solutions?}
\end{center}
These questions are hard to investigate for the equation at hand, since the diffusion operator is of order zero, hence, a priori no regularization in space has to be expected. Nonetheless, we can prove that integrable data produce bounded solutions when $p>2$, something false even for $p=2$. In turn, this  allows to show higher time regularity, and conservation of the local modulus of continuity of the initial datum, up to order $p$. We shall also clarify what we mean by classical solutions, but the short message is that  \textit{bounded solutions are classical}.

In this paper we answer the above questions, by providing a quite complete theory that besides settling some existence, uniqueness and contractivity estimates in $\LL^q$, it ensures some surprising new smoothing effects, and higher regularity estimates (up to $C^\infty$, and even analyticity when $2\le p\in \N$), that we conjecture to be sharp. Our results extend  in many unexpected directions the known ones, see the monograph \cite{AMRT}, the more recent survey \cite{R20} and references therein. To the best of our knowledge, the regularity results are unexpectedly new also in the linear case $p=2$, as we detail below.

\noindent\textbf{The linear case. }When $p=2$  solutions to the Cauchy Problem \eqref{eqC} do not regularize as expected from a linear heat-type flow. For this reason $\J_2$ can be considered a ``black sheep'' among Levy operators.  In Theorem 1.4 and Lemma 1.6  of \cite{AMRT}
(see also \cite[Section 5]{BE}), an integral representation of the unique solution is obtained, and it is shown that solutions do not exhibit any smoothing effect. Indeed, it is not difficult to show that solutions to the Cauchy problem have the following form:
\begin{equation}\label{lin.sol.form}
u(x,t)= e^{-t} u_0(x) + W(x,t)\,,
\end{equation}
where $W(x,t)$ is a suitable (possibly smooth, depending on $J$) function given by
\[
W(x,t):=\int_{0}^{t}e^{-(t-\tau)} \int_{\R^n} u(y,\tau)J(x-y)dy\,d\tau\,.
\]
On the one hand, it is quite clear that a priori time regularity can improve, on the other hand, it is also quite clear that the spatial regularity of $u( \cdot,t)$ cannot be better than the one of the initial datum, due to the first term of \eqref{lin.sol.form}. For this reason, smoothing effects ($\LL^q$--$\LL^\infty$ regularization in space) are \textit{not} possible in this case (when $q<\infty$).   The basic properties (existence, uniqueness, comparison, mass conservation, etc.) of this linear flow are proven in \cite{AMRT} via Fourier transform methods, exploiting a representation formula, similar to the one above (in spirit, yet ``less practical'' to use).

In this case, we are able to prove an \textit{optimal regularity result}: we show that \textit{bounded solutions are $C^{\infty}$ in time, even analytic } (we  provide quantitative and explicit bounds), while \textit{in space they preserve the regularity of the initial datum as well as of all its derivatives of arbitrary order}. This result is new, to the best of our knowledge, and follows from elementary proofs, see Section \ref{ssec.Reg.LIN}, Theorem \ref{thm.c.conservation} and Corollary \ref{Coro1.c.intro}. Our approach allows to reach sharp regularity results and it avoids the use of Fourier transform, a key tool in the analysis performed in \cite[Chapter 1]{AMRT}.  Our method  has the advantage of being completely local, hence ready to be used also in the case of problems posed on bounded domains with different boundary conditions, where Fourier analysis gets even more involved. It also lay a path to attack the same issues in the more delicate nonlinear case.

\noindent\textbf{The nonlinear case. Regularity in space and time. }When $p\ne 2$, a priori, nothing can be expected to improve spatial regularity, but there is hope for the time regularity, inspired by the linear case. Let us provide a short overview of our new main results.  \begin{itemize}[leftmargin=*]\itemsep2pt \parskip2pt \parsep0pt
\item  When $p>2$ there is not an explicit formula for the solutions, as \eqref{lin.sol.form} in the case $p=2$. A priori this provides an extra difficulty. However, exploiting the convexity of the nonlinearity, we are able to show remarkable regularization properties of the solutions. First we prove quantitative $\LL^q$--$\LL^\infty$ smoothing effects, namely, that for all $q\ge 1$  we have
\[
u_0\in \LL^q(\R^n)\qquad\mbox{implies}\qquad \|u(t)\|_{\LL^\infty(\R^n)} \lesssim  t ^{-\frac{1}{p-2}}   +  \|u_0\|_{\LL^q(\R^n)}\qquad\mbox{for all  $t>0$}\,.
\]
We then show that solutions are regular in time, namely $u(x,\cdot)\in C^{1,1}_t((0,\infty))$ for \text{a.e.} $x\in \R^n$\,. Indeed, we can bootstrap the time regularity to get that $u(x,\cdot)\in C^{[p],p-[p]}_t((0,\infty))$ for \text{a.e.} $x\in \R^n$, and we conjecture that \textit{this regularity is optimal}: it is indeed sharp \textit{when $p\ge 2$ is an integer}, since we can reach $C^\infty$ regularity with quantitative estimates that allow to \textit{conclude even analyticity.}\\
As far as spatial regularity is concerned, we show that solutions keep the same modulus of continuity of the initial datum and of all its derivatives up to order $p$ (at least when the kernel $J$ is sufficiently regular). When $p>1$ is integer, we can extend this result to derivatives of arbitrary order.

Summing up, when $p>2$, merely integrable solutions are bounded, have continuous time derivatives up to order $p$, and preserve the smoothness of the initial data up to order $p$. When $p$ is integer, this can be extended to derivatives of arbitrary order, and we obtain analyticity in time.  Rigorous statements are formulated in Theorems \ref{thm.exist.intro}, \ref{thm.C.reg.intro} and \ref{thm.c.conservation}.  
\item As we have already discussed, when $p=2$, our results show that bounded solutions are $C^{\infty}$ and analytic in time and they preserve the regularity of the initial datum as well as of all its derivatives. See Theorem \ref{thm.c.conservation} and Corollary \ref{Coro1.c.intro}. These results are new to the best of our knowledge.
\item When $p\in (1,2)$ the situation is similar to the case $p=2$, but solutions have H\"older continuous time derivative, namely
    $u_t\in C^{p-[p]}_t$, and they preserve the continuity of the initial datum up to order $p$.  See the precise statements in Theorem \ref{thm.c.conservation}.
\end{itemize}

\noindent\textbf{Classical solutions. }Another important point in the study of the regularity of solutions to the zero-order $p$-Laplacian, is to identify classical solutions, i.e. solutions that satisfy the equation at all points in the sense of continuous functions: as a consequence of the above discussion we deduce that \textit{bounded solutions, starting from continuous initial data, are classical for all $p>1$.}

\noindent\textbf{Nonlinear preservation of regularity, analyticity and optimality. }When $p>1$, we have seen that regularity in time of bounded solutions improves up to order $p$ \footnote{  ``up to order $p$'' has to be understood in the sense of all the $k^{th}$-order derivatives  up to the integer order $k=[p]$ ($[p]$ here is the integer part of $p$), plus the finite H\"older norm with $p-[p]\in (0,1]$ of all the $[p]^{th}$-order derivatives.}, while the spatial regularity is only preserved up to order $p$. However, when $p\ge 2$ is integer, we obtain analyticity in time and preserve the spatial continuity of the initial data and all its derivatives of arbitrary order. We conjecture that this regularity is optimal in all cases: on the one hand, at least in the linear case $p=2$ it corresponds to the regularity expressed by the ``representation formula'' \eqref{lin.sol.form}, hence is optimal. For integer $p>2$ we see no reason to think that our result is not optimal, even if there is not a representation formula.  On the other hand, when $p\ne 2$, we do not have explicit examples of solutions at hand, but we have done some numerical simulations that suggest that also in this case, our result should be optimal, especially in the fine analysis of the loss of regularity due to the lack of regularity of the kernel.

\noindent\textbf{Local regularity estimates for nonlocal problems and evolution of singularities. }We can also consider a dual point of view: analyze the evolution of the initial singularities. This is possible due to the local nature of our estimates (something remarkable for a nonlocal problem). A first consequence in this direction is, roughly speaking,  that \textit{spatial singularities do not move in time}, or to be more precise, the singular set of $u(t)$ (where $u(\cdot,t )$ is not continuous in the spatial variables)  will always be contained in the singular set of the initial datum $u_0$. This also implies that \textit{new discontinuities cannot be created in future times}. A third consequence of the locality of the regularity estimates is that it allows to \textit{deal with different boundary conditions, }we shall focus on two important cases, \textit{the Dirichlet and Neumann problems on bounded domains. }See Remark \ref{Rem.Locality} for more details.

\newpage

\noindent\textbf{Sharpness of the results: numerical evidence. }It is quite clear that our results are sharp when $p>1$ is integer. We believe that our regularity results are sharp also for non-integer $p$: on the one hand, to the best of our knowledge, there do not exist explicit examples of solutions in the literature, which could support or disprove the sharpness of our results.
On the other hand, we have some numerical examples that somehow confirm our beliefs, see figures\footnote{These plots have been obtained using a modification of a MatLab code developed in the paper \cite{DTL3}, courtesy of the authors.}  \ref{fig.1} and \ref{fig.2}. Some comments are in order.

\begin{figure}[h]
  \includegraphics[scale=0.105]{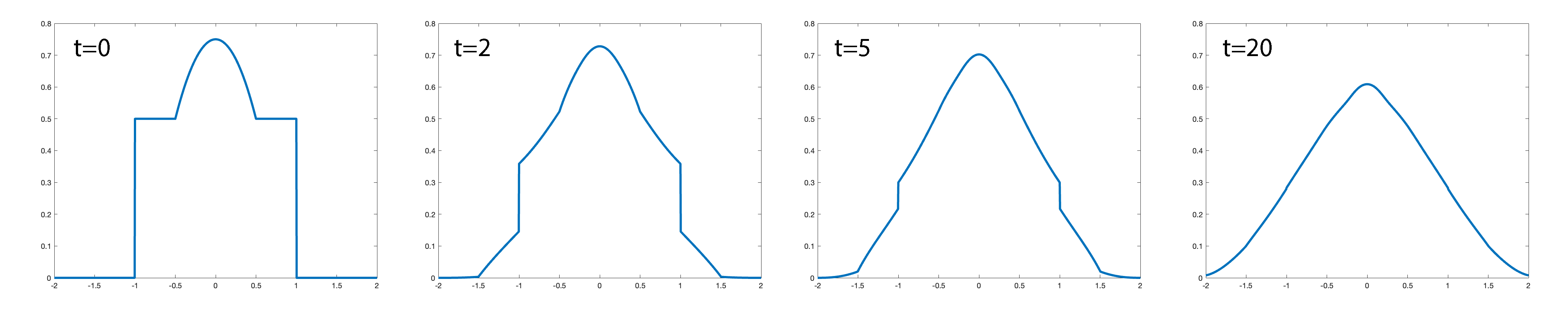}
  \caption{$p=3$, $J(x)=1$ in $[-\tfrac12,\tfrac12]$, $J(x)=0$ otherwise.}\label{fig.1}
\end{figure}

\begin{figure}[h]
  \includegraphics[scale=0.105]{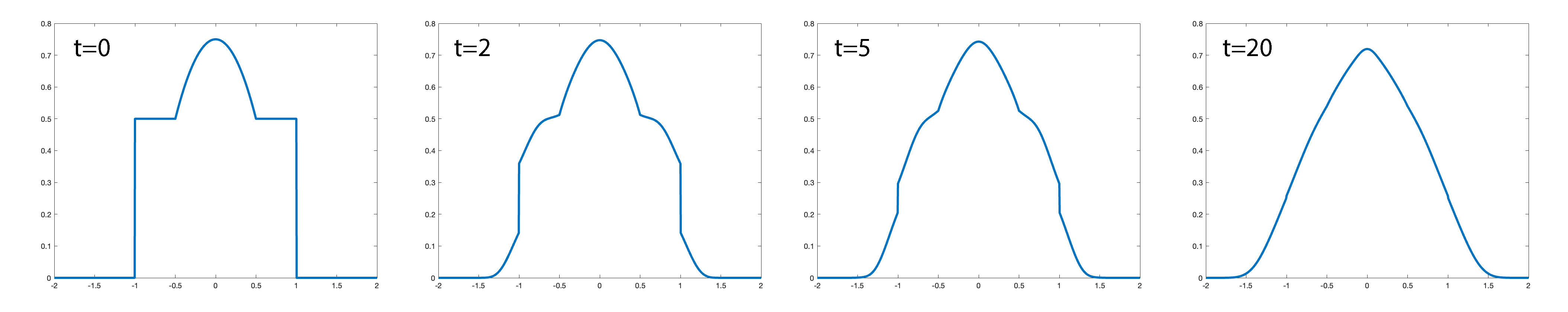}
  \caption{$p=3$, $J(x)=630(\tfrac14-x^2)^4$ in $[-\tfrac12,\tfrac12]$, $J(x)=0$ otherwise.}\label{fig.2}
\end{figure}

\begin{itemize}[leftmargin=*]\itemsep2pt \parskip3pt \parsep0pt
\item We shall now comment about differences in regularity between Figure \ref{fig.1}, where $J$ is less regular (a step function), and \ref{fig.2} where $J$ is smooth. What we expect from our regularity results is that, locally, when $J$ is smoother than $u_0$, the initial regularity is preserved. On the contrary,when $J$ is less regular than $u_0$, we can only ensure the preservation of the worst modulus of continuity: indeed if $J\in C^{\omega_J}(B_1)$, even if $u_0\in C^1(B_1)$, we can only expect $u(t)\in C^{\omega_J}(B_1)$. This is also clear comparing Figures \ref{fig.1} and \ref{fig.2} in a small interval $I(x_0)$ around $x_0= 1.5$ (the same happens around $x_0=-1.5$)\,, where the initial data is set to be zero in both cases, hence $u_0\in C^\infty(I(x_0))$.  In Figure \ref{fig.1}, a (Lipschitz or H\"older type) corner appears in  $I(x_0)$, while in Figure \ref{fig.2} the solution is still smooth in  $I(x_0)$.

\item In both cases it is quite clear that \textit{the singularities in space   do not evolve in time, }and these are the only possible singularities starting from bounded initial data: these are the ``jumps'' of $u_0$, and  we have set them  at $x=\pm 1$ in both simulations. From our theoretical results it is quite clear that  the  local oscillation is preserved in time, provided $J$ is (locally) more regular. In any case, the worst modulus of continuity (between $J$ and $u_0$) is preserved. The difference can be appreciated looking at the difference between Figure \ref{fig.1}, where $J$ is less regular (hence some corners may appear) and \ref{fig.2} where $J$ is smooth (and all the corners are immediately smoothed out). This local preservation of continuity implies that whenever there is an initial closed discontinuity set, then outside that set the solution has to be regular, as a consequence of our local estimates. Therefore such discontinuity sets are stationary in time. We also observe that the size of the jump decreases with time, and that bounded solutions have the tendency to be continuous after a certain amount of time, which agrees with the oscillation estimates, valid both in local and global versions, which are essentially non-increasing in time. This first observed in the case of the Total Variation Flow (1-Laplacian) by Figalli and the first author in \cite{BFJDE}. The same remarks can be applied to the modulus of continuity of all higher derivatives: in connection to the previous comment, when a Lipschitz corner appears at $x_0=1.5$ in Figure \ref{fig.1} it is quite clear from the simulation that it does not move in time, the reason being the same as for the modulus of continuity of $u$. Note that a Lipschitz corner for $u$, is nothing but a jump discontinuity for the first derivative, to which our previous reasoning applies, supported by the same local estimates. This shows a different behavior from the smooth situation of Figure \ref{fig.2}, where no corners appear.
\end{itemize}

\subsection{\textbf{The case of bounded domains.}}We also consider problems posed in bounded open sets $\Omega\subset \R^n$, and show analogous results: existence, uniqueness, contractivity and higher regularity   (in the interior, we do not study the boundary regularity).   The proofs in this case are essentially the same as in the Cauchy problem, in particular we stress that the proofs of the regularity estimates are ``local in essence'', hence they can be adapted to these cases as well (under some extra assumptions on $J$ and $\Omega$ in the Neumann case), as we shall explain in Section \ref{sec.nd.intro}. In addiction, in this case, we study the asymptotic behavior of solutions for large times, which differs according to the different boundary conditions.

More specifically, given $u_0\in \LL^q(\Omega)$ with $q\in [1,\infty]$ we consider the  homogeneous \emph{Dirichlet} problem
\begin{align} \label{eqD} \tag{D}
\begin{cases}
u_t(x,t)=\J_p u(x,t), &x\in \Omega, t>0,\\
u(x,t)=0, &x\in \R^n\setminus  \Omega, t>0,\\
u(x,0)=u_0(x), &x\in \Omega.
\end{cases}
\end{align}
Sometimes it is useful to rewrite this problem as follows
\begin{align*}
\begin{cases}
 u_t(x,t)=\widetilde \J_{p,\Omega}u(x,t)  , &x\in \R^n, t>0,\\
u(x,0)=u_0(x), &x\in \Omega,
\end{cases}
\end{align*}
where, writing $\tilde u(x,t)$ to mean that $u(x,t)$ when $x\in \Omega$ and to $\tilde u(x,t)=0$ when $x\not\in \Omega$, we define
$$
\widetilde \J_{p,\Omega}u(x,t):=\int_{\R^n} J(x-y)|\tilde u(y,t)-u(x,t)|^{p-2}(\tilde u(y,t)-u(x,t))\,dy.
$$
We consider next the  homogeneous \emph{Neumann} problem
\begin{align} \label{eq} \tag{N}
\begin{cases}
u_t(x,t)=\J_{p,\Omega} u(x,t), &x\in \Omega, t\in (0,\infty),\\
u(x,0)=u_0(x),&x\in \Omega,
\end{cases}
\end{align}
Here, we need to define the operator in a slightly different way, as follows
$$
\J_{p,\Omega} u(x,t):=\int_\Omega J(x-y)|u(y,t)-u(x,t)|^{p-2}(u(y,t)-u(x,t))\,dy.
$$

\subsection{Relevance of the model and possible applications. }\label{ssec.model}The model under consideration, has a simple probabilistic interpretation when $p=2$, being a zero-order Levy operator. For instance, if we define $u(x,t)$ to be the density of a population at point $x$ and at time $t$, and $J(x-y)$ as the probability of jumping from location $y$ to location $x$, then the convolution $\int_{\R^n} J(x-y)u(y,t)\,dy$ represents the rate at which individuals arrive at $x$ from other locations. Similarly, $-\int_{\R^n} J(x-y)u(x,t)\,dy$ represents the rate at which individuals leave location $x$ to travel to other sites. The absence of external sources leads to the density $u$ satisfying \eqref{eqC} (with $p=2$).   When $p\ne 2$, we shall think that the jumping probability depends on the density itself, and the model becomes more involved. Still, in the case of classical (normalized) $p$-Laplacian or of the fractional one, there are probabilistic interpretations in terms of tug-of-war-like games, see \cite{CCF12,MPR10, PS09}. When the diffusion occurs in an open bounded set $\Omega\subset \R^n$, and there is no flux of individuals across the boundary, we consider the homogeneous Neumann equation \eqref{eq}. Finally, when diffusion occurs in the whole $\R^n$ and $u$ is assumed to vanish outside $\Omega$, this means that any individuals leaving $\Omega$ shall die, which gives rise to the homogeneous Dirichlet problem \eqref{eqD}. For further information, we refer to \cite{AMRST08, AMRT, AMRT2, BC99-2, FW98, R20} and the references therein.

These type of nonlocal problems were extensively studied in the last decade. The simplest linear model (i.e. $p=2$) and variations of it, have been widely used to model diffusion processes in many contexts \cite{BC99,BHZ06}, phase transitions \cite{BC99-2,FW98}, image processing \cite{BCM06,KOJ05}, etc. See also \cite{AMRST08,AMRT,  R20}.

\noindent\textbf{Connection with the classical $p$-Laplacian and its fractional variants. }
One of the reasons why we call these operators ``zero-order $p$-Laplacians'' is the following. Consider the family of rescaled kernels:
\[
J_\varepsilon(x)=\frac{c_{J,p}}{\varepsilon^{n+p}}J\left(\frac{x}{\varepsilon}\right)
\]
for some explicit $c_{J,p}>0$. In \cite{AMRT} it is shown (under some assumptions on the kernel $J$) that solutions to the Cauchy problem with the nonlocal operators $\mathcal{L}_{J_\varepsilon}u$ converge in $\LL^\infty((0,\infty): \LL^1(\R^n))$  to solutions to the Cauchy problem for the classical $p$-Laplacian operator $\Delta_p u=\nabla\cdot(|\nabla u|^{p-2}\nabla u)$.  Also, in \cite{DTL} the authors prove that also the operator with kernel $J_\varepsilon$ converge to the $p$-Laplacian operator (not only the solution) in their natural topology. Similar results are obtained in \cite{DTMO} for the fractional $(s,p)$-Laplacian by considering the approximations with the kernel $J_\varepsilon(x)=c\, \varepsilon^{-(n+sp)}J (x/\varepsilon)$.

It is also worth noticing that  our zero-order operators do not include directly numerical approximations for local or nonlocal $p$-Laplacians, see \cite{DTL2,DTL3, DTMO}, which would require the kernel $J$ to be an atomic measure, and which is not directly admissible in our setting. But an important point is that many of our (new) ideas and techniques, which are based on elementary estimates, can be useful also in the numerical setting, and could lead to improved stability or consistency estimates (in a stronger localized version). On the one hand, the smoothing effects and the time regularity can easily be adapted to hold also in the numerical schemes, since they do not depend on the regularity of $J$. On the other hand, for instance in \cite{DTL3, DTMO}, the authors show that solutions preserve the same global H\"older constant as the initial datum:  we obtain something of a different nature, more local. Indeed, we show that any local modulus of continuity is ``preserved'', in the sense that it can differ by a multiplicative constant (which we estimate precisely), but it stays in the same class: for instance if $u_0\in C^{\alpha}(B_{r_0}(x_0))$, then also $u(t)\in C^{\alpha}(B_{r_0}(x_0))$, and this class can change from ball to ball, but it is preserved in time (and the same happens for every other modulus). Contrarily to \cite{DTL3, DTMO}, where only the ``first'' modulus of continuity is preserved, we also obtain the same results for higher derivatives (at least up to order $p$), something that would be extremely useful in numerical approximations and has its own interest. We refer to Section \ref{Sec.1.4} for more details on the state of the art, related results and the ideas of the proof.

\medskip

\subsection{Precise statement of the main results} ~

We present a more detailed description of our results below.   With respect to the existing literature, see for instance \cite{AMRT,R20}, we complete  the existence and uniqueness theory allowing data in arbitrary $\LL^q$ spaces. Then we deal with regularity issues, starting from boundedness of solutions. On the one hand, when $p\in (1,2]$, the only way to obtain bounded solutions is to start from bounded initial data, as already mentioned. On the other hand, when $p>2$ we show surprising smoothing effects that guarantee that merely integrable data produce bounded solutions. We show next the sharp regularity of bounded solutions: roughly speaking,  these solutions automatically gain $p$ bounded derivatives in time, and maintain the moduli of continuity of the initial data, up to $p$ derivatives in space. We can reach $C^\infty$ both in space and time, when $p$ is integer. We support our regularity results with quantitative and constructive (i.e. all constants are computable) estimates.

\subsubsection{Main results for the Cauchy problem}

We shall begin by stating the basic results.
\begin{thm}[Cauchy problem: Existence, uniqueness, $\LL^q$--contractivity] \label{thm.exist.intro}
Let $p>1$, $n\ge 1$, and let  $J\colon\R^n\to\R$ be a bounded nonnegative function with compact support, $J(0)>0$ and normalized such that $\|J\|_{\LL^1(\R^n)}=1$.
When $p \ge 2$, let the initial datum $u_0\in \LL^q(\R^n)$, for any $q\in [1,\infty)$. Then, there exists a unique solution $u$ to the Cauchy problem \eqref{eqC}. Moreover, letting $u_0,v_0 \in \LL^q(\R^n)$ and $u,v$ the corresponding solutions of \eqref{eqC},  the following contractivity estimate holds
\begin{equation}\label{Thm.1.Cauchy.i}
\|u(t)-v(t)\|_{\LL^q(\R^n)}\le \|u_0-v_0\|_{\LL^q(\R^n)}\qquad\mbox{for all }t\ge 0\,.
\end{equation}
When $p\in (1,2)$, all the above results hold for data in $\LL^q(\R^n)$, with $q\in \{p, 2\}$.

If moreover $u_0,v_0\in \LL^\infty(\R^n)\cap\LL^q(\R^n)$ then inequality \eqref{Thm.1.Cauchy.i} holds for any $p>1$ and all $q\in [1,\infty]$.
\end{thm}\vspace{-.5cm}
\begin{proof}
  The existence and uniqueness of solutions in the cases $p>2$ and initial datum $u_0\in \LL^q(\R^n)$ with  $q\in [1,\infty)$, and $p\in (1,2]$ and $u_0\in \LL^q(\R^n)$ with $q\in \{2,\infty\}$  is proved in Theorem \ref{existencia}. The case $p\in (1,2]$ and $u_0\in \LL^p(\R^n)$ is proved in Theorem 6.3.7 of \cite{AMRT}. The $\LL^q$--contractivity \eqref{Thm.1.Cauchy.i} is proved in Proposition \ref{norm.decreasing.difference}.
\end{proof}\vspace{-4mm}

As far as the gain of regularity is concerned, this is our first sharp result.

\begin{thm}[Cauchy problem: boundedness and time regularity] \label{thm.C.reg.intro}
Let $p>1$, $n\ge 1$, and let  $J\colon\R^n\to\R$ be a bounded nonnegative function with compact support, $J(0)>0$ and normalized such that $\|J\|_{\LL^1(\R^n)}=1$. Let  $u_0 \in \LL^q(\R^n)$ for some $q\ge 1$ and $u $ be the corresponding solutions of \eqref{eqC}.\vspace{-2mm}
\begin{enumerate}[leftmargin=*, label=(\alph*)]\itemsep1pt \parskip1pt \parsep0pt
\rm \item \it {\rm Boundedness of solutions. }The solution $u(t)$  is   bounded in one of the following cases.
\begin{itemize}[leftmargin=*]
\item[$\circ$] {\rm $\LL^q$--$\LL^\infty$ smoothing effects when $p>2$. }Let $u_0\in \LL^q(\R^n)$, $q\in [1,\infty]$, then we have
\begin{equation}  \label{Thm.1.Cauchy.smooth1}
\|u( t)\|_{\LL^\infty(\R^n)} \leq
\frac{\mathsf{K}_1}{t^{\frac{1}{p-2}}}   + \mathsf{K}_2\|u( t_0)\|_{\LL^q(\R^n)}  \quad \text{ for all } 0\le t_0\leq  t\,,
\end{equation}
where $\mathsf{K}_1, \mathsf{K}_2>0$ depend only on $p$, $q$, $n$ and $\|J\|_{\LL^\infty(\R^n)}$ and are  explicitly given in \eqref{ctes}.

\item[$\circ$] {\rm $\LL^\infty$ stability for all $p>1$. }Let $u_0\in \LL^\infty(\R^n)$, then we have
\begin{equation}  \label{Thm.1.Cauchy.smooth2}
\|u(t)\|_{\LL^\infty(\R^n)}   \leq \|u(t_0)\|_{\LL^\infty(\R^n)}  \quad \text{ for all }  0\leq t_0 \leq t.
\end{equation}
\end{itemize}\vspace{-.1cm}
\rm \item \it {\rm Higher regularity in time. }If $u(t_0)\in \LL^\infty(\R^n)$, then   $u(x,\cdot) \in C_t^{[p], p-[p]}([t_0,\infty))$ for all $t_0>0$ and $x\in \R^n$.  More precisely, there exists a positive constant $\mathsf{c}_p$ that depends on  $p$, $q$, $n$, $\|J\|_{\LL^\infty(\R^n)}$ and on $\|u(t_0)\|_{\LL^\infty(\R^n)}$, such that   for almost all $x\in \R^n$ and all $t_0\leq t_1<t_2<\infty$ we have
\begin{equation} \label{thm.hi.t.intro}
\max_{k=0,\dots,[p]-1}\frac{|\partial_t^{k} u(x,t_1) - \partial_t^{k} u(x,t_2) |}{|t_1-t_2|}
+ \frac{|\partial_t^{[p]} u(x,t_1) - \partial_t^{[p]} u(x,t_2) |}{|t_1-t_2|^{p-[p]}} \leq \mathsf{c}_p.
\end{equation}
Moreover, when $2\leq p\in\N$, this estimate holds for any $k\in\N$.
\end{enumerate}
\end{thm}
\vspace{-.4cm}
\begin{proof} (a)  {\sl Boundedness of solutions.}
The $\LL^q$--$\LL^\infty$ smoothing effect for $p\in (2,\infty)$ and $u_0\in \LL^q(\R^n)$, with $q\in [1,\infty]$ as well as estimate \eqref{Thm.1.Cauchy.smooth1} is proved in Theorem \ref{main}. The $\LL^\infty$ stability estimate for $p\in (1,\infty)$ and $u_0\in \LL^\infty(\R^n)$ stated in \eqref{Thm.1.Cauchy.smooth2} is a consequence of Proposition \ref{norm.decreasing.difference}.

\noindent(b) {\sl Higher regularity in time.} It is proved in Theorems \ref{thm.time.reg.p.leq.2} and \ref{prop.hi.ut}.
\end{proof}
As far as the conservation of spatial regularity is concerned, this is our result that we conjecture to be sharp also when $p$ is not integer.
\begin{thm}[Cauchy problem: higher space-time regularity] \label{thm.c.conservation}
Let $p>1$, $n\ge 1$, and let  $J\colon\R^n\to\R$ be a  nonnegative function with compact support, $J(0)>0$ and normalized such that $\|J\|_{\LL^1(\R^n)}=1$. Let  $u_0 \in \LL^q(\R^n)$ for some $q\ge 1$ and $u $ be the corresponding solutions of \eqref{eqC}.

Assume   moreover  that:
\begin{enumerate}[leftmargin=15pt, label=(\roman*)]\itemsep2pt \parskip3pt \parsep2pt
\rm \item \it
$D^\alpha u_0\in C^{0,\omega}(\R^n)\cap \LL^\infty(\R^n)$ for any $1\le |\alpha|\leq [p]-1$,
and define $\mathsf{m}_p$ as
\begin{equation} \label{cond.i.intro}
\mathsf{m}_p:=\sum_{0\leq |\alpha|\leq [p]-1}\|D^\alpha  u_0\|_{\LL^\infty(\R^n)},
\end{equation}

\rm \item \it
there exists a modulus of continuity $\omega_{J,p}$ such that  for \text{a.e.} $ x,y \in \R^n$
\begin{equation} \label{cond.ii.intro}
 \sum_{0\leq |\alpha|\leq [p]-1}\int_{\R^n} \left|D^\alpha J(y-z)-D^\alpha J(x-z) \right|\,dz \leq \omega_{J,p}(|x-y|),
\end{equation}
\end{enumerate}
and define the modulus of continuity  $\bar \omega(\rho):=\max\{ \rho, \rho^{p-2},  \omega_{J,p}(\rho), \omega(\rho)\}$.

Then, for all $p>1$ we have that $u\in C_x^{[p]-1,\bar \omega}(\R^n)\cap C_t^{[p],p-[p]}([t_0,\infty))$ for all $t_0>0$. More precisely, there exists   $\mathsf{c}_p>0$ such that for all $x_1,x_2\in\R^n$ and $t_0\leq t_1<t_2<\infty$ we have
    \begin{equation}\label{est.space.time.k}
    \max_{\stackrel{|\beta|\leq [p]-1}{j=0,\ldots, [p]-1}} \frac{|D^\beta_x \partial_t^j u(x_1,t_1) - D^\beta_x \partial_t^j u(x_2,t_2)|}{\bar\omega(|x_1-x_2|)+ |t_1-t_2|}+ \max_{|\beta|\leq [p]-1}\frac{|D_x^\beta \partial_t^{[p]}u(x_1,t_1)-D_x^\beta \partial_t^{[p]}u(x_2,t_2) |}{\bar\omega(|x_1-x_2|) + |t_1-t_2|^{p-[p]}} \le \mathsf{c}_p\,.
    \end{equation}
    Note that $\mathsf{c}_p$ only depends on $p$, $q$, $n$,  $\|J\|_{\LL^\infty(\R^n)}$, $\mathsf{m}_p$ and $t_1,t_2$.
\end{thm}\vspace{-.1cm}
Estimates \eqref{est.space.time.k} of Theorem \ref{thm.c.conservation} are proved by combining two kinds of bounds, which of course are slightly stronger and may have their own independent interest.  First we obtain that time derivatives up to order $[p]-1$ keep the initial regularity in space, and then we prove that space derivatives up to order $[p]-1$ are continuous in time.   More precisely, with the notation of Theorem \ref{thm.c.conservation} we get:
\\
$(i)~$  $\partial_t^j u(\cdot,t) \in C_x^{[p]-1,\bar \omega}(\R^n)$ for any $j\in \N_0$ such that $j\leq [p]-1$. When $2\le p\in \N$,   this holds for any $j\in \N_0$. More precisely, there exists $\mathsf{c}_p(t)>0$ 
such that
$$
  \frac{|D^\beta \partial_t^{j} u(x_1,t)- D^\beta \partial_t^{j} u(x_2,t)| }{\bar\omega(|x_1-x_2|)}
\leq \mathsf{c}_p(t)
$$
holds for all $x_1,x_2\in \R^n$ and all $t\ge t_0>0$. More details are given in Theorem \ref{teo.holder.cont.1}.
\\
$(ii)~$  $D^\beta u(x,\cdot) \in C_t^{[p],p-[p]}([t_0,\infty))$ holds for any multi-index $\beta$ such that $|\beta|=k\leq[p]-1$ and any $t_0>0$. When $2\le p\in \N$,  this holds for any $k\in \N$. More precisely, there exists $\mathsf{c}_p>0$ 
such that
\begin{align*}
\max_{j=0,\ldots,[p]-1}  \frac{|D^\beta \partial_t^{j} u(x,t_1)- D^\beta \partial_t^{j} u(x,t_2)| }{|t_1-t_2|} +  \frac{|D^\beta \partial_t^{[p]} u(x,t_1)- D^\beta \partial_t^{[p]} u(x,t_2)| }{|t_1-t_2|^{p-[p]}}
\leq \mathsf{c}_p\,.
\end{align*}
for all $x\in \R^n$ and $t_0\leq t_1<t_2<\infty$. More details are given in Theorem \ref{teo.holder.cont.2}.

\medskip

When $p\ge 2$ is integer, we have a sharp regularity result, since it allows to reach $C_{x,t}^{\infty}$.
\begin{thm}[Sharp Regularity for integer $p$] \label{Coro1.c.intro}
Let $2\le p\in \N$, $n\ge 1$, and let  $J\in C^\infty_c(\R^n)$ with $J(0)>0$ and  $\|J\|_{\LL^1(\R^n)}=1$. Let  $u_0 \in \LL^q(\R^n)$ for some $q\ge 1$ and $u $ be the corresponding  solutions of \eqref{eqC}, that we know to be bounded for all $t>0$.\vspace{-.1cm}
\begin{enumerate}[leftmargin=*, label=(\alph*)]\itemsep1pt \parskip1pt \parsep0pt
\item Then for all $t_0>0$ we have that  $u(x,\cdot)\in C^\infty_t ([t_0,\infty))$ for all $x\in \R^n$.
In particular, $u(x,\cdot)$ is analytic in time with radius of analyticity $r_0= 1\wedge t_0$, and there exists  $\mathsf{c}_p>0$ independent of $k$  such that for all $k\in\N$, for almost all $x\in \R^n$ and all $t_0\leq t_1< t_2 <\infty$ we have that
$$
\frac{|\partial_t^{k} u(x,t_1) - \partial_t^{k} u(x,t_2) |}{|t_1-t_2|}\,
\leq \mathsf{c}_p\, k!\,.
$$
\vspace{-.3cm}
\item If moreover $u_0 \in  C^{\infty}(\R^n)$, then we have that $u\in C^{\infty}_{x,t}(\R^n\times [t_0,\infty))$ for all $t_0>0$. In particular, $D^\beta u(x,\cdot)$ is analytic in time with radius of analyticity $r_0= 1\wedge t_0$, and there exists
  $c_p>0$ independent of k such that for any $\beta\in \N_0^n$ and $k\in \N$, for almost all $x_1,x_2\in\R^n$ and all $t_0\leq t_1<t_2$ we have that
$$
\frac{|D^\beta_x \partial_t^k u(x_1,t_1) - D^\beta_x \partial_t^k u(x_2,t_2)|}{\bar\omega(|x_1-x_2|)+ |t_1-t_2|} \leq \mathsf{c}_p \, k!\,.
$$
\vspace{-.3cm}
\end{enumerate}
\end{thm}\vspace{-.2cm}
The proofs of Theorems \ref{thm.c.conservation} and \ref{Coro1.c.intro} (Higher space-time regularity) follows from Theorem \ref{teo.holder.cont.final}. 

\medskip

When $p>1$ is not integer, we have the following result that we conjecture to be sharp.

\begin{cor}[``Almost sharp" regularity when $1<p\not\in \N$] \label{Coro2.c.intro}
Let $p>1$, $n\ge 1$, and let  $J\in C^{[p]-1}(\R^n)$ with $J(0)>0$ and  $\|J\|_{\LL^1(\R^n)}=1$. Let  $u_0 \in \LL^q(\R^n)$ for some $q\ge 1$  and $u$ be the corresponding  solution of \eqref{eqC}, that we know to be bounded for all $t>0$. Let $\bar \omega(\rho)=\max\{\rho,\rho^{p-2}\}$.
\begin{enumerate}[leftmargin=*, label=(\alph*)]\itemsep2pt \parskip3pt \parsep0pt
\item Then for all $t_0>0$ we have that  $u(x,\cdot)\in C_t^{[p],p-[p]}([t_0,\infty))$ for all $x\in \R^n$. In particular, estimates \eqref{thm.hi.t.intro} hold for all $k\le [p]-1$.

\item If moreover $u_0\in C^{[p]-1,1}(\R^n)$, then we have that $u(\cdot, t)\in C^{[p]-1, \bar \omega}_{x}(\R^n)$ for all $t>0$. In particular, estimates \eqref{est.space.time.k} hold for all $j\le [p]-1$.
\end{enumerate}
\end{cor}
 
\begin{rem}[\textbf{On the locality of the regularity estimates in space}]\label{Rem.Locality} We shall emphasize that even if the equation is nonlocal, the nature of our regularity results is local, and this has some important consequences, that we shall emphasize here.
\begin{enumerate}[leftmargin=15pt, label=(\roman*)]\itemsep2pt \parskip2pt \parsep2pt
\item\textit{Preservation of the local modulus of continuity. }A careful inspection of the proofs shows that indeed we have a stronger result: we preserve the local modulus of continuity, essentially in two ways.
\begin{itemize}[leftmargin=*]\itemsep2pt \parskip3pt \parsep0pt
\item On the one hand, assuming full regularity of the kernel $J$, we preserve any modulus of continuity of the initial datum and its derivatives: \it  let $J\in C^{\infty}(\R^n)$, and $\beta\in \N$ and  $D^\beta u(0,\cdot)\in C^{\omega}(K)$ where $K\subset\subset \Omega$, then $D^\beta u(t,\cdot)\in C^{\omega}(K)$ for all $t>0$, that is, for all $x_1,x_2\in K$ and $t>0$:
\[\begin{split}
|D^\beta u(x_1,t)- D^\beta u(x_2,t)| &\leq A_1(t) |D^\beta u(x_1,0)- D^\beta u(x_2,0)| + A_2(t) |x_1-x_2|  \\
&\le \overline{C}(t)\max\left\{\omega (|x_1-x_2|)\,,\, |x_1-x_2|  \right\}\le  \overline{C}(t) \omega(|x_1-x_2|)\,.
\end{split}
\]\rm
\item On the other hand, if $J$ is not regular, then the continuity modulus is the worst between the one of the data and the one of $J$. For instance: \it let $J\in C^{\omega_J}(\R^n)$, and $u(0,\cdot)\in C^{\omega}(K)$ where $K\subset\subset \Omega$, then $u(t,\cdot)\in C^{\bar\omega}(K)$ for all $t>0$, where
\[
\bar\omega(|x_1-x_2|):=\max\left\{\omega(|x_1-x_2|)\,,\,\omega_J(|x_1-x_2|) \right\}
\]
that is, for all $x_1,x_2\in K$ and $t>0$:
\[\begin{split}
|D^\beta u(x_1,t)- D^\beta u(x_2,t)| &\leq A_1(t) |D^\beta u(x_1,0)- D^\beta u(x_2,0)| + A_2(t)\omega_J(|x_1-x_2|) \\
&\le \overline{C}(t) \bar\omega(|x_1-x_2|) \,.
\end{split}
\]\rm
\end{itemize}
\item\textit{Singularities do not move in time, nor can be created along the flow. }We are dealing with bounded solutions, that may start from possibly unbounded initial data $u_0\in \LL^q(\R^n)$ when $p>2$, or from bounded initial data when $p\in (1,2]$. In both cases, bounded solutions can only have  discontinuities with bounded oscillation. Let us define the ``singular set in the space variables'' $\mathcal{S}[f]$  of a $L^1_{\rm loc}$ function $f$, as the closure of the complementary of the set where $f$ is continuous, or, to be more precise, where $f$ admits an $L^1_{\rm loc}$ representative which is continuous). A consequence of the above estimates is that ``singularities do not move in time'', or to be more precise, the singular set of $u(t)$ will always be contained in the singular set of $u_0$: $\mathcal{S}[u(t)]\subseteq \mathcal{S}[u_0]$. This clearly follows by the fact that if $x_0\not \in \mathcal{S}[u_0]$ then there exists an open ball $B_r(x_0)$ where $u_0\in C^0(B_r(x_0))$ (indeed it will be in some $C^{\omega}(B_r(x_0))$ for some $\omega$). Then, the conservation of the local continuity modulus ensures that $u(t)\in C^0(B_r(x_0))$ for all $t>0$. This clearly implies that $\mathcal{S}[u(t)]\subseteq \mathcal{S}[u_0]$, as we claimed. Of course, this also implies that \textit{new discontinuities cannot be created in future times}, or simply that
    \[
    \qquad\mbox{if $u_0 \in C^{\omega}(B_r(x_0))$, then $u(t) \in C^{\omega}(B_r(x_0))$ for all $t>0$\,.}
    \]
    A similar remark applies also to derivatives of arbitrary order.

\item\textit{Application to other problems, different boundary conditions. }The last consequence of the locality of the regularity estimates is that essentially the same proof allows to treat different problems: they can be set in open domains of $\R^n$ and they may have different boundary conditions. The above regularity estimate would apply as ``interior regularity estimates''. We shall see with full details what happens in two important cases, namely the Dirichlet and Neumann problems on bounded domains in what follows.
\end{enumerate}

\end{rem}

\medskip

\subsubsection{Main results for the Dirichlet and Neumann problems} \label{sec.nd.intro}

We state as follows our results concerning  solutions of the Dirichlet and Neumann problems \eqref{eqD} and \eqref{eq}.

While for the Dirichlet problem no extra condition  on $J$ is really needed, for the Neumann problem we shall introduce the following technical assumption:
\begin{equation} \label{HJ} \tag{$H_J$}
\exists \kappa_{J,\Omega}>0 \,\text{ s.t. }\, \inf_{x\in\Omega} \int_{\Omega} J(x-y)\,dy \geq \kappa_{J,\Omega}>0.
\end{equation}
We shall further comment about this condition below, see Remark \ref{Rem.Cone.Cond}. We only stress here that it is strictly needed to quantify precisely the regularity estimates in a constructive way.

\begin{thm}[\bf Existence and uniqueness for  Neumann and Dirichlet Problems]

Let $p>1$, $n\ge 1$, and let  $J\colon\R^n\to\R$ be a bounded nonnegative function with compact support, $J(0)>0$ and normalized such that $\|J\|_{\LL^1(\R^n)}=1$. Then, for all $p>1$ and all $u_0\in \LL^q(\Omega)$, the same existence and contractivity estimates \eqref{Thm.1.Cauchy.i} as in Theorem $\ref{thm.exist.intro}$ hold also for the Dirichlet problem. Being $\Omega$ bounded, we can include the case $q=\infty$ as well.
\\
If additionally $J$  fulfills condition \eqref{HJ}, the same conclusions hold for the Neumann problem \eqref{eq}.

\end{thm}

\begin{thm}[\bf Regularity for the  Dirichlet and Neumann  Problems] \label{Thm.1.Neumann}

Let $p>1$, $n\ge 1$, $\Omega\subset \R^n$ be an open and bounded set, and let  $J\colon\R^n\to\R$ be a bounded nonnegative function with compact support, $J(0)>0$ and normalized such that $\|J\|_{\LL^1(\R^n)}=1$.

\noindent$\bullet$~\textsf{Dirichlet Problem. }Let $u$  be the solution of \eqref{eqD} corresponding to the initial datum $u_0\in \LL^1(\Omega)$. Then the following holds:

\noindent{\rm Boundedness of solutions. }The solution $u(t)$ is bounded in one of the following cases: either $p>2$ and we have  $u_0\in \LL^q(\R^n)$, $q\in [1,\infty)$, and estimate \eqref{Thm.1.Cauchy.smooth1} holds, or for all $p>1$, we have  $u_0\in \LL^\infty(\R^n)$ and estimate \eqref{Thm.1.Cauchy.smooth2} holds.

\noindent{\rm Regularity of  bounded solutions. }Let $u$ be a solution of \eqref{eqD} such that $u(t_0)\in \LL^\infty(\Omega)$ for some $t_0\ge 0$. Then the following holds:

\begin{enumerate}[leftmargin=*, label=(\alph*)]\itemsep2pt \parskip3pt \parsep0pt

\item {\rm Higher regularity in time: }we have that $u(x,\cdot) \in C_t^{[p], p-[p]}([t_0,\infty))$ for all $t_0>0$ and $x\in \Omega$ and estimate \eqref{thm.hi.t.intro} holds. 

\item {\rm Higher space-time regularity: }if we further assume conditions \eqref{cond.i.intro} and \eqref{cond.ii.intro}, then
$$
u\in C_x^{[p]-1,\bar \omega}(\R^n)\cap C_t^{[p],p-[p]}([t_0,\infty)) \text{ for all }  t_0>0
$$
where $\bar \omega(\rho):=\max\{ \rho, \rho^{p-2},  \omega_{J,p}(\rho), \omega(\rho)\}$, and estimate \eqref{est.space.time.k} holds.

\item {\rm Sharp regularity when $2\leq p\in \N$: }if $J$ is moreover smooth, then for all $t_0>0$ and $x\in \Omega$ we have that $ u(x,\cdot)\in C_t^\infty([t_0,\infty))$ and estimates \eqref{thm.hi.t.intro} hold for all $k\in \N$.\\
If moreover $u_0 \in  C^{\infty}(\Omega)$, then we have that $u\in C^{\infty}_{x,t}(\Omega\times [t_0,\infty))$ for all $t_0>0$. In particular, estimates \eqref{est.space.time.k} hold for all $k\in \N$.
\end{enumerate}

\noindent$\bullet$~\textsf{Neumann Problem. }If $J$ additionally satisfies condition \eqref{HJ}, then the results above hold  for solutions $u$ of the Neumann problem \eqref{eq}.
\end{thm}

When $p>2$ we can also obtain the decay in time of the solutions in $\LL^\infty$, obtaining an improvement on the existing results, see for instance \cite{AMRT, R20}.

\begin{thm}[\bf Asymptotic Behavior of solutions] \label{asint.intro}
Let  $p> 2$,   $n\ge 1$, $\Omega\subset \R^n$ be an open and bounded set, and let  $J\colon\R^n\to\R$ be a bounded nonnegative function with compact support, $J(0)>0$ and normalized such that $\|J\|_{\LL^1(\R^n)}=1$.

\noindent$\bullet$~\textsf{Dirichlet Problem. }Let $u$  be the solution of \eqref{eqD} corresponding to the initial datum $u_0\in \LL^1(\Omega)$. Then we have that
$$
\|u(t)\|_{\LL^\infty(\Omega)} \leq \mathsf{c} t^{-\frac1p} \quad   \forall t\gg 1,
$$
where $\mathsf{c}$ is a positive constant   depending on $J$, $n$, $p$,  $|\Omega|$ and $\|u_0\|_{\LL^1(\Omega)}$.

\noindent$\bullet$~\textsf{Neumann Problem. }If $J$ additionally satisfies  \eqref{HJ} and $u$ is the  solution of \eqref{eq}   corresponding to  $u_0\in \LL^1(\Omega)$, then it holds that
$$
\|u(t)-\overline u_0\|_{\LL^\infty(\Omega)} \leq \mathsf{c} t^{-\frac1p} \quad   \forall t\gg 1,
$$
where $\overline u_0=\frac{1}{|\Omega|}\int_\Omega u_0(x)\,dx$  and $\mathsf{c}$ is a positive constant depending on $J$, $n$, $p$, $\kappa_{J,\Omega}$, $|\Omega|$ and $\|u_0\|_{\LL^1(\Omega)}$.

\end{thm}

\begin{rem}\label{Rem.Cone.Cond}\rm
Some comments on condition \eqref{HJ} are in order.
\begin{enumerate}[leftmargin=15pt, label=(\roman*)]\itemsep2pt \parskip3pt \parsep2pt
\rm \item
Hypothesis \eqref{HJ} is always satisfied when the set $\Omega$ is sufficiently regular. Indeed, when $\Omega$ satisfies the so-called \emph{weak cone condition} (see Section \ref{sect.weak.cone}), i.e., for every $x\in \overline{\Omega}$, the cone $\Gamma(x)$ satisfies for some $\delta$
$$
|\Gamma(x)| = |\{y\in R(x)\colon |y-x|<1\}| \geq \delta>0,
$$
being $R(x)$ the union of line segments emanating from $x$ contained in $\Omega$. In this case, given $x\in \partial\Omega$, there exists $\eta>0$ such that the cone $\Gamma(x)$ is contained in the ball $B_\eta(x)\subset \supp J$ with center in $x$ and radius $\eta$. This yields
\[
\int_{\Omega} J(x-y)\,dy \geq \int_{\Omega\cap B_\eta (x)} J(x-y)\,dy  \geq \varepsilon_\eta \int_{\Gamma(x)}\,dy = \varepsilon_\eta |\Gamma(x)|\geq  \varepsilon_\eta \delta>0
\]
where we have used that $\inf \{ J(z) \colon z\in B_\eta(x) \} \geq \varepsilon_\eta$ for some $\varepsilon_\eta>0$, which is always true since $J$ is continuous, nonnegative and $B_\eta(x)\subset \supp J$.

\rm \item
The prototypical example of nonnegative  radial kernel $J$ with compact support is given by
$$
J(z)=\mathsf{c}_{n,a,R}^{-1}(R-|z|)^a, \qquad \text{ where } \mathsf{c}_{n,a,R}=n\omega_n R^{n+a} B(n,a+1)
$$
with $a>0$, $\supp J = B_R(0)$ and where $B$ denotes the  Beta function. In this case the constant $\kappa_{J,\Omega}$ can be computed explicitly. Indeed, given $x\in \Omega$
\begin{align*}
\int_{B_R(x)} J(x-y)\,dy &= \int_{B_R(0)} (1-|z|)^{a}\,dz = n\omega_n \int_0^R (1-r)^a r^{n-1}\,dr\\
&= n\omega_n R^{n+a}\int_0^1 (1-r)^a r^{n-1}\,dr
= \mathsf{c}_{n,a,R},
\end{align*}
being $\omega_n$ the measure of the unit ball in $\R^n$.
\end{enumerate}
\end{rem}

\noindent\textbf{Proof of the results for Dirichlet and Neumann problems. }The existence, uniqueness, and regularity results have been proven with slight modifications, similar to those used in the Cauchy problem. The proof of the smoothing effect, however, is separated and can be found in Theorem \ref{main2}. The remaining proofs are presented in the same theorems as in the Cauchy case.

\noindent The asymptotic results of Theorem \ref{asint.intro} are proved in Theorems \ref{decay.D} and \ref{decay.N}.\hfill \qed

\subsection{Related results, novelties and main ideas of the proofs}\label{Sec.1.4} In the local case there exist a huge literature about the celebrated p-Laplacian evolution equations, whose prototype is $u_t=\Delta_p(u)=\nabla\cdot(|\nabla u|^{p-2}\nabla u)$. Without any aim of completeness we quote here some related results: Local smoothing effect were known since the pioneering work of DiBenedetto \cite{D93}, with DeGiorgi method, Smoothing effects via nonlinear adaptation of Gross' method were proven in the 2000's,
see \cite{CG02,BG06}, and via Moser iteration in \cite{BIV10}. There has been an intense work, and the state of the art of higher integrability estimates can be found in the more recent contributions \cite{BDL21,BDKS20,BMS18, IJS17, KM11, KM12}, even for a more general case of nonlinear evolution equations. We refer to \cite{DGV,D93,BSS22,BIV10} for a complete account of local and global Harnack inequalities. Higher regularity estimates, the maximum possible being $C^{1,\alpha}$ regularity in space and $C^{1,\beta}$ in time, were obtained for bounded solutions by DiBenedetto and Friedman, see the monograph \cite{D93}.

Our first objective is to establish the boundedness of solutions ($\LL^q-\LL^\infty$-smoothing effects), since this kind of estimate is the first step towards further regularity properties. The Moser iteration approach \cite{M64, M67} is a standard method to obtain such results, but the quadratic form associated to the operator has to satisfy some Gagliardo-Nirenberg-Sobolev and Stroock-Varopoulos type inequalities. This has been used in the linear case, for a class of nonlocal operators including the standard fractional Laplacian $(-\Delta)^s$ with $s\in (0,1)$ in \cite{ K1,K3,K4,K2,S06}.

An alternative to the Moser iteration is the Green function method introduced by Vazquez and the first author in \cite{BV15} for the Dirichlet problem for Fractional Porous Medium Type Equation (FPME) $u_t= -\mathcal{L} u^m$, where $\mathcal{L}$ is a linear nonlocal operator, typically a fracional Laplacian with Dirichlet boundary conditions (there are three possible different choices!). This method has been exploited in several directions in \cite{BFV18,BFR17,BV15a, BE} for more general nonlinear nonlocal (degenerate and singular) diffusions in bounded domains and on the entire $\R^n$, and even on manifolds \cite{BBGG,BBGM}.  The key point of this method is to have good estimates of the kernel of $(-\mathcal{L})^{-1}$, that is, the Green function of $\mathbb{G}_{-\mathcal{L}}$.  This approach allows to obtain $\LL^1$--$\LL^\infty$ smoothing results for a quite wide class of linear operators, including Levy operators and much more, see \cite{BE} for a complete account on the equivalences between: existence of green functions and heat kernels for $\mathcal{L}$, the validity of suitable GNS inequalities and smoothing effects for nonlocal PME-type equations. See also \cite{BII} where a comparison between the Moser iteration and the Green function method is made in the case of Fast Diffusion equations on bounded domains. Unfortunately, the Green function method seems not to be compatible with the structure of p-Laplacian operators: In the problem under consideration, the nonlinear nature of the operator makes the Green function method unsuitable for our purposes. On the other hand, the DeGiorgi-Nash-Moser iteration is flexible enough to be adapted to the fractional $p-$Laplacian case $(-\Delta_p)^s$, with $s\in (0,1)$, see for instance \cite{CH21, GT18, V20, V21}. For higher regularity estimates one can rely again on classical DeGiorgi iterations, see \cite{D93}, or also can use nonlinear potential estimates, see \cite{DM10, KM11,KM12, KM13, KM13-2, KM14}.

Our problem presents extra difficulties, since the diffusion operator is both nonlinear and of order zero: first we observe that Gagliardo-Nirenberg-Sobolev type inequalities are not available in this setting, preventing us from implementing a Moser or DeGiorgi method. Also, we observe that it is not possible to have useful GNS for the nonlinear case. Indeed, if we had some GNS, say for $p>2$, these would imply GNS for the quadratic form of the linear operator ($p=2$), which are known to be equivalent to $\LL^q-\LL^\infty$-smoothing effects, which we know to be not true for the linear case (as we have seen above), see \cite{BE}. This may suggest that $\LL^q-\LL^\infty$-smoothing effects are simply not true in the nonlinear case. We surprisingly show the contrary.

The only result existing in literature which is comparable to our new smoothing effect, to the best of our knowledge, has been obtained for nonlocal porous medium type equations of ``zero order'': unexpected smoothing effects of the form similar to \eqref{Thm.1.Cauchy.smooth1} hold for merely integrable data, see \cite[Theorem 3.5]{BE}. The proof exploits the dual equation (that in the present case we do not have), and the strict convexity of the nonlinearity. As already explained, the linear case ($p=2$ here) does not satisfy smoothing effects. In our case, when $p>2$, we can show that merely integrable initial data produce bounded solutions, but we go way further: bounded solutions turn out to be smooth in time and even smooth (classical) in space when the kernel and the data allows it. We also perform a delicate analysis, where we compare the loss of regularity due to the ``low regularity'' of the kernel, versus the high regularity of the data, which is always preserved (at least) up to order $p$.

We propose here a new approach to regularity for nonlinear zero order operators, that exploits \textit{the strict convexity and homogeneity of the nonlinearity }(that we have when $p>2$) together with \emph{elementary numerical inequalities}. A key tool in our arguments, is the so-called \emph{B\'enilan-Crandall estimate} (time-monotonicity) that holds for nonnegative solutions:
$$
u_t(x,t)\geq -\frac{u(x,t)}{(p-2)t} \quad \text{ a.e. } x \text{ and } t>0.
$$
This follows by comparison and time scaling. With this, we prove that solutions of the Cauchy, Neumann and Dirichlet problems corresponding to $u_0\in\LL^q$, $1\leq q <\infty$ are indeed bounded, and satisfy precise $\LL^q$--$\LL^\infty$ smoothing estimates.

We shall give a flavour of the main ideas in our proofs in the simplest possible scenarios.

\subsubsection{Smoothing effects when $p>2$} In this case, we want to sketch the proof of the smoothing effect \eqref{Thm.1.Cauchy.smooth1} of Theorem \ref{thm.C.reg.intro}, that reads: \it Let $u(x,t)$ be a positive solution of \eqref{eqC} corresponding to the positive initial datum $u_0\in \LL^1(\R^n)$, then we have that
\begin{equation}\label{Intro.C.se}
\|u( t)\|_{\LL^\infty(\R^n)} \leq
\frac{\mathsf{K}_1}{t^{\frac{1}{p-2}}}   + \mathsf{K}_2\|u( t_0)\|_{\LL^q(\R^n)}  \quad \text{ for all } 0\le t_0\leq  t\,,
\end{equation}
for some constants $\mathsf{K}_1$ and $\mathsf{K}_2$. \rm Let us fix a time $t\in (0,T]$. The following inequality (Lemma \ref{pointw.ineq})
$$
a^{p-1}-|a-b|^{p-2}(a-b)\leq (p-1)\max\{a^{p-2},b^{p-2}\}b, \qquad a,b>0,
$$
applied to $a=u(x,t)$, $b=u(y,t)$  together to the fact that $\|J\|_{\LL^1(\R^n)}=1$ yields
$$
u(x,t)^{p-1} +\J_p u(x,t) \leq (p-1) \int_{\R^n} J(x-y)(u(x,t)^{p-1}+u(y,t)^{p-1}) u(y,t)\,dy:=\mathcal{I}_{p,J}u(x,t).
$$
Using the B\'enilan-Crandall estimate and the fact that $u$ is a solution, we obtain the inequality
$$
- \frac{u(x,t)}{(p-2)t} \leq  u_t(x,t)  = \J_p u(x,t) .
$$
Then, the last two relations lead to the following:
$$
u(x,t)^{p-1} \leq \frac{u(x,t)}{(p-2)t} + \mathcal{I}_{p,J}u(x,t).
$$
Repeated application of Young's inequality gives that for any $\varepsilon>0$ there exists $\mathsf{c}_{p,J,\varepsilon}>0$ such that
$$
u(x,t)^{p-1} \leq \varepsilon u^{p-1}(x,t) + \mathsf{c}_{p,J, \varepsilon} \left( t^{-\frac{p-1}{p-2}}   +    \|u(t)\|_{\LL^{p-1}(\R^n)}^{p-1} \right),
$$
Choosing $\varepsilon$ sufficiently small gives a $\LL^{p-1}$-$\LL^\infty$ smoothing effect: we find an explicit $\mathsf{\bar c}_{J,p,\varepsilon}>0$ such that
\begin{equation} \label{eqq1.intro}
\|u(t)\|_{\LL^\infty(\R^n)}^{p-1} \leq   \mathsf{\bar c}_{p,J,\varepsilon}  \left( t^{-\frac{p-1}{p-2}}   +     \|u(t)\|_{\LL^{p-1}(\R^n)}^{p-1} \right).
\end{equation}
This is the basic smoothing, which is a self-improving inequality: using again Young's inequality, we obtain that
$$
\|u(t)\|_{\LL^{p-1}(\R^n)}^{p-1} \leq \varepsilon_1 \|u(t)\|_{\LL^\infty(\R^n)}^{p-1} +  \mathsf{\tilde c}_{\varepsilon_1,J, p} \|u(t)\|_{\LL^1(\R^n)}^{p-1},
$$
where $\mathsf{\tilde c}_{\varepsilon_1,J,p}>0$ and $\varepsilon_1>0$. Choosing $\varepsilon_1$ small and combining this inequality with \eqref{eqq1.intro} yields the $\LL^1$--$\LL^\infty$ smoothing \eqref{Intro.C.se}. \qed

\noindent\textbf{Remark. }As already observed before, here \textit{the smoothing surprisingly depends only on the strict convexity of the nonlinearity}, indeed, when $p=2$ this result is false, see \eqref{lin.sol.form}.

\subsubsection{Regularity of solutions}\label{ssec.Reg.LIN} In this case, we want to sketch the proof of the regularity results in the simplest possible case, which is the Cauchy problem for the linear case $p=2$. This will provide the basic  ideas of the proof of the nonlinear case $p>1$, which of course is technically much more involved. Once we have identified a class of bounded solutions, we would like to show that they possess indeed higher regularity in time and, that  they preserve  the initial modulus of continuity. 

Let $u$ be a solution of the Cauchy problem, with $u_0\in \LL^\infty(\R^n)$.

\noindent$\circ~$ \emph{Higher regularity in time}.  From the equation it is immediate the following estimate,
$$
|u_t(x,t)|\leq \int_{\R^n} J(x-y)|u(y,t)-u(x,t)|\,dy  \leq 2\|u_0\|_{\LL^\infty(\R^n)}
$$
which gives the boundedness of $u_t$ as follows
\begin{equation} \label{intro.cota.1}
\|u_t(t)\|_{\LL^\infty(\R^n)}\leq 2\|u_0\|_{\LL^\infty(\R^n)}.
\end{equation}
We can bootstrap the argument: given positive numbers $t_1<t_2$, using \eqref{intro.cota.1} we get
\begin{align*}
|u_t(x,t_2)-u_t(x,t_1)|&\leq
\int_{\R^n} J(x-z)|(u(z,t_2)-u(z,t_1))-(u(x,t_2)-u(x,t_1))|\,dz\\
&\leq
2|t_1-t_2| \|u_t(t)\|_{\LL^\infty(\R^n)} \int_{\R^n} J(x-z) \,dz
\leq  4|t_1-t_2| \|u_0\|_{\LL^\infty(\R^n)}\,.
\end{align*}
Hence $u_t(x,\cdot)\in C_t^{0,1}([0,\infty))$. Moreover, using mean value theorem and \eqref{intro.cota.1} we get
\begin{align} \label{intro.cota.2}
\begin{split}
u_{tt}(x,t)&=\lim_{h\to0^+} \frac{u_t(x,t+h)-u_t(x,t)}{h}\\
&=\lim_{h\to0^+} \frac{1}{h} \int_{\R^n}  J(x-y)\left\{ (u(y,t+h)-u(y,t))-(u(x,t+h)-u(x,t)) \right\}\,dy\\
&=\int_{\R^n}  J(x-y) \left\{ u_t(y,t_1^*)-u_t(x,t_2^*)) \right\}\,dy\\
&\leq
\int_{\R^n}  J(x-y)\left\{\|u_t(t_1^*)\|_{\LL^\infty(\R^n)}+ \|u_t(t_2^*)\|_{\LL^\infty(\R^n)}   \right\}   \,dy
\leq 4 \|u_0\|_{\LL^\infty(\R^n)}
\end{split}
\end{align}
where $t_i^*\in (t_i,t_i+h)$, $i=1,2$,  and therefore
\begin{equation} \label{intro.cota.3}
\|u_{tt}(t)\|_{\LL^\infty(\R^n)}\leq 4\|u_0\|_{\LL^\infty(\R^n)}.
\end{equation}
We can bootstrap a second time: using \eqref{intro.cota.2}, \eqref{intro.cota.3} and the mean value theorem we can write
\begin{align*}
u_{tt}(x,t_1)-u_{tt}(x,t_2)&= \int_{\R^n} J(x-y) \left\{ (u_t(y,t_{1,1}^*)-u_t(y,t_{2,1}^*)) -(u_t(x,t_{1,2}^*)-u_t(x,t_{2,2}^*)) \right\}\,dy\\
&=
\int_{\R^n} J(x-y)\{  u_{tt}(y,\tau_1) |t_{1,1}^* - t_{2,1}^*| - u_{tt}(y,\tau_2) |t_{1,2}^* - t_{2,2}^*| \} \,dy\\
&\leq 8 |t_1-t_2|  \|u_0\|_{\LL^\infty(\R^n)}
\end{align*}
where $\tau_i\in (t^*_{i,1},t^*_{i,2})$, $i=1,2$,  which gives that $u_{tt}\in C^{0,1}_t([0,\infty))$.

\noindent We can repeat the argument inductively,  to get that for any $k\in \N$ it holds
$$
\|\partial_t^k u(t)\|_{\LL^\infty(\R^n)} \leq 2^k \|u_0\|_{\LL^\infty(\R^n)} \qquad \text{ and } \qquad \partial_t^k u(x,\cdot)\in  C^{0,1}_t([0,\infty)).
$$
As a consequence, a bounded solution $u$ of the Cauchy problem in the linear case possesses arbitrary continuous time derivatives, that is,
$
u(x,\cdot)\in C^\infty([0,\infty))
$
for almost all $x\in \Omega$\,.  Here it clearly appears that solutions are analytic in time, since the above estimates clearly implies that
\[
\|\partial_t^k u(t)\|_{\LL^\infty(\R^n)}\le M_0\, k!\,,  \qquad\mbox{with }M_0:= 2 \|u_0\|_{\LL^\infty(\R^n)}\,,
\]
and the radius of convergence $r_0$ of the Taylor series is uniform, indeed, $r_0= 1\wedge t$.
 
\noindent$\circ~$ \emph{Conservation of the initial continuity}. To simplify the exposition, let us assume that the kernel $J$ has a H\"older modulus of continuity: there exists $\alpha_J \in (0,1]$ such that for all $x_1,x_2\in \R^n$ we have
$$
\int_{\R^n} |J(x_1-y)-J(x_2-y)|\,dy \leq |x_1-x_2|^{\alpha_J}.
$$
Under these assumptions, we can ensure that the modulus of continuity of the initial datum is ``preserved", meaning that is does not change its class. Indeed, suppose that $u_0\in C^{0,\alpha}(\R^n)$ for some $\alpha\in (0,\alpha_J]$. For a solution $u$, expression \eqref{lin.sol.form} holds and the $\LL^\infty$-norm is decreasing in time, so that
\begin{align*}
u(x_1,t)-u(x_2,t)&=e^{-t}(u_0(x_1)-u_0(x_2)) + \int_{0}^{t}e^{-\tau} \int_{\R^n}  u(y,\tau)(J(x_1-y)-J(x_2-y))  dy\,d\tau\\
&\leq
e^{-t}|x_1-x_2|^\alpha + \|u_0\|_{\LL^\infty(\R^n)} |x_1-x_2|^{\alpha_J} (1-e^{-t}).
\end{align*}
As a consequence, since $\alpha\le \alpha_J\le 1$, we have that $u(\cdot,t)\in C^{0,\alpha}(\R^n)$ (the same as $u_0$) and we have
$$
\frac{|u(x_1,t)-u(x_2,t)|}{|x_1-x_2|^\alpha} \leq  e^{-t} +
(1-e^{-t}) \|u_0\|_{\LL^\infty(\R^n)}.
$$
Indeed we show a finer result: the same proof allows to show that for all $t>0$  \it the solution $u(\cdot, t)$ preserves the worst modulus of continuity between the initial data and the kernel $J$, \rm indeed, if $\alpha_J< \alpha$ the above estimate would hold with $\alpha_J$.

\noindent\textbf{Remark. }Also, the local nature of this estimates allows to deduce easily that: initial discontinuity points do not move in time, and no new discontinuities are created.

\subsection{Organization of the paper. }

The paper is structured as follows. Section \ref{section2} introduces the notation we will use along the paper.  Section \ref{section3} is dedicated to extend the existence and uniqueness of solutions to equations more general initial data; we also prove some qualitative properties of solutions, including a comparison principle, contractivity and time monotonicity properties. Section \ref{section4} contains the proof of our smoothing effect results. In Section \ref{section6}, we establish a time regularity result as well as a result regarding the preservation of the initial regularity. Finally, in Section \ref{section7}, we establish some consequences of our main results. We collect in the appendix some basic facts about gradient flows in Hilbert spaces. 

\section{Notations} \label{section2}
The following notations will be used throughout the paper.
\begin{itemize}[leftmargin=*]\itemsep2pt \parskip2pt \parsep0pt

\item[] For all $p>1$ and $\tau\in \R$ we define the functions
$$
L_p(\tau)=|\tau|^{p-2}\tau, \qquad M_p(\tau)=|\tau|^{p-2}.
$$

\item[]  For any $u,v\in \LL^p(\Omega)$ we consider the \emph{energy functional} $\E$ as
$$
\E(u,v):=\frac12\int_\Omega\int_\Omega J(x-y)|u(x)-u(y)|^{p-2}(u(x)-u(y))(v(x)-v(y))\,dxdy.
$$
When $\Omega=\R^n$ we just write $\mathcal{E}_p$ instead of $\E$.

\item[] $[x]$ denotes the integer part of $x\in \R$.

\item[] $B_R(x)$ denotes the ball in $\R^n$ with radius $R>0$ and center $x\in \R^n$. The unit ball unit ball $B_1(0)$ in $\R^n$ is denoted by $\mathcal{S}^{n-1}$ and its measure by  $\omega_n:=\frac{\pi^\frac{n}{2}}{\Gamma(\tfrac{n}{2}+1)}$, where $\Gamma(\cdot)$ stand for the Gamma function.

\item[]  $\LL^p(0,T;X)$ denotes the  set of measurable functions $u\colon [0,T]\to X$ such that:
if $1\leq p<\infty$
$$
\|u\|_{\LL^p(0,T;X)}:=\left(\int_0^T \|u(t)\|^p_X \,dt\right)^\frac1p <\infty;
$$
if $p=\infty$
$$
\|u\|_{\LL^\infty(0,T;X)}:=\text{ess} \sup_{0\leq t\leq T} \|u(t)\|_X<\infty.
$$

\item[] $W^{1,1}(0,T;X)$ denotes the set of functions $u\in \LL^1(0,T;X)$ such that $u'(t) \in \LL^1(0,T;X)$.

\item[] $AC(0,T;X)$ denotes the set of functions $u\colon [0,T] \to X$ such that are absolutely continuous on $[0,T]$.

\item[] $C^0(\Omega)$ denotes the set of continuous functions from $\Omega$ to $\R$.

\item[] $C^k(\Omega)$ denotes the set of  functions from $\Omega$ to $\R$ with $k\in \N$ continuous derivatives.

\item[] $C^{0,\omega}(\Omega)$ denotes the set of  functions with a modulus of continuity $\omega$ in $\Omega$ (see Section \ref{sec.holder}).

\item[] $C^{0,\alpha}(\Omega)$ denotes the spaces of H\"older $\alpha$ functions in $\Omega$ (see Section \ref{sec.holder}).

\item[] $C^{k,\alpha}(\Omega)$ denotes the spaces of functions having continuous derivatives up to order $k$ and such that the $k$-th partial derivatives are H\"older continuous with exponent $\alpha\in (0,1]$(see Section \ref{sec.holder}).

\item[] $C^p(\Omega):=C^{[p],p-[p]}(\Omega)$ for any $p\in \R_+$.

\item[] $\Lip((t_0,t_1))$ denotes the space of Lipschitz functions in $(t_0,t_1)$ (see Section \ref{sec.holder}).

\item[] $D^\beta$ denotes the  derivative with respect to the multi-index $\beta=(\beta_1,\ldots, \beta_n)$ (see Section \ref{sec.holder}).

\end{itemize}

\section{Properties of solutions} \label{section3}
We establish some useful yet basic properties of solutions in this section. We provide a general result about existence and uniqueness of solutions, which complements the previous ones \cite{AMRT}. Once existence in ensured, we prove time-monotonicity estimates for these solutions, often called Benilan-Crandall estimates  \cite{BC81}, that constitute a key ingredient in the proof of the smoothing effects. The time-monotonicity estimates are a consequence of the homogeneity of the nonlinearity, together with a weak comparison principle, that we also prove in this Section. Indeed, we prove T-contractive estimates (which imply weak comparison) in $\LL^1$ and also $\LL^q$-contractivity properties for all $q\in [1,\infty]$.

\begin{defn}[Definition of solutions to different problems]\label{Def.Soln}
A \emph{solution}  of \eqref{eq} in $[0,T]$ is a function $u\in W^{1,1}(0,T;\LL^q(\Omega))$ that satisfies
\begin{align*}
\begin{cases}
u_t(x,t)=\J_{\Omega,p} u(x,t) &\text{ a.e. } x\in\Omega \text{ and } t\in(0,T),\\
u(x,0)=u_0(x) &\text{ a.e. }x\in\Omega.
\end{cases}
\end{align*}
When $\Omega=\R^n$ this gives a  \emph{solution} of \eqref{eqC} in $[0,T]$.

Similarly, a \emph{solution} of  \eqref{eqD} in $[0,T]$ is defined as a function $u\in W^{1,1}(0,T;\LL^q(\Omega))$ that satisfies
\begin{align*}
\begin{cases}
u_t(x,t)=\J_p u(x,t), &\text{ a.e. }x\in \Omega \text{ and }  t>0,\\
u(x,t)=0, &\text{ a.e. }x\in \R^n\setminus  \Omega \text{ and }  t>0,\\
u(x,0)=u_0(x), &\text{ a.e. } x\in \Omega.
\end{cases}
\end{align*}
\emph{Supersolutions} and \emph{subsolutions} of \eqref{eq}, \eqref{eqC} and \eqref{eqD} are defined by replacing $=$ with $\geq$ or $\leq$, respectively.
\end{defn}

\begin{rem} The diffusion operator $\J_p$ is of order zero (no weak derivatives involved, not even fractional) hence the concept of solution is particularly simple: the equation is satisfied almost everywhere in space and time. In the theory of nonlinear parabolic PDEs, this usually corresponds to the class of the so-called \textit{strong solutions }(i.e. $u_t$ is a $\LL^q$ function, a regular distribution).

The difference between weak, mild, strong and classical solutions, usually a technical obstacle in nonlinear parabolic theories, in this case is somehow easier: for instance, \textit{classical solutions }are just solutions according to the above definition, which are $C^1$ in time and $C^0$ in space.

The drawback is that we do not have useful functional inequalities (of Sobolev type for instance) naturally associated to the operator, hence is it is not possible, to the best of our knowledge, to prove regularity results through a DeGiorgi-Nash-Moser approach.

Note that for a given subset $\Omega\subseteq \mathbb{R}^n$ and $1\leq q \leq \infty$, we have $W^{1,1}(0,T;\LL^1(\Omega))\subset AC([0,T];\LL^q(\Omega))$. Moreover, if $w\in W^{1,1}(0,T;\LL^q(\Omega))$, then $w$ is almost everywhere equal to a function that is absolutely continuous on $[0,T]$ with values in $\LL^q(\Omega)$. In particular, $u$ is defined almost everywhere for $x\in \Omega$ and $t\in [0,T]$.
\end{rem}

By studying the accretivity and range of the operators using nonlinear semigroup theory,  the following existence and uniqueness result is obtained in \cite{AMRT} when the initial data belong to $\LL^p$.

\begin{prop}[Theorems 6.2, 6.24, and 6.37 in \cite{AMRT}]
Suppose $p>1$ and  let $u_0\in \LL^p(\Omega)$. Then for any $T>0$ there exists a unique solution of \eqref{eqC}.

Suppose that $\Omega\subset \R^n$ is open and bounded and let $u_0\in \LL^p(\Omega)$. Then for any $T>0$ there exists a unique solution of \eqref{eqD} and \eqref{eq}.
\end{prop}

We extend the previous result to initial data in $\LL^q$ with $q\in [1,\infty)$. Our proof is based on the \emph{evolution variational inequality} formulation of gradient flows (see Appendix \ref{appendix.1} for further details) and the use of the  a priori $\LL^q$--$\LL^\infty$ smoothing effect obtained in the next section.

\begin{thm}[Existence and uniqueness] \label{existencia}
Suppose $p>2$ and  let  $u_0\in \LL^q(\R^n)$ with $q\in[1,\infty)$. Then for any $T>0$ there exists a unique solution of \eqref{eqC}. When $p \in (1,2]$ then we can take $u_0 \in \LL^q(\R^n)$ with $q\in \{2,p\}$.

Suppose that $\Omega\subset \R^n$ is an open and bounded set and $u_0\in \LL^q(\Omega)$ with $q\in [1,\infty]$. Then for any $T>0$ there exists a unique solution of \eqref{eqD}. If $\Omega$ additionally satisfies \eqref{HJ}, then there exists a unique solution of \eqref{eq}.  When $p\in (1,2]$ we can take $u_0\in \LL^p(\Omega)$.

\end{thm}

\begin{proof}
We will prove the result for equation \eqref{eqC}. Using similar arguments, we can also establish the same result for the Neumann and  Dirichlet problems. Note the in these problem we can approximate the initial data with simple functions when $u_0\in \LL^\infty(\Omega)$.

\noindent$\circ~$\textsc{Step 1}. Let us prove first that for every initial datum $u_0\in \LL^{2}(\R^n)$ and all $p>1$, there exists a unique gradient flow solution $u\in \LL^{2}(\R^n)$ of \eqref{eqC}.\\
Consider   the energy functional $\mathcal{I}_{p}\colon \LL^{2}(\R^n)\to [0,\infty]$  defined as
\begin{align*}
\mathcal{I}_{p}(u):=
\begin{cases}
\frac1p \mathcal{E}_p(u,u) &\quad \text{ if } u \in \LL^{2}(\R^n)\\
+\infty &\quad \text{ otherwise},
\end{cases}
\end{align*}
which  is   convex and lower semicontinuous, moreover, its gradient flow   coincides with the equation
\begin{equation*} 
u_t(x,t) = \J_p u(x,t), \qquad\mbox{for a.e. $x\in \R^n$, and all $t>0$.}
\end{equation*}
Let us prove that $u\in \LL^2(\R^n)$ is a gradient flow solution of \eqref{eqC} by using the equivalence with EVI solutions stated in Proposition \ref{equivalence}.  Indeed, suppose that $w\in \LL^2(\R^n)$. By applying Proposition \ref{parts} and Young's inequality, we obtain the following inequality for $t>0$:
 \begin{align*}
\int_{\R^n} &(u(x,t)-w(x)) \J_{p} u(x,t) \,dx = -\mathcal{E}_p(u(x,t),u(x,t)) + \mathcal{E}_p(u(x,t),w(x))\\
&\leq
-\mathcal{E}_p(u(x,t),u(x,t))\\
& \qquad + \frac12\iint_{\R^n\times\R^n} J(x-y)^\frac{p-1}{p} |u(x,t)-u(y,t)|^{p-1} J(x-y)^\frac{1}{p} |w(x)-w(y)|\,dxdy\\
&\leq
-\mathcal{E}_p(u(x,t),u(x,t)) + \frac{1}{p}\mathcal{E}_p(w(x),w(x)) + \frac{p-1}{p}\mathcal{E}_p(u(x,t),u(x,t))\\
&=\mathcal{I}_{p}(w)- \mathcal{I}_{p}(u),
\end{align*}
from which we deduce that
$$
\frac12 \frac{d}{dt} \|u( t)-w\|_{\LL^2(\R^n)}^2 = \int_{\R^n} (u(x,t)-w(x)) \J_{p}(u)\,dx \leq \mathcal{I}_{p}(w)- \mathcal{I}_{p}(u),
$$
and hence  the curve $u(\cdot,t)$ satisfies \eqref{es.evi}.

Therefore, due to   the Brezis-Komura Theorem \ref{BK}, the gradient flow solution $u(\cdot, t)$ is unique and  satisfies the contractivity property
$$
\| u(t) - v(t)\|_{\LL^2(\R^n)} \leq \|u_0-v_0\|_{\LL^2(\R^n)} \quad \text{for all }t>0,
$$
where $u$ and $v$ are gradient flow solutions of \eqref{eqC} corresponding to the initial data $u_0,v_0\in \LL^{2}(\R^n)$.

\noindent$\circ~$\textsc{Step 2}. We extend  the existence and uniqueness of solution for more general initial data.

Given $u_0\in \LL^q(\R^n)$ with  $q\in [1,\infty)$ (when $p>2$) or $u_0\in \LL^\infty(\R^n)$ in the case $p\in (1,2]$, let $\{\tilde u_{0,k}\}_{k\in\N}\subset  \LL^{2}(\R^n)$ be a sequence of simple functions such that $\tilde u_{0,k}\to u_0$ strongly in $\LL^q(\R^n)$.   Consider the monotone sequence $
\{u_{0,k}\}_{k\in\N}$ given by
$$
u_{0,k}:=\min\{k, \tilde u_{0,k}\} \chi_{B_k(0)}.
$$
Then, as $k\to \infty$ we have that $ u_{0,k}(x) \nearrow u_0(x)$ for \text{a.e.} $x\in\R^n$ and $t>0$ and  $u_{0,k} \to u_0$ strongly in $\LL^q(\R^n)$. Moreover, for each $k\in \N$
\begin{equation} \label{tilde.0}
\|u_{0,k}\|_{\LL^q(\R^n)}\leq  \|u_0\|_{\LL^q(\R^n)}.
\end{equation}
Due to Step 1, for each $k\in \N$ there exists a  gradient flow  solution $u_k( t)\in \LL^{2}(\R^n)$ of \eqref{eqC} that  corresponds to the  initial datum $u_{0,k}$. According to Proposition \ref{weak.comp}, $\{u_k(t)\}_{k\in\N}$ is a monotone sequence, so its pointwise limit  always exists and then we can  define the candidate to limit solution as
\begin{equation} \label{tilde.ae}
u(x,t):=\liminf_{k\to\infty} u_k(x,t).
\end{equation}
Observe that $\{ u_k( t)\}_{k\in\N}$ is a Cauchy sequence in $\LL^q(\R^n)$:  by using Proposition \ref{norm.decreasing.difference} we have that
\begin{align*}
\| u_k(t)- u_m(\cdot,t)\|_{\LL^q(\R^n)} &\leq \| u_{0,k}- u_{0,m}\|_{\LL^q(\R^n)}\\ &\leq \| u_{0,k}- u_0\|_{\LL^q(\R^n)} + \| u_{0,m}- u_0\|_{\LL^q(\R^n)} \to 0 \; \text{ as }m,k\to\infty.
\end{align*}
Therefore $u_k(t) \to u(t)$  strongly in  $\LL^q(\R^n)$ for any $t>0$.

\medskip

We will now show that this limit function is a solution of \eqref{eqC} that corresponds to the initial datum $u_0$. For \text{a.e.} $x\in \R^n$ and $t>0$ let us denote
\begin{align*}
I(x,t)&:=\lim_{k\to \infty} \int_{\R^n} f_k(x,y,t) \,dy, \quad f_k(x,y,t):= J(x-y)| u_k(y,t)-  u_k (x,t)|^{p-2}( u_k(y,t)-   u_k (x,t)),\\
v(x,t)&:=\lim_{k\to \infty} (  u_k)_t(x,t).
\end{align*}
Due to \eqref{tilde.ae} we have the following pointwise limit:
$$
\lim_{k\to \infty} f_k(x,y,t) = f(x,y,t):= J(x-y)| u(y,t)-  u(x,t)|^{p-2}(  u(y,t)-  u(x,t)).
$$
Moreover, from Theorem \ref{main}  and \eqref{tilde.0} we get that
\begin{align*}
\int_{\R^n} |f_k|\,dy \leq 2^p\|  u_k(t)\|^{p-1}_{\LL^\infty(\R^n)} \int_{\R^n} J(x-y)\,dy\leq
\begin{cases}
\mathsf{c}_{p,q,J} \left(t^{-\frac{1}{p-2}}  + \| u_{0}\|_{\LL^q(\R^n)} \right)^{p-1} &\text{ when } p>2\\
2^p \|u_0\|_{\LL^\infty(\R^n)}^{p-1} &\text{ when } p\in (1,2],
\end{cases}
\end{align*}
where $\mathsf{c}_{p,q,J}$ is a constant independent of $k$. As a consequence, by the Dominated Convergence Theorem we get that
$$
I(x,t)=\int_{\R^n} f(x,y,t)\,dy = \J_p   u(x,t) \qquad \text{ a.e. } x\in\R^n \,\text{ and } t>0.
$$
Let us identify the function $v$. First, observe that by the lower semicontinuity of the norm, Theorem \ref{smooth.ut} and \eqref{tilde.0}, $v$ is uniformly bounded for any $t>0$, indeed,
\begin{align*}
\|v(t)\|_{\LL^\infty(\R^n)} \leq \liminf_{k\to\infty} \| ( u_k)_t(\cdot,t)\|_{\LL^\infty(\R^n)}
\leq
\begin{cases}
\mathsf{c}_{p,q,J}  \left(   t^{-\frac{1}{p-2} }  + \|u_0\|_{\LL^q(\R^n)} \right)^{p-1} &\text{ when } p>2\\
 2^p\|u_0\|_{\LL^\infty(\R^n)}^{p-1} &\text{ when } p\in (1,2].
\end{cases}
\end{align*}
Moreover, $\{(u_k)_t\}_{k\in\N}$ is  integrable as function of the time: from Proposition \ref{smooth.ut}, for $t,h>0$ we have that
\begin{equation} \label{bound.t}
\int_t^{t+h} |(  u_k)_t(\cdot,\tau)|\,d\tau
\leq
2^p \int_{t}^{t+h}\|u(t_0)\|_{\LL^\infty(\R^n)}^{p-1} \,d\tau <\mathsf{c}
\end{equation}
where $t> 0$ and  $\mathsf{c}$ is independent of $k$. We write now the difference quotient for $u_k$ in terms of  a Steklov average as
$$
\frac{  u_k(x,t+h) -   u_k(x,t)}{h} = \frac{1}{h} \int_t^{t+h} (u_k)_t (x,\tau)\,d\tau, \qquad \text{ a.e. }x\in \R^n \text{ and } t>0.
$$
Using \eqref{tilde.ae}, \eqref{bound.t}, and the Dominated Convergence Theorem,  as $k\to\infty$ we get
$$
\frac{  u(x,t+h) -   u(x,t)}{h} = \frac{1}{h} \int_t^{t+h} v (x,\tau)\,d\tau, \qquad \text{ a.e. }x\in \R^n \text{ and } t>0.
$$
Finally, as $h\to 0^+$ gives $u_t=v$. Hence, $u$ solves $u_t(x,t)= \J_p u(x,t)$ for \text{a.e.} $x\in \R^n$ and $t>0$, with initial datum $u(x,0)=u_0(x)$, \text{a.e.} $x\in \R^n$, that is, $u$ is a solution of \eqref{eqC}.
\end{proof}

The following contraction principles for solutions are stated in  \cite{AMRT}.

\begin{prop}[Theorems 6.2, 6.24 and 6.37 in \cite{AMRT}] 
Given $p>1$, let $u$ and  $v$ be solutions in $[0,T]$ to \eqref{eqC} with initial data $u_0, v_0 \in \LL^1(\R^n)$, respectively. Then
\begin{equation} \label{t.decay.1}
\int_{\R^n} (u(x,t)- v(x,t))^+\,dx \leq \int_{\R^n} (u_0(x)- v_0(x))^+\,dx \quad \forall \in [0,T].
\end{equation}
Moreover, if $u_0,v_0\in \LL^p(\R^n)$, then
$\|u( t)-v( t)\|_{\LL^p(\R^n)}\leq \|u_0-v_0\|_{\LL^p(\R^n)}$ for all $t\in[0,T]$.

Given an open bounded set $\Omega\subset \R^n$, the same contraction holds for solutions
$u$ and  $v$ in $[0,T]$ of \eqref{eqD} and \eqref{eq} with initial data $u_0, v_0 \in \LL^1(\Omega)$, respectively.
\end{prop}

We extend the above result to initial data in  $\LL^q$ for $1\leq q\leq \infty$.

\begin{prop} [$\LL^q$-contractivity] \label{norm.decreasing.difference}
Given $p>1$,  let $u,v$ be   solutions of \eqref{eqC} corresponding to the initial data $u_0,v_0\in \LL^q(\R^n)$ with $1\leq q\leq \infty$. Then
$$
\|u(t)-v(t)\|_{\LL^q(\R^n)} \leq \|u(t_0)-v(t_0)\|_{\LL^q(\R^n)} \leq \|u_0-v_0\|_{\LL^q(\R^n)}
$$
for all $0< t_0\leq t \leq T$.

The same holds for solutions of \eqref{eqD} and \eqref{eq} when $\Omega\subset \R^n$ is open and  bounded.
\end{prop}

\begin{proof}
We will prove the result for solutions of \eqref{eqC}. A similar reasoning can be applied for solutions of \eqref{eqD} and \eqref{eq} and when $\Omega\subset \R^n$ is open and bounded.

Suppose we are given solutions $u$ and $v$ of \eqref{eqC} with initial data $u_0,v_0\in \LL^q(\R^n)$, where $1\leq q\leq \infty$.  First, we assume that $1< q<\infty$  and we obtain the following:	
\begin{align*}
\frac{d}{dt} &\|u(t)- v(t)\|_{\LL^q(\R^n)}^q =q\int_{\R^n} L_q(u(x,t)-v(x,t))(u_t(x,t)-v_t(x,t))\,dx\\
&=
-q\iint_{\R^n\times\R^n} J(x-y) L_q(u(x,t)-v(x,t)) G_p(u,v)\,dxdy\\
&=
q\iint_{\R^n\times\R^n} J(x-y) L_q(u(y,t)-v(y,t)) G_p(u,v)\,dxdy
\end{align*}
where  we interchanged the  variables $x$ and $y$ and used the symmetry of $J$. Here we have denoted
$$
G_p(u,v):=L_p(u(x,t)-u(y,t))) -  L_p(v(x,t)-v(y,t)).
$$
Therefore, denoting $H_q(u,v):=L_q(u(x,t)-v(x,t)) - L_q(u(y,t)-v(y,t))$ we get
\begin{align*}
\frac{d}{dt} &\|u(t)- v(t)\|_{\LL^q(\R^n)}^q =
-\frac{q}{2}\iint_{\R^n\times\R^n} J(x-y) H_q(u,v) G_p(u,v)\,dxdy
\end{align*}
We denote $A=u(x,t)-u(y,t)$, $B=v(x,t)-v(y,t)$, $a=u(x,t)-v(x,t)$, $b=u(y,t)-v(y,t)$. Since $A-B=a-b$, when  $A=B$ or $a=b$, we have that $
\frac{d}{dt} \|u(t)- v(t)\|_{\LL^q(\R^n)}=0
$.

Consider now the case $A\neq B$ and $a\neq b$. We use  the following well-known numerical inequality for $\alpha,\beta\in \R$ and $r>1$:
\begin{align*} 
(|\alpha|^{r-2} \alpha - |\beta|^{r-2})(\alpha-\beta)  \geq
\begin{cases}
\mathsf{c}_r |\alpha-\beta|^r & \text{ if } r\geq 2\\
\mathsf{c}_r \frac{|\alpha-\beta|^2}{|\alpha|^{2-r} + |\beta|^{2-r}} & \text{ if } 1<r<2\\
\end{cases}
\end{align*}
where $\mathsf{c}_r\sim (r-1)$. See Lemma 4.4 in  \cite{D93} and Appendix A.3 in \cite{BIV10} for a proof of it. When $p,q\geq 2$
$$
H_q(a,b)(a-b) \geq \mathsf{c}_q |a-b|^q>0, \quad G_p(A,B)(A-B) \geq \mathsf{c}_p |A-B|^p> 0,
$$
from where, since $|A-B|=|a-b|$ we obtain that
$$
- H_q(a,b) G_p(A,B) = - \frac{1}{|a-b|^2} H_q(a,b)(a-b) G_p(A,B)(A-B) < 0,
$$
from where, since $J\geq 0$, we get that $\frac{d}{dt} \|u(t)- v(t)\|_{\LL^q(\R^n)}^q \leq 0$. A similar argument can be applied when $1<p<2$ or  $1<q<2$ to reach the same conclusion.

Moreover, the same estimate still true when $q\to 1$ since for all $\alpha,\beta\in \R$
$$
\left(\frac{\alpha}{|\alpha|}- \frac{\beta}{|\beta|} \right)(\alpha-\beta) \geq
\begin{cases}
0 & \text{ if } \alpha \text{ and } \beta \text{ have the same sign}\\
2|\alpha-\beta| & \text{ if } \alpha \text{ and } \beta \text{ have different sign},
\end{cases}
$$
and therefore $H_1(a,b)(a-b)\geq 0$.

Finally, from the previous computations, for any $p>1$ and any $q\ge 1$
$$
\frac{d}{dt}\|u(t)-v(t)\|_{\LL^q(\R^n)} = -\frac12 \|u(t)-v(t)\|_{\LL^q(\R^n)}^{1-q} \iint_{\R^n\times\R^n} J(x-y) H_q(u,v) G_p(u,v)\,dxdy \leq 0
$$
Letting $q\to\infty$ in the above expression gives the result for the $L^\infty$ norm. 
\end{proof}

The following comparison principle is a consequence of the  contractivity for the difference of solutions stated in \eqref{t.decay.1}.

\begin{prop}[Weak comparison] \label{weak.comp}
Given $p>1$, let $u$ and $v$ be solutions in $[0,T]$ of \eqref{eqC} with initial data $u_0,v_0\in \LL^q(\R^n)$, $1\leq q < \infty$,  respectively. Hence, if $u_0 \leq v_0$ \text{a.e.} in $\Omega$, then $u(x,t)\leq v(x,t)$ \text{a.e.} in $\R^n$ and $t\in (0,T)$.

When $\Omega\subset \R^n$ is open and bounded, the same conclusion holds for solutions $u$ and $v$ in $[0,T]$ of \eqref{eqD} and \eqref{eq} with initial data $u_0,v_0\in \LL^q(\Omega)$, $1\leq q\leq \infty$, respectively.
\end{prop}

\begin{proof}
Let $u_0,v_0\in \LL^q(\R^n)$ with $1\leq q\leq \infty$   be such that $u_0\leq v_0$ \text{a.e.}  $x\in \R^n$. Consider the sequences
of approximations $\{u_{0,k}\}_{k\in\N},\{v_{k,0}\}_{k\in\N}\in \LL^\infty_c(\R^n)$ given by $u_{0,k}:=\min\{u_0,k\}\chi_{B_{k}(0)}$, $v_{0,k}:=\min\{v_0,k\}\chi_{B_{k}(0)}$. Therefore, $\{u_{0,k}\}_{k\in\N}$ and $\{v_{0,k}\}_{k\in\N}$ are monotone sequences and
\begin{align*}
&u_{0,k} \to u_0, \quad v_{0,k}\to v_0 \quad \text{strongly in }\LL^q(\R^n) \text{ and a.e. in } \R^n,\\
&\|u_{0,k}\|_{\LL^q(\R^n)} \leq \|u_0\|_{\LL^q(\R^n)}, \qquad\|v_{0,k}\|_{\LL^q(\R^n)} \leq \|v_0\|_{\LL^q(\R^n)}.
\end{align*}
Then, it follows that
\begin{equation} \label{e.cd}
\liminf_{k\to\infty} (u_{0,k}-v_{0,k})_+ =(u_0-v_0)_+=0,\quad  \text{ and } \quad  \|(u_{0,k}-v_{0,k})_+\|_{\LL^q(\R^n)} \leq \|u_0\|_{\LL^q(\R^n)}+ \|v_0\|_{\LL^q(\R^n)}.
\end{equation}

Due to Theorem \ref{existencia}, for each $k\in\N$ there exist $u_k$, $v_k$ solutions of \eqref{eq} corresponding to the initial data $u_{0,k}$, $v_{0,k}$, respectively. By \eqref{t.decay.1}, $\{u_k\}_{k\in \N}$ and $\{v_k\}_{k\in\N}$ are also monotone, so for \text{a.e.} $x\in \R^n$ and $t>0$ we can define the limits
$$
u(x,t):=\liminf_{k\to\infty} u_k(x,t), \qquad v(x,t):=\liminf_{k\to\infty} v_k(x,t),
$$
and $u_k(t) \to u(t)$, $v_k(\cdot,t) \to v(t)$ strongly in $\LL^q(\R^n)$ for $t>0$. Moreover,
$$
\liminf_{k\to\infty}(u_k(x,t)-v_k(x,t))_+ = (u(x,t)-v(x,t))_+.
$$
Then, by Fatou's Lemma, \eqref{t.decay.1}, and the Dominated Convergence Theorem with \eqref{e.cd}, we get that
\begin{align*}
\| (u(t) - v(t) )_+\|_{\LL^1(\R^n)} \leq
\liminf_{k\to\infty} \| (u_k(t) - v_k(\cdot,t) )_+\|_{\LL^1(\R^n)}
\leq \lim_{k\to\infty}
\| (u_{k,0} - v_{k,0} )_+\|_{\LL^1(\R^n)} =0,
\end{align*}
that is, $u(x,t)\leq v(x,t)$ \text{a.e.} $x\in \R^n$ and $t>0$.

The proof for the   Dirichlet and Neumann problems is similar.
\end{proof}

\begin{rem}
The comparison principles available in the literature often require additional assumptions on the regularity of solutions and initial data (c.f. Corollary 2.9 in \cite{AMRT} and Theorem 2.1 in \cite{FPL}). However, our Proposition \ref{weak.comp} holds for solutions in $W^{1,1}(0,T;\LL^q(\Omega))$ and initial data in $\LL^q(\Omega)$ with $1\leq q \leq \infty$.
\end{rem}

We recall a useful  integration formula stated in Lemma 6.5 of \cite{AMRT}.

\begin{prop}[Integration by parts] \label{parts}
For every $u,v\in \LL^p(\Omega)$, $\Omega\subseteq\R^n$, it holds that
$$
-\int_\Omega v(x) \J_{p,\Omega} u(x) \,dx = \E(u (x),v (x)).
$$
\end{prop}

We prove now a  proposition in the spirit of the celebrated  B\'enilan-Crandall estimates \cite{BC81}.

\begin{prop}[Benilan-Crandall type estimates when $p>2$] \label{BC}
Let $p>2$ and let $u\ge 0$ be a solution of \eqref{eqC}. We have that
\begin{equation}\label{BC0}
u_t(x,t)\geq -\frac{u(x,t)}{(p-2)t}, \quad \emph{ a.e. } x\in \R^n \text{ and } t>0.
\end{equation}
The same holds for solutions of \eqref{eqD} or of \eqref{eq} when $\Omega\subset \R^n$ is a bounded set.
\end{prop}
\begin{proof}
Take $\lambda\geq 1$ and consider the rescaled (in time) solution to \eqref{eqC}
$$
u_\lambda(x,t)=\lambda^\frac{1}{p-2} u(x,\lambda t) \quad \text{ with }\quad  u_\lambda(x,0)=\lambda^\frac{1}{p-2}u_0(x) \geq u_0(x).
$$
By the comparison principle (Lemma  \ref{weak.comp}), since $\lambda \geq 1$ it follows that $u_\lambda(x,t)\geq u(x,t)$ \text{a.e.}  $x\in\Omega$ and $t>0$. Moreover, taking $\lambda=\frac{t+h}{t}\geq 1$ and adding and substracting $u_\lambda(x,t)$ we obtain
\begin{align*}
\frac{u(x,t+h)-u(x,t)}{h}
&=
\frac1h \left[\left(\frac{t+h}{t}\right)^{-\frac{1}{p-2}}- 1\right]u_\lambda(x,t)+ \frac{u_\lambda(x,t)-u(x,t)}{h}\\
&\geq
\frac1h \left[\left(\frac{t+h}{t}\right)^{-\frac{1}{p-2}}- 1\right]u_\lambda(x,t)\\
&=\frac{(t+h)^{-\frac{1}{p-2}} - t^{-\frac{1}{p-2}}}{h} (t+h)^\frac{1}{p-2} u(x,t+h).
\end{align*}
The result follows by letting $h\to 0^+$ . The reasoning for solutions of \eqref{eqD} or of \eqref{eq} is identical.
\end{proof}

\begin{lema}[Time-monotonicity  for $p>2$]  \label{time.mono}
Let $p>2$ and $u$ be a solution corresponding to the datum $u_0$ to any of the problems \eqref{eqC}, \eqref{eqD} or  \eqref{eq}. Then, the map
\[
t\mapsto t^\frac{1}{p-2}u(x,t)\qquad\mbox{is nondecreasing for all $t>0$ and \emph{a.e.} $x\in \R^n$},
\]
for solutions to \eqref{eqC}. For solutions to \eqref{eqD} or  \eqref{eq} the same holds for \emph{a.e.} $x\in \Omega$.
\end{lema}

\begin{proof}Integrating in time the Benilan-Crandall estimates \eqref{BC0} on $[t,t+h]$  we immediately obtain that $t^\frac{1}{p-2} u(x,t)\leq (t+h)^\frac{1}{p-2}u(x,t+h)$ for all $h\geq 0$, as required.
\end{proof}

\subsection{The weak cone condition} \label{sect.weak.cone}

We close this section stating a technical condition we use to deal with the smoothing effect for the Neumann problem.

Given $x\in \Omega$, let $R(x)$ consist of all points $y\in \Omega$ such that the line segment joining $x$ to $y$ lies entirely in $\Omega$; thus $R(x)$ is a union of rays and line segments emanating from $x$. Let
$$
\Gamma(x)=\{y\in R(x)\colon |y-x|<1\},
$$
and let $|\Gamma(x)|$ denote the $n-$dimensional Lebesgue measure of $\Gamma(x)$. We say that $\Omega$ satisfies the \emph{weak cone condition} if there exists a number $\delta>0$ such that
$$
|\Gamma(x)| \geq \delta \quad \text{for all }x\in\Omega.
$$
Clearly, the cone condition implies the weak cone condition.

\begin{lema} \cite[Lemma 1]{AF}
Let $\Omega\subset \R^n$ satisfy the weak cone condition. Then there exist positive constants $\eta \leq 1$, $A$ and $B$ depending on $n$ and $\delta$, and for each $x\in \Omega$ a subset $P_{x,\eta}\subset \mathcal{S}^{n-1}$, such that $\lambda(P_{x,n})=A$ and $x+t\sigma \in \Omega$ if $\sigma\in P_{x,\eta}$ and $0<t<\eta$. In particular, for each $x\in \Omega$ and each $\xi$ satisfying $0<\xi\leq \eta$, the \emph{generalized cone} $C_{x,\xi}=\{y=x+t\sigma \in \R^n \colon \sigma \in P_{x,n}, 0<t<\xi \}$ satisfies $C_{x,\xi}\subset \Omega$ and $|C_{x,\xi}|=B \xi^n$.
\end{lema}

\section{$\LL^q$--$\LL^\infty$ smoothing effect for solutions} \label{section4}
In this section, we prove an $\LL^q$--$\LL^\infty$ smoothing effect for solutions of \eqref{eqC}, \eqref{eq}, and \eqref{eqD}, for all $q\in [1,\infty]$.

\medskip

For the reader's convenience, we recall the following version of the  \emph{Young's inequality}. For $\varepsilon>0$, $a,b>0$ and $p,q>1$ such that $\frac{1}{p}+\frac{1}{q}=1$ it holds that
$$
ab\leq \varepsilon a^p+(\varepsilon p)^{-\frac{q}{p}}q^{-1} b^q.
$$

The following numerical estimate is key in our arguments.
\begin{lema} \label{pointw.ineq}
Given $p\geq 2$, for all $a,b\geq 0$ it holds that
$$
a^{p-1}-|a-b|^{p-2}(a-b)\leq (p-1)\max\{a^{p-2},b^{p-2}\}b.
$$
\end{lema}

\begin{proof}
We assume $p>2$ since the inequality is trivial when  $p=2$. We have to prove that
\begin{align*}
a^{p-1}\left( 1-\left|1-\frac{b}{a} \right|^{p-2}\left(1-\frac{b}{a}\right)\right)\leq
\begin{cases}
(p-1) a^{p-2}b &\text{ if } b\leq a,\\
(p-1) b^{p-1} &\text{ if } a<b.
\end{cases}
\end{align*}
Denoting $t=\frac{b}{a}$, the last expression is equivalent to
\begin{align*}
1-\left|1-t \right|^{p-2}\left(1-t\right) \leq
\begin{cases}
(p-1) t &\text{ if } t\leq 1,\\
(p-1) t^{p-1} &\text{ if } t>1.
\end{cases}
\end{align*}
When $t\leq 1$, by using Taylor's expansion we have
\begin{align*}
1-|1-t|^{p-2}(1-t)&=1-(1-t)^{p-1}\\
&=1-\left( 1-(p-1)t + (p-1)(p-2) (1-\tilde t)^{p-3} \frac{t^2}{2}\right)
 \leq (p-1)t
\end{align*}
where  $\tilde t\in [0,t]$, which gives the result when $t\le 1$. When $t>1$ we have
\begin{align*}
1-|1-t|^{p-2}(1-t)=1+(t-1)^{p-1} \leq (p-1)t^{p-1}
\end{align*}
where in the last inequality we used that  $f(t)=1+(t-1)^{p-1}$ and $g(t)=(p-1)t^{p-1}$ are such that $f(1)=g(1)$ and $f'(t)\leq g'(t)$ for $t>1$. This gives the result for $t>1$ and concludes the proof.
\end{proof}

\begin{thm}[Cauchy] \label{main}
Let $u$ be the solution of \eqref{eqC}  corresponding to the initial datum $u_0$. Then the following holds:

\begin{enumerate}[leftmargin=15pt, label=(\roman*)]\itemsep2pt \parskip3pt \parsep2pt
\rm \item \it For any $p\in (2,\infty)$ and $u_0\in \LL^q(\R^n)$, $q\in [1,\infty]$
$$
\|u( t)\|_{\LL^\infty(\R^n)} \leq
\frac{  \mathsf{\tilde K}_p}{t^{\frac{1}{p-2}}}   + \mathsf{K}_{p,q,J}\|u( t_0)\|_{\LL^q(\R^n)}  \quad \text{ for all } 0<t_0\leq  t\leq T.
$$
When $u_0\in \LL^\infty(\R^n)$, we can set $t_0=0$  and get the estimate below.

\rm \item \it For any $p\in (1,\infty)$ and $u_0\in \LL^\infty(\R^n)$
$$
\|u(t)\|_{\LL^\infty(\R^n)}   \leq \|u(t_0)\|_{\LL^\infty(\R^n)}  \quad \text{ for all }  0\leq t_0 \leq t.
$$
\end{enumerate}
Here, $\mathsf{\tilde K}_p$ and $\mathsf{K}_{p,q,J}$ are the explicit constants depending of $p$, $q$, $n$ and $\|J\|_{\LL^\infty(\R^n)}$ given in \eqref{ctes}.

\end{thm}

\begin{proof}

Let $u$ be the solution of \eqref{eqC} with $p\in (1,\infty)$ corresponding to the initial datum $u_0$.

When $u_0\in \LL^\infty(\R^n)$,   inequality $(ii)$ is just a consequence of the $\LL^\infty$--contractivity property of solutions stated in Proposition \ref{norm.decreasing.difference}. Hence, we focus in $(i)$ for a fixed $p\in (2,\infty)$. For the sake of simplicity we split the proof in several steps.

\medskip

\noindent$\circ~$\textsc{Step 1}. \textit{Reduction argument}. Given a nonnegative function $u_0\in \LL^q(\R^n)$, $q\geq 1$, we consider the sequence of approximations $\{u_{0,k}\}_{k\in\N}$ given by $u_{0,k}=\min\{u_0,k\}$.

We denote by $u_k$ the solution of \eqref{eqC} corresponding to the datum $u_{0,k}$. Observe that by Proposition \ref{weak.comp}, $u_k(x,t)\geq 0$ for \text{a.e.} $x\in \R^n$ and $t\in (0,T)$. Moreover, by the contraction principle stated in Proposition \ref{norm.decreasing.difference} we have that $u_k( t) \in \LL^\infty(\R^n)$ for every $t\in [0,T]$.

\medskip

\noindent$\circ~$\textsc{Step 2}. \textit{$\LL^{p+q-2}$--$\LL^\infty$ smoothing effect}. By definition of solution  we have that $(u_k)_t(x,t)=\J_pu_k(x,t)$ \text{a.e} in $\R^n\times (0,T)$. Adding and subtracting $u_k^{p-1}(x,t)$,  making use of the B\'enilan-Crandall estimate given in Lemma \ref{BC} and the fact that $\|J\|_{\LL^1(\R^n)}=1$, it is obtained that
$$
-\frac{u_k(x,t)}{(p-2)t} \leq (u_k)_t(x,t) = -u_k^{p-1}(x,t) + \mathcal{I}(u_k), \quad \text{ a.e. }x\in \R^n \text{ and } t\in (0,T)
$$
where  we have denoted
$$
\mathcal{I}(u_k):=\int_{\R^n} J(x-y)\left\{|u_k(y,t)-u_k(x,t)|^{p-2}(u_k(y,t)-u_k(x,t))+u_k^{p-1}(x,t)\right\} \,dy.
$$
This gives the following relation for \text{a.e.} $x\in \R^n$ and $t\in (0,T)$
$$
u_k^{p-1}(x,t) \leq  \frac{u_k(x,t)}{(p-2)t}+\mathcal{I} (u_k).
$$
From this,  using  Young's inequality with $\varepsilon>0$ to determine, we get
$$
u_k^{p-1}(x,t) \leq  \varepsilon u_k^{p-1}(x,t) + \mathsf{c}_{\varepsilon} t^{-\frac{p-1}{p-2}} +\mathcal{I}(u_k)
$$
with $\mathsf{c}_{\varepsilon}:=\varepsilon^{-\frac{1}{p-2}}(p-2)^{-\frac{1}{p-2}}(p-1)^{-\frac{p-1}{p-2}}$, from where
\begin{equation} \label{eq.u.1}
u_k^{p-1}(x) \leq   \frac{\mathsf{c}_{\varepsilon}}{1-\varepsilon} t^{-\frac{p-1}{p-2}}+\frac{1}{1-\varepsilon}\mathcal{I}(u_k).
\end{equation}
Let us estimate $\mathcal{I}(u_k)$. Using Lemma \ref{pointw.ineq} we get
\begin{align*}
\mathcal{I}(u_k) &\leq
(p-1)\int_{\R^n} J(x-y)\,\max\{ u_k(x,t)^{p-2}, u_k(y,t)^{p-2}\} u_k(y,t)\,dy\\
&\leq
(p-1) u_k(x,t)^{p-2}\int_{\R^n} J(x-y) u_k(y,t)\,dy+(p-1)\int_{\R^n} J(x-y) u_k(y,t)^{p-1}\,dy:= (i)+ (ii).
\end{align*}
Since $J\in \LL^1(\R^n)\cap \LL^\infty(\R^n)$ and $\|J\|_{\LL^1(\R^n)}=1$, using interpolation, for any $1\leq r \leq \infty$ we have that
\begin{equation} \label{J.interp}
\|J\|_{\LL^r(\R^n)} \leq \|J\|_{\LL^1(\R^n)}^\frac{1}{r} \|J\|_{\LL^\infty(\R^n)}^{1-\frac{1}{r}} =  \|J\|_{\LL^\infty(\R^n)}^{1-\frac{1}{r}}.
\end{equation}
We use Young's inequality with $\varepsilon_1>0$ to determinate, H\"older's inequality  and \eqref{J.interp} to get that
\begin{align*}
(i)&\leq \varepsilon_1 u_k(x,t)^{p-1} + \mathsf{c}_{\varepsilon_1} \left(\int_{\R^n} J(x-y) u_k(y,t)\,dy\right)^{p-1}\\
&\leq \varepsilon_1 u_k(x,t)^{p-1} + \mathsf{c}_{\varepsilon_1} \left(\int_{\R^n} J(x-y)^\frac{p+q-2}{p+q-3}\,dy\right)^\frac{(p-1)(p+q-3)}{p+q-2} \left(\int_{\R^n} u_k(y,t)^{p+q-2}\,dy \right)^\frac{p-1}{p+q-2}\\
&\leq \varepsilon_1 u_k(x,t)^{p-1} +
\mathsf{c}_{\varepsilon_1}\|J\|_{\LL^\infty(\R^n)}^\frac{p-1}{p+q-2} \|u_k(t)\|_{\LL^{p+q-2}(\R^n)}^{p-1},
\end{align*}
where $\mathsf{c}_{\varepsilon_1}$ is given by
$$
\mathsf{c}_{\varepsilon_1}:=\left(\frac{p-2}{\varepsilon_1}\right)^{p-2}.
$$
Moreover, using \eqref{J.interp} we have that
\begin{align*}
(ii) &\leq (p-1) \left( \int_{\R^n} u_k(y,t)^{p+q-2}\,dy\right)^\frac{p-1}{p+q-2} \left( \int_{\R^n} J(x-y)^\frac{p+q-2}{q-1}\,dy\right)^\frac{q-1}{p+q-2}\\
&\leq (p-1)\|J\|_{\LL^\infty(\R^n)}^\frac{p-1}{p+q-2} \|u_k(t)\|_{\LL^{p+q-2}(\R^n)}^{p-1}.
\end{align*}
Gathering the last two bounds yields
\begin{equation}\label{eq.I.1}
\mathcal{I}(u_k) \leq  \varepsilon_1 u_k(x,t)^{p-1} +  (1+\mathsf{c}_{\varepsilon_1} ) \mathsf{c}_1 \|u_k(t)\|_{\LL^{p+q-2}(\R^n)}^{p-1},
\end{equation}
where $\mathsf{c}_1:=(p-1)\|J\|_{\LL^\infty(\R^n)}^\frac{p-1}{p+q-2}$. Inserting \eqref{eq.I.1} into \eqref{eq.u.1} yields
\begin{align*}
u_k^{p-1}(x,t) \leq   \frac{\mathsf{c}_{\varepsilon}}{1-\varepsilon} t^{-\frac{p-1}{p-2}}+
\frac{\varepsilon_1}{1-\varepsilon} u_k(x,t)^{p-1} + \frac{ \mathsf{c}_1 (1+\mathsf{c}_{\varepsilon_1})  }{1-\varepsilon}\|u_k(t)\|_{\LL^{p+q-2}(\R^n)}^{p-1}
\end{align*}
and then, denoting $\kappa_{\varepsilon_1}=\mathsf{c}_1(1+\mathsf{c}_{\varepsilon_1}) $, we get
\begin{align} \label{step2}
u_k^{p-1}(x,t) \leq   \frac{\mathsf{c}_{\varepsilon} t^{-\frac{p-1}{p-2}} }{1-\varepsilon-\varepsilon_1} +
\frac{\kappa_{\varepsilon_1}}{1-\varepsilon-\varepsilon_1}  \|u_k(t)\|_{\LL^{p+q-2}(\R^n)}^{p-1}, \quad  \text{ a.e. }x\in \R^n,\ \text{ and } t\in (0,T).
\end{align}

\medskip

\noindent$\circ~$\textsc{Step 3}. \textit{$\LL^{q}$--$\LL^\infty$ smoothing}. We use now Young's inequality with $\varepsilon_2>0$ to determinate, to get that
\begin{align*}
\|u_k(t)\|_{\LL^{p+q-2}(\R^n)}^{p-1} = \left(\int_{\R^n} u_k(x,t)^{p-2} u_k (x,t)^q \,dx\right)^\frac{p-1}{p+q-2} &\leq \|u_k(t)\|_{\LL^\infty(\R^n)}^\frac{(p-1)(p-2)}{p+q-2} \|u_k(t)\|_{\LL^q(\R^n)}^\frac{q(p-1)}{p+q-2}\\
&\leq \varepsilon_2 \|u_k(t)\|_{\LL^\infty(\R^n)}^{p-1} + \kappa_{\varepsilon_2} \|u_k(t)\|_{\LL^q(\R^n)}^{p-1},
\end{align*}
where $\kappa_{\varepsilon_2}= \left( \frac{p-2}{\varepsilon_2}\right)^\frac{p-2}{q} q (p+q-2)^{-\frac{p+q-2}{q}}$. Hence, inserting this expression into \eqref{step2}, we get
\begin{align*}
\left(1- \frac{\kappa_{\varepsilon_1}}{1-\varepsilon-\varepsilon_1}  \varepsilon_2 \right) \|u_k(t)\|_{\LL^\infty(\R^n)}^{p-1} \leq   \frac{\mathsf{c}_{\varepsilon}}{1-\varepsilon-\varepsilon_1} t^{-\frac{p-1}{p-2}}
+
\frac{\kappa_{\varepsilon_1} \kappa_{\varepsilon_2} }{1-\varepsilon-\varepsilon_1}  \|u_k(t)\|_{\LL^q(\R^n)}^{p-1},
\end{align*}
which gives
\begin{align} \label{step3}
 \|u_k(t)\|_{\LL^\infty(\R^n)}^{p-1} \leq
\frac{\mathsf{c}_{\varepsilon}t^{-\frac{p-1}{p-2}}  + \kappa_{\varepsilon_1}\kappa_{\varepsilon_2} \|u_k(t)\|_{\LL^q(\R^n)}^{p-1}}{1-\varepsilon-\varepsilon_1-\varepsilon_2 \kappa_{\varepsilon_1}} .
\end{align}

\medskip

\noindent$\circ~$\textsc{Step 4}. \textit{Choosing $\varepsilon$, $\varepsilon_1$ and $\varepsilon_2$}. We choose $\varepsilon=\varepsilon_1=\frac18$ and $\varepsilon_2=\frac{1}{4 \kappa_{\varepsilon_1}}$. This election gives that  $\varepsilon+\varepsilon_1+\varepsilon_2 \kappa_{\varepsilon_1} =\frac12$. Therefore,  \eqref{step3}  and the  sub-additivity of the concave function $0\leq x\mapsto x^\frac{1}{p-1}$ yield
\begin{align*} 
 \|u_k(t)\|_{\LL^\infty(\R^n)}  \leq
 \mathsf{\tilde K}_p t^{-\frac{1}{p-2}}  + \mathsf{K}_{p,q,J} \|u_k(t)\|_{\LL^q(\R^n)}
\end{align*}
where the constants can be taken as
\begin{equation} \label{ctes}
\mathsf{\tilde K}_p= 2 \left( \frac{8}{p-2}\right)^\frac{1}{(p-2)(p-1)}, \quad
(\mathsf{K}_{p,q,J})^{p-1}=
q(8p)^\frac{p(p+q)}{q} \|J\|_{\LL^\infty(\R^n)}^\frac{p-1}{q}.
\end{equation}

\medskip

\noindent$\circ~$\textsc{Step 5}. \textit{Limit as $k\to\infty$}. By construction  we have that $0\leq u_{0,k} \leq u_{0,k+1}$ \text{a.e.} in $\R^n$ for all $k\in\N$ and
$$
\lim_{k\to\infty} u_{0,k} = u_0 \quad \text{a.e. monotonically from below.}
$$
Then, Proposition \ref{weak.comp} yields
$$
0\leq u_k(x,t) \leq u_{k+1}(x,t) \quad \text{a.e. } x\in \R^n \text{ and } t\in [0,T).
$$
By monotonicity, the pointwise limit of $\{u_k(x,t)\}_{k\in\N}$ always exists (possibly being $+\infty$ on a set of measure zero), and then we define the candidate to limit solution as
$$
u(x,t):=\liminf_{k\to\infty} u_k(x,t).
$$
Observe that by the uniqueness of solution, we have that the limit function $u$ is indeed a solution of \eqref{eqC} with datum $u_0$. Moreover, by the lower semicontinuity of the norm, the Monotone Convergence  Theorem  and Proposition \ref{norm.decreasing.difference}, we have that
$$
\|u(t)\|_{\LL^q(\R^n)}\leq \liminf_{k\to\infty} \|u_k(t)\|_{\LL^q(\R^n)}\leq \|u(t)\|_{\LL^q(\R^n)} \leq \|u_0\|_{\LL^q(\R^n)}.
$$
As a consequence, the set of $(x,t)\in \R^n\times[0,T]$ where $u(x,t)=+\infty$ has measure zero, and then the convergence above holds almost everywhere.  Finally, the above estimate together with the lower semicontinuity of the norm and  Proposition \ref{norm.decreasing.difference} yield
\begin{align*}
\|u(t)\|_{\LL^\infty(\R^n)} \leq \liminf_{k\to\infty} \|u_k(t)\|_{\LL^\infty(\R^n)} &\leq
\mathsf{\tilde K}_{p} t^{-\frac{1}{p-2}}  + \mathsf{K}_{p,q,J} \liminf_{k\to\infty} \|u_k(t)\|_{\LL^q(\R^n)}\\
&\leq
\mathsf{\tilde K}_{p} t^{-\frac{1}{p-2}}   + \mathsf{K}_{p,q,J} \|u_0\|_{\LL^q(\R^n)}.
\end{align*}
Therefore, for any nonnegative solution corresponding to the nonnegative initial datum $u_0\in \LL^q(\R^n)$ with $q\in [1,\infty]$ it holds that
\begin{equation} \label{step6}
\|u(t)\|_{\LL^\infty(\R^n)}\leq
\mathsf{\tilde K}_{p} t^{-\frac{1}{p-2}} + \mathsf{K}_{p,q,J}\|u_0\|_{\LL^q(\R^n)}.
\end{equation}

\medskip

\noindent$\circ~$\textsc{Step 6}. \textit{$\LL^q$--$\LL^\infty$ Smoothing for signed solutions}.  In this final step we get rid of the nonnegative assumption on solutions. To do so, let $u$ be a solution of \eqref{eqC} with initial datum $u_0\in \LL^q(\R^n)$.  We remark  that both $u$ and $u_0$ may change sign.

By definition we have that $u_0\leq u_0^+:=\max\{u_0,0\}$ in $\R^n$. Consider the solution $\widetilde{u^+}$ of \eqref{eqC} corresponding to the initial datum $u_0^+$. The comparison principle given in Proposition \ref{weak.comp} yields that $\widetilde{u^+}$ is nonnegative and $u\leq \widetilde{u^+}$ for \text{a.e.} $x\in \R^n$ and $t\in [0,T)$. Therefore, by using the smoothing obtained in \eqref{step6} for nonnegative solutions we get
$$
u(x,t) \leq \|u(t)\|_{\LL^\infty(\R^n)} \leq \|\widetilde{u^+}(\cdot,t)\|_{\LL^\infty(\R^n)} \leq
\mathsf{\tilde K}_{p} t^{-\frac{1}{p-2}}  + \mathsf{K}_{p,q,J} \|u_0^+\|_{\LL^q(\R^n)}  \quad \forall  t\in (0,T].
$$
Observe that by the oddness of $\J_p$, $-u$ is also solution of \eqref{eqC} with initial datum $-u_0$. In this case  we have that $-u_0 \leq u_0^-:=\max\{-u_0,0\}$ in $\R^n$, and we can consider the solution $\widetilde{u^-}$ of \eqref{eqC} with initial datum $u_0^-$. Again, by comparison, $\widetilde{u^-}$ is nonnegative and $-u \leq \widetilde{u^-}$ for \text{a.e.} $x\in \R^n$ and $t\in [0,T)$, from where \eqref{step6} gives that
$$
-u(x,t) \leq \widetilde{u^-}(x,t) \leq \|\widetilde{u^-}(\cdot,t)\|_{\LL^\infty(\R^n)} \leq
\mathsf{\tilde K}_{p}  t^{-\frac{1}{p-2}}  + \mathsf{K}_{p,q,J} \|u_0^-\|_{\LL^q(\R^n)}  \quad \forall  t\in (0,T].
$$
The last two estimates give that
\begin{align*}
\|u(t)\|_{\LL^\infty(\R^n)} &\leq  2 \mathsf{\tilde K}_p t^{-\frac{1}{p-2}}  + \mathsf{K}_{p,q,J} \|u_0^-\|_{\LL^q(\R^n)} +\mathsf{K}_{p,q,J} \|u_0^+\|_{\LL^q(\R^n)} \\
&=  2 \mathsf{\tilde K}_{p} t^{-\frac{1}{p-2}}  + \mathsf{K}_{p,q,J} \|u_0\|_{\LL^q(\R^n)}
\end{align*}
since $u_0^+$ and $u_0^-$ have disjoint supports. The proof is now complete.
\end{proof}

Using similar arguments, we can obtain the smoothing for both the Neumann and homogeneous Dirichlet problems. For the Neumann problem, as mentioned in the introduction, we additionally assume condition \eqref{HJ} on $J$.

\begin{thm}[Neumann and Dirichlet] \label{main2}
Let $\Omega\subset \R^n$ be open and bounded and let $u_D$ and $u_N$ be the solutions of \eqref{eqD} and \eqref{eq}, respectively,  with   initial data $u_{D,0}$ and $u_{N,0}$. Then the following holds:
\begin{enumerate}[leftmargin=15pt, label=(\roman*)]\itemsep2pt \parskip3pt \parsep2pt
\rm \item \it for any $p\in (2,\infty)$ and $u_{D,0}\in \LL^q(\Omega)$, $q\in [1,\infty]$
$$
\|u_D( t)\|_{\LL^\infty(\Omega)} \leq
\frac{\mathsf{\tilde K}_p}{t^{\frac{1}{p-2}}}   + \mathsf{K}_{p,q,J}\|u_D( t_0)\|_{\LL^q(\Omega)}  \quad \text{ for all } 0\le t_0\leq  t\leq T.
$$
When $\Omega$ additionally satisfies condition \eqref{HJ}, then
for any $u_{N,0}\in \LL^q(\Omega)$, $q\in [1,\infty]$
$$
\|u_N( t)\|_{\LL^\infty(\Omega)} \leq
\kappa_{J,\Omega}^{-1} \left(  \frac{\mathsf{\tilde K}_p}{t^{\frac{1}{p-2}}}   + \mathsf{K}_{p,q,J}\|u_N( t_0)\|_{\LL^q(\Omega)} \right) \quad \text{ for all } 0\le t_0\leq  t\leq T.
$$

\rm \item \it For any $p\in (1,\infty)$ and $u_0\in \LL^\infty(\R^n)$
$$
\|u_D(t)\|_{\LL^\infty(\Omega)}   \leq \|u_D(t_0)\|_{\LL^\infty(\Omega)}  \quad \text{ for all }  0\leq t_0 \leq t
$$
and the same holds for $u_N$.
\end{enumerate}
Here,  $\mathsf{\tilde K}_p$ and $\mathsf{K}_{p,q,J}$ are the explicit constants  given in \eqref{ctes}.
\end{thm}

\begin{proof}
The proof in the homogeneous Dirichlet case runs in the very same way as in the Cauchy problem.

We will now highlight the main differences and provide a brief outline of the proof for the Neumann case. Given a nonnegative function $u_0\in \LL^q(\Omega)$, $1\leq q\leq \infty$, we consider the sequence of approximations $\{u_{0,k}\}_{k\in\N}$ given by $u_{0,k}=\min\{u_0,k\}$. Denote by $u_k$ the solution of \eqref{eq} corresponding to the datum $u_{0,k}$. By Proposition \ref{weak.comp} $u_k(x,t)\geq 0$ \text{a.e.} $x\in \Omega$ and $t\in (0,T)$, and  $u_k(t)\in \LL^\infty(\Omega)$ for every $t\in [0,T]$.

By definition of solution  we have that $(u_k)_t(x,t)=\J_{p,\Omega}u_k(x,t)$ \text{a.e.} in $\Omega\times (0,T)$. Adding and subtracting $u_k^{p-1}(x,t)$,  and using  Lemma \ref{BC} it is obtained that
$$
-\frac{u_k(x,t)}{(p-2)t} \leq (u_k)_t(x,t) = -\mathsf{c}_{J,\Omega}(x) u_k^{p-1}(x,t) + \mathcal{I}(u_k)  \quad \text{ a.e. }x\in \Omega t\in (0,T)
$$
where  we have denoted
$$
\mathcal{I}(u_k):=\int_{\Omega} J(x-y)\left\{|u_k(y,t)-u_k(x,t)|^{p-2}(u_k(y,t)-u_k(x,t))+u_k^{p-1}(x,t)\right\} \,dy,
$$
and for $x\in \Omega$, $\mathsf{c}_{J,\Omega}(x)$ is given by $\mathsf{c}_{J,\Omega}(x):= \int_{\Omega} J(x-y)\,dy$. Then,  assumption \eqref{HJ} gives
$$
u_k^{p-1}(x,t) \leq  \left( \frac{u_k(x,t)}{(p-2)t}+\mathcal{I} (u_k) \right) \frac{1}{\kappa_{J,\Omega}}.
$$
Proceeding as in Steps 2 and 3 of the proof of Theorem \ref{main} we get
\begin{align} \label{step3'}
\|u_k(t)\|_{\LL^\infty(\Omega)}^{p-1} \leq
\frac{\mathsf{c}_\varepsilon t^{-\frac{p-1}{p-2}}  + \kappa_{\varepsilon_1}\kappa_{\varepsilon_2} \|u_k(t)\|_{\LL^q(\Omega)}^{p-1}}{\kappa_{J,\Omega}-\varepsilon-\varepsilon_1-\varepsilon_2 \kappa_{\varepsilon_1}} .
\end{align}
Choosing $\varepsilon=\varepsilon_1=\frac18 \kappa_{J,\Omega}$  and $\varepsilon_2=\frac{\kappa_{J,\Omega}}{4 \kappa_{\varepsilon_1}}$ gives $\kappa_{J,\Omega}-\varepsilon-\varepsilon_1-\varepsilon_2 \kappa_{\varepsilon_1} =\frac12 \kappa_{J,\Omega}$. Therefore,  \eqref{step3'}  and the  sub-additivity of the concave function $0\leq x\mapsto x^\frac{1}{p-1}$ yield
\begin{equation} \label{stepp'}
\|u_k(t)\|_{\LL^\infty(\Omega)}  \leq
 \frac{\mathsf{\tilde K}_p}{\kappa_{J,\Omega}} t^{-\frac{1}{p-2}}  + \frac{\mathsf{K}_{p,q,J}}{\kappa_{J,\Omega}} \|u_k(t)\|_{\LL^q(\Omega)}
\end{equation}
where $\mathsf{\tilde K}_p$ and $\mathsf{K}_{p,q,J}$ are the constants given  in \eqref{ctes}.
Then, as in Steps 6 and 7 of the proof of Theorem \ref{main}, we can take $k\to\infty$ to obtain that \eqref{stepp'} holds true for signed solutions. This completes the proof.
\end{proof}

\subsection{Alternative form of the smoothing estimate}
By using the time-monotonicity and time-scaling property of solutions, we provide for an alternative form of Theorem \ref{main}.

\begin{prop}
Let $u$ be a solution of \eqref{eqC} with $p>2$  corresponding to the initial datum $u_0\in \LL^q(\R^n)$ with $q\in [1, \infty]$ and let $\mathsf{\tilde K}_p$ and $\mathsf{K}_{p,q,J}$ be the explicit constants  given in \eqref{ctes}.  Then
\begin{align*}
\|u(t)\|_{\LL^\infty(\R^n)}\leq
\begin{cases}
2 \mathsf{\tilde K}_p t^{-\frac{1}{p-2}} & \text{ if } 0<t\leq t_*,\\
2 \mathsf{K}_{p,q,J} \|u_0\|_{\LL^q(\R^n)} & \text{ if } t>t_*,
\end{cases}
\end{align*}
where
$$
t_*=\left(\frac{\mathsf{K}_{p,q,J}}{\mathsf{\tilde K}_p}\|u_0\|_{\LL^q(\R^n)} \right)^{2-p}.
$$
Similarly, when $\Omega\subset \R^n$ is an open  bounded set and $u_0\in \LL^q(\R^n)$ with $q\in [1,  \infty]$, the same holds for solutions of \eqref{eqD} by replacing $\R^n$ with $\Omega$. When $\Omega$ also satisfies \eqref{HJ}, the result holds for solutions of \eqref{eq} up to the multiplicative constant $\kappa_{J,\Omega}^{-1}$.

\end{prop}

\begin{proof}
Let $u$ be solution of \eqref{eqC} corresponding to the initial datum $u_0\in \LL^q(\R^n)$. Given $\lambda>0$, consider the scaled function $u_\lambda(x,t):=\lambda^\frac{1}{p-2}u(x,\lambda t)$. Then, Theorem \ref{main} together with the time-scaling property for solutions
and Proposition \ref{norm.decreasing.difference} gives
$$
\|u_\lambda(\tau)\|_{\LL^\infty(\R^n)} \leq
\mathsf{\tilde K}_p   \tau^{-\frac{1}{p-2}}   + \mathsf{K}_{p,q,J}\|u_\lambda(0)\|_{\LL^q(\R^n)}  \qquad \forall \tau>0,
$$
that is,
$$
\lambda^\frac{1}{p-2}\|u (\lambda \tau)\|_{\LL^\infty(\R^n)} \leq \lambda^\frac{1}{p-2}
\mathsf{\tilde K}_p   (\lambda \tau)^{-\frac{1}{p-2}}   + \lambda^\frac{1}{p-2}\mathsf{K}_{p,q,J}\|u_0\|_{\LL^q(\R^n)}  \qquad \forall \tau>0.
$$
We can optimize by choosing
$$
(\lambda \tau)^{-\frac{1}{p-2}} = \frac{\mathsf{K}_{p,q,J}}{\mathsf{\tilde K}_p}\|u_0\|_{\LL^q(\R^n)} \quad \text{i.e. } \quad \lambda \tau = \left(\frac{\mathsf{K}_{p,q,J}}{\mathsf{\tilde K}_p}\|u_0\|_{\LL^q(\R^n)} \right)^{2-p}:=t_*,
$$
to obtain $
\|u(t_*)\|_{\LL^\infty(\R^n)} \leq 2 \mathsf{K}_{p,q,J} \|u_0\|_{\LL^q(\R^n)}$.
Using the time-monotonicity given in Lemma \ref{time.mono} we have that $u(x,t)\leq \left(\tfrac{t_*}{t}\right)^\frac{1}{p-2}u(x,t_*)$ holds \text{a.e. }$x\in \R^n$ and $0<t\leq t_*$, from where we are lead to
\begin{align*}
\|u(t)\|_{\LL^\infty(\R^n)}&\leq \left(\frac{t_*}{t}\right)^\frac{1}{p-2} \|u(t_*)\|_{\LL^\infty(\R^n)}\leq \left(\frac{t_*}{t}\right)^\frac{1}{p-2}
2 \mathsf{K}_{p,q,J} \|u_0\|_{\LL^q(\R^n)}
\leq
2 \mathsf{\tilde K}_p t^{-\frac{1}{p-2}}.
\end{align*}
For $t> t_*$,  using Proposition \ref{norm.decreasing.difference} we get $\|u(t)\|_{\LL^\infty(\R^n)}\leq \|u(t_*)\|_{\LL^\infty(\R^n)} \leq 2 \mathsf{K}_{p,q,J} \|u_0\|_{\LL^q(\R^n)}$.

This concludes the proof for the Cauchy problem.
The Neumann and  homogeneous Dirichlet cases follow analogously by using Theorem \ref{main2} instead of Theorem \ref{main}.
\end{proof}

\subsection{Smoothing for $u_t$}
We prove now that the smoothing effect is also valid for the time derivative of solutions.

\begin{thm}[Smoothing for $u_t$] \label{smooth.ut}
Let  $u$ be a solution of \eqref{eqC} with $p\in (1,\infty)$ corresponding to the initial datum $u_0\in \LL^1_{loc}(\R^n)$. The following holds:

\begin{enumerate}[leftmargin=15pt, label=(\roman*)]\itemsep2pt \parskip3pt \parsep2pt
\rm \item For any $p\in (1,\infty)$ and  $u_0\in \LL^\infty(\R^n)$, then $u_t(t)\in \LL^\infty(\R^n)$ for any $t\geq t_0\geq 0$ and
\begin{align*}
\|u_t(t)\|_{\LL^\infty(\R^n)} \leq
2^{p} \|u(t_0)\|_{\LL^\infty(\R^n)}^{p-1}.
\end{align*}

\rm \item For any $p\in (2,\infty)$ and $u_0\in \LL^q(\R^n)$ with $q\in [1,\infty]$, then $u_t(t) \in \LL^r (\R^n)$ with $r\in [1,\infty]$ for any $t\geq t_0>0$ and
$$
\|u_t(t)\|_{\LL^r(\R^n)} \leq  2^p
 \|u(t_0)\|_{\LL^q(\R^n)}^\frac{q}{r} \|u(t_0)\|_{\LL^\infty(\R^n)}^{p-1-\frac{q}{r}}.
$$
In this expression we can take $t_0=0$ when $u_0\in \LL^\infty(\R^n)$.

\end{enumerate}

When $\Omega\subset \R^n$ is open and bounded the same estimates hold for the solution of \eqref{eqD} by replacing $\R^n$ with $\Omega$. If $\Omega$ also fulfills \eqref{HJ}, then the same  holds for problem \eqref{eq}, and additionally (ii) reads as
$$
\|u_t(t)\|_{\LL^r(\Omega)}
\leq 2^p |\Omega|^\frac1r  \|u(t_0)\|_{\LL^\infty(\R^n)}^{p-1},
$$
for any $p\in (2,\infty)$ and $u_0\in \LL^q(\Omega)$ with $q\in [1,\infty)$ for any $t\geq t_0>0$. In this expression we can take $t_0=0$ when $u_0\in \LL^\infty(\Omega)$.
\end{thm}

Due to the smoothing effect established in Theorems \ref{main} and \ref{main2} we have obtained that solutions are bounded. Indeed, if $u$ is solution  of \eqref{eqC} corresponding to the initial datum $u_0$ we have that
\begin{align}  \label{bound.M0}
\begin{split}
\|u(t)\|_{\LL^\infty(\R^n)}   \leq \mathsf{\bar M}_0:=
\begin{cases}
\|u(t_0)\|_{\LL^\infty(\R^n)}  &\text{ for $t\geq t_0>0$, and } u_0\in \LL^q(\R^n), q\in [1,\infty], p\in (2,\infty)\\
\|u_0\|_{\LL^\infty(\R^n)} &\text{ for $t\geq t_0\geq 0$, and } u_0\in \LL^\infty(\R^n), p\in (1,\infty).
\end{cases}
\end{split}
\end{align}
where in the first case
$$
\|u(t_0)\|_{\LL^\infty(\R^n)}  \leq  \mathsf{\tilde K}_p t_0^{-\frac{1}{p-2}}   + \mathsf{K}_{p,q,J} \|u_0\|_{\LL^q(\R^n)},
$$
Moreover, Theorem \ref{smooth.ut} proves  that the boundedness of the $u_0$ implies the boundedness of $u_t(t)$:
\begin{align}  \label{bound.M1}
\begin{split}
\|u_t(t)\|_{\LL^\infty(\R^n)}   \leq \mathsf{\bar M}_1:= 2^p \mathsf{\bar M}_0^{p-1} \text{ where }
\begin{cases}
\text{ $t\geq t_0>0$, and } u_0\in \LL^q(\R^n), q\in [1,\infty), p\in (2,\infty)\\
\text{ $t\geq t_0\geq 0$, and } u_0\in \LL^\infty(\R^n), p\in (1,\infty).
\end{cases}
\end{split}
\end{align}
Analogous estimates hold for solution of \eqref{eqD} and \eqref{eq} by replacing $\R^n$ with $\Omega$.

In the following result we extend Theorem \ref{smooth.ut}  for higher order time derivatives.

\begin{thm} \label{smooth.ut.high}
For any $p\in (2,\infty)$ and $u_0\in \LL^q(\R^n)$ with $q\in [1,\infty]$, then $\partial_t^k u(\cdot, t) \in \LL^\infty (\R^n)$ for all $k=0,\ldots, [p]$, and all $t\geq t_0>0$. More precisely, there exists $\mathsf{M}_p>0$ depending only on $n, p$ and $\mathsf{\bar M}_1$ such that
\begin{equation}\label{der.k.t.infty}
\|\partial_t^k u(t)\|_{\LL^\infty(\R^n)} \leq \mathsf{M}_p.
\end{equation}
In this expression we can take $p\geq 2$ and $t_0=0$ when $u_0\in \LL^\infty(\R^n)$.
 
\noindent When $p\in \N$ there exists a $\mathsf{\tilde M_p}>0$  independent of $k$,  such that  for all $k\in \N_0$ and $t\geq t_0>0$ we have
$$
\| \partial_t^{k} u(t)\|_{\LL^\infty(\R^n)}  \leq  \mathsf{\tilde M}_p\,k!\,.
$$
As a consequence, $u(x,\cdot)$ is real analytic in time, with analyticity radius $r_0= 1\wedge t_0$. 

When $\Omega\subset \R^n$ is open and bounded the same estimates hold for the solution of \eqref{eqD} by replacing $\R^n$ with $\Omega$. If $\Omega$ also fulfills \eqref{HJ}, then the same results hold for problem \eqref{eq}.
\end{thm}

\begin{proof}[Proof of Theorem \ref{smooth.ut}]

We split the proof into several steps. The proof of the Cauchy case is the same as for the Dirichlet problem (by using the zero-extension of $u$ outside the domain), hence we omit it.

\noindent$\circ~$\textsc{Step 1}. \textit{The case $p\in (1,2]$. }Let $u$ be the solution of \eqref{eq} with $p\in (1,2]$, $\Omega\subseteq \R^n$ corresponding to the initial datum $u_0\in \LL^\infty(\Omega)$.  For \text{a.e.} $x\in \Omega$ and $t\geq t_0\geq 0$, we have that
\begin{align*}
|u_t(x,t)|&= \left| \int_{\R^n} J(x-y) L_p(u(y,t)-u(x,t)) \,dy \right|\leq   2^{p} \|u(t)\|_{\LL^\infty(\R^n)}^{p-1}  \int_{\R^n} J(x-y)\,dy \leq  2^{p} \|u(t_0)\|_{\LL^\infty(\R^n)}^{p-1},
\end{align*}
which gives item $(i)$ both for the Cauchy and Neumann case.

\noindent$\circ~$\textsc{Step 2}. \textit{The Cauchy problem when $q,r\in [1,\infty)$. }Let $u$ be a solution of \eqref{eqC} corresponding to the initial datum $u_0\in \LL^q(\R^n)$ with $q\in [1,\infty)$. Then, for any $t\geq t_0>0$ and \text{a.e.} $x\in \R^n$,
\begin{align*}
|u_t(x,t)|&= \left| \int_{\R^n} J(x-y) L_p(u(y,t)-u(x,t))\,dy \right|\leq   2^{p-1} \|u(t)\|_{\LL^\infty(\R^n)}^{p-2}  \int_{\R^n} J(x-y)|u(x)-u(y)|\,dy .
\end{align*}
Hence,   we obtain that for any $r\in [1,\infty)$
\begin{align*}
\int_{\R^n} |u_t(x,t)|^r \,dx \leq 2^{r(p-1)} \|u(t)\|_{\LL^\infty(\R^n)}^{r(p-2)} \int_{\R^n} \left( \int_{\R^n} J(x-y)|u(x,t)-u(y,t)|\,dy \right)^r \,dx.
\end{align*}
Since $\|J\|_{\LL^1(\R^n)}=1$, by Jensen's inequality we have that
\begin{align*}
\Big( \int_{\R^n}  J(x-y)&|u(x,t)-u(y,t)|\,dy \Big)^r  \leq
\int_{\R^n} J(x-y)|u(x,t)-u(y,t)|^r\,dy\\
&\leq
2^{r-1}|u(x,t)|^r \int_{\R^n} J(x-y)\,dy  + 2^{r-1}\int_{\R^n} J(x-y) |u(y,t)|^r\,dy\\
&\leq
2^{r-1}|u(x,t)|^r    + 2^{r-1}\int_{\R^n} J(x-y) |u(y,t)|^r\,dy.
\end{align*}
Then, using Jensen's inequality, Fubini's Theorem and the fact that  $\|J\|_{\LL^1(\R^n)}=1$,
\begin{align*}
\int_{\R^n}&\Big( \int_{\R^n}  J(x-y)|u(x,t)-u(y,t)|\,dy \Big)^r \,dx \\
&\leq 2^{r-1}\int_{\R^n}|u(x,t)|^r \,dx   + 2^{r-1}\int_{\R^n}\int_{\R^n} J(x-y) |u(y,t)|^r\,dy\,dx\\
&\leq 2^{r-1}\|u(t)\|_{\LL^r(\R^n)}^r  + 2^{r-1}\int_{\R^n}|u(y,t)|^r \left(\int_{\R^n} J(x-y) \,dx\right)dy=  2^r \|u(t)\|_{\LL^r(\R^n)}^r.
\end{align*}
Moreover, observe that, for $q\in [1,\infty)$ it holds that
$$
\int_{\R^n} |u(x,t)|^r \,dx = \int_{\R^n} |u(x,t)|^{r-q}|u(x,t)|^{q} \,dx  \leq \|u(t)\|_{\LL^\infty(\R^n)}^{r-q} \|u(t)\|_{\LL^q(\R^n)}^q.
$$
Gathering the last inequalities and using Proposition \ref{norm.decreasing.difference} yield that, for  $q,r\in [1,\infty)$
\begin{align*} 
\|u_t(t)\|_{\LL^r(\R^n)} \leq 2^p
\|u(t_0)\|_{\LL^\infty(\R^n)}^{p-1-\frac{q}{r}} \|u(t_0)\|_{\LL^q(\R^n)}^\frac{q}{r}.
\end{align*}

\noindent\textit{Dirichlet case. }In this case, since $u$ is the zero extension outside $\Omega$, clearly the $\LL^q$ norms of both $u$ and $u_t$ on the whole space coincide with the same norms on $\Omega$.

\medskip

\noindent$\circ~$\textsc{Step 3}. \textit{The Neumann problem when $r\in [1,\infty)$ and $q\in [1,\infty]$. }When $u$ is a solution of \eqref{eq} corresponding to $u_0 \in \LL^q(\Omega)$ with $q\in [1,\infty)$, $\Omega\subset \R^n$ is a bounded domain, and $J$ satisfies \eqref{HJ}, as in Step 1, we get that for any $0<t_0\leq t$ and \text{a.e.} $x\in \Omega$,
\begin{align*}
\int_{\Omega} |u_t(x,t)|^r \,dx \leq 2^{r(p-1)} \|u(t)\|_{\LL^\infty(\R^n)}^{r(p-2)} \int_{\Omega} \left( \int_{\Omega} J(x-y)|u(x,t)-u(y,t)|\,dy \right)^r dx.
\end{align*}
In this case, the boundedness of $\Omega$ allows to  estimate of the right-hand side as follows
$$
\left( \int_\Omega  J(x-y)|u(x,t)-u(y,t)| \,dy \right)^r \leq 2^r \|J\|_{\LL^1(\R^n)}^r \|u(t)\|_{\LL^\infty(\Omega)}^r.
$$
Using that $\|J\|_{\LL^1(\R^n)}=1$ and Proposition \ref{norm.decreasing.difference}, for $q\in [1,\infty)$ and $r \in [1,\infty)$ we get
\begin{align} \label{cota.acotado}
\|u_t(t)\|_{\LL^r(\Omega)}  \leq 2^p |\Omega|^\frac{1}{r}
\|u(t_0)\|_{\LL^\infty(\Omega)}^{p-1}.
\end{align}
When $u_0\in \LL^\infty(\Omega)$ we can take $t_0=0$ in \eqref{cota.acotado}.

\medskip

\noindent$\circ~$\textsc{Step 4}. \textit{The case $r=\infty$. }Let us analyze now this case for all problems at once. Given $\Omega\subseteq\R^n$, for \text{a.e.} $x\in \Omega$ and $t\geq t_0>0$, we have that
\begin{align*}
|u_t(x,t)|&= \left| \int_{\Omega} J(x-y) L_p(u(y,t)-u(x,t))\,dy \right|\leq   2^{p} \|u(t)\|_{\LL^\infty(\Omega)}^{p-1}  \int_{\R^n} J(x-y)\,dy  \leq  2^{p} \|u(t)\|_{\LL^\infty(\Omega)}^{p-1},
\end{align*}
and then, using Proposition \ref{norm.decreasing.difference} we get
\begin{align*} 
\|u_t(t)\|_{\LL^\infty(\Omega)} \leq
\begin{cases}
2^{p} \|u(t_0)\|_{\LL^\infty(\Omega)}^{p-1} &\text{ when } u_0\in \LL^q(\Omega) \text{ with } q\in [1,\infty)\\
2^{p} \|u_0\|_{\LL^\infty(\Omega)}^{p-1} &\text{ when } u_0\in \LL^\infty(\Omega).
\end{cases}
\end{align*}
This concludes the proof.
\end{proof}

\begin{proof}[Proof of Theorem \ref{smooth.ut.high}]

We prove the result for the Cauchy problem. The proofs for the Dirichlet and Neumann case are analogous.

Let $u$ be a solution of \eqref{eqC} corresponding to the initial datum $u_0\in \LL^q(\R^n)$ with $q\in [1,\infty]$ and let $t\geq t_0>0$ ($t_0=0$ when $u_0\in \LL^\infty(\R^n)$).

Observe that the case $k=0$ is just \eqref{bound.M0}. We prove the result by induction on $k\geq 1$.

\noindent \textsc{Step $k=1$}. This case is given in \eqref{bound.M1}.

\noindent \textsc{Inductive step. } Assume that the estimate
$$
\| \partial_t^j u(t) \|_{\LL^\infty(\R^n)} \leq \mathsf{M}_p \quad
\text{for all } j=0\ldots [p]-1
$$
holds for $t\geq t_0>0$ ($t_0=0$ when $u_0\in \LL^\infty(\R^n)$) for some constant $\mathsf{M}_p$ depending only on $n, p$ and $\mathsf{\bar M}_1$.

Let us see that the $\partial_t^{[p]} u(t)$ is bounded in terms of the constant $\mathsf{M}_p$. Indeed, for $x\in \R^n$ and $t\geq t_0$
\begin{align*}
\partial_t^{[p]} u(x,t) &=
\lim_{h\to 0^+} \frac{\partial_t^{[p]-1} u(x,t+h) -   \partial_t^{[p]-1}u(x,t)}{h}\\
&=
\lim_{h\to 0^+} \partial_t^{[p]-2} \left(\frac{ u_t(x,t+h) -   u_t(x,t)}{h} \right)\\
&=
\lim_{h\to 0^+} \partial_t^{[p]-2} \left(\frac{1}{h} \int_{\R^n} J(x-y) (L_p(u(y,t+h)-u(x,t+h))- L_p(u(y,t)-u(x,t))) \,dy \right)\\
&=
\lim_{h\to 0^+}  \left(\frac{1}{h} \int_{\R^n} J(x-y) \partial_t^{[p]-2}(A_p(x,y,t+h)- A_p(x,y,t)) \,dy \right)
\end{align*}
where we have used the equation and dominated convergence, and denoted $$A_p(x,y,t):=L_p(u(y,t)-u(x,t)).$$
Observe that there is $t^* \in (t,t+h)$ such that $
A_p(x,y,t+h)- A_p(x,y,t) = \partial_t (A_p(x,y, t^*))h.
$
Hence
\begin{equation} \label{bound.k1}
\partial_t^{[p]} u(x,t)  =
 \int_{\R^n} J(x-y) \partial_t^{[p]-1} A_p(x,y,t^*) \,dy.
\end{equation}
The $[p]-$th time derivative of $A_p(x,y,t)$ can be written by using the \emph{Fa\'a di Bruno's formula} (see for instance \cite{CS96}) as
\begin{equation} \label{faa.formula}
\partial_t^{[p]-1} A_p(x,y,t) = \partial_t^{[p]-1} (L_p(g(t))) = \sum\frac{([p]-1)!}{m_1! m_2! \cdots m_{[p]}!} L_p^{(m_1+\cdots+m_{[p]})} g(t)  \prod_{j=1}^{[p]-1} \left( \frac{g^{(j)}(t)}{j!}\right)^{m_j}
\end{equation}
where the sum is over all $n-$uples of nonnegative integers $(m_1,\ldots, m_k)$ satisfying the constraint
$$
m_1 + 2m_2 + 3m_3 + \cdots + k m_k=[p]-1.
$$
Here $g\colon \R \to \R$ denotes	 the function $g(t)=u(y,t)-u(x,t)$, $x,y\in \R^n$.

An inspection of the Fa\'a di Bruno's formula reveals  that for all $x,y\in \R^n$ and $t\geq t_0$ it holds that
\begin{align} \label{faaa}
|\partial_t^{[p]-1} A_p(x,y,t)| & \leq ([p]-1)!
\sum_{j=1}^{[p]-1}\|L_p^{(j)}g(t)\|_{\LL^\infty(\R^n)} \left( \prod_{\ell=1}^{[p]-1}  \sum_{j=1}^{[p]-1} \|g^{(j)}(t)\|^\ell_{\LL^\infty(\R^n)} \right).
\end{align}
Since $|L_p^{(j)}(z)| = (p-1)\cdots (p-j)|z|^{p-1-j}$, the inductive hypothesis gives   for any $0\leq j \leq [p]-1$
\begin{align*}
\|L_p^{(j)}g(t)\|_{\LL^\infty(\R^n)}
&\leq (p-1)^{[p]-1} \max\{ \|2u(t)\|_{\LL^\infty(\R^n)}^{p-1},\|2u(t)\|_{\LL^\infty(\R^n)}^{p-[p]} \}\\
&\leq (p-1)^{[p]-1} \max\{ (2\mathsf{M}_p)^{p-1},(2\mathsf{M}_p)^{p-[p]} \}.
\end{align*}
Moreover, due to the  inductive hypothesis, for all $0\leq j \leq  [p]-1$ it holds that $\|g^{(j)}(t)\|_{\LL^\infty(\R^n)}$ is bounded by $2 \|\partial_t^j u(t)\|_{\LL^\infty(\R^n)}\leq 2 \mathsf{M}_p$. Then, in light of  \eqref{faaa} these expressions yield for all $x,y\in \R^n$ and $t\geq t_0$
\begin{align} \label{cota.Ap}
|\partial_t^{[p]-1} A_p(x,y,t)| \leq  \mathsf{\tilde M}_p
\end{align}
where $\mathsf{\tilde M}_p$ depends only on $n$, $p$ and $\mathsf{M}_p$. Inserting \eqref{cota.Ap} into \eqref{bound.k1} finally gives
$$
\| \partial_t^{[p]} u(t)\|_{\LL^\infty(\R^n)}  \leq \mathsf{\tilde M}_p
 \int_{\R^n} J(x-y) \,dy \leq \mathsf{\tilde M}_p.
$$
This concludes the proof when $p$ is not an integer.

\noindent\textbf{Analyticity. }When $p\in \N$, we obtain that $|L_p^{(j)}g|=|g^{(j)}|=0$ for any $j>p-1$, which in light of the previous computations gives that $\partial_t^{j}u(t)=0$, hence formula \eqref{faaa} gives, for all $k\in \N$
\begin{align*} 
|\partial_t^{k} A_p(x,y,t)| & \leq k!\,
\sum_{j=1}^{k}\|L_p^{(j)}g(t)\|_{\LL^\infty(\R^n)} \left( \prod_{\ell=1}^{k}  \sum_{j=1}^{k} \|g^{(j)}(t)\|^\ell_{\LL^\infty(\R^n)} \right)\\
&=  k!\,
\sum_{j=1}^{[p]-1}\|L_p^{(j)}g(t)\|_{\LL^\infty(\R^n)} \left( \prod_{\ell=1}^{[p]-1}  \sum_{j=1}^{[p]-1} \|g^{(j)}(t)\|^\ell_{\LL^\infty(\R^n)} \right)
\le \mathsf{\tilde M}_p \, k!\,.
\end{align*}
As a consequence we obtain that for all $k\in \N$
\begin{equation} \label{cota.para.todo.k}
\| \partial_t^{k} u(t)\|_{\LL^\infty(\R^n)}  \leq \mathsf{\tilde M}_p\, k!\,
 \int_{\R^n} J(x-y) \,dy \leq \mathsf{\tilde M}_p\,k!\,,
\end{equation}
which clearly implies that $u(t)$ is analytic with radius $r_0= 1\wedge t_0$. This concludes the proof.
\end{proof}

\section{Higher Regularity of solutions} \label{section6}
In this section, we study regularity properties of bounded (in space) solutions. Roughly speaking, solutions become $C_t^{[p], p-[p]}$ in time, and  the modulus of continuity of the initial data and its derivatives is preserved under certain conditions.

We start with two lemmas describing the behavior of the function $L_p$ and its derivatives.

\begin{lema} \label{lema.desig.num}
Let $a,b\in \R$.  For $1<p\leq 2$ there exists a constant $\mathsf{c}_p$ such that
$$
|L_p(a)-L_p(b)|\leq \mathsf{c}_p |a-b|^{p-1}.
$$
For any $p\geq 2$  it holds that
\begin{align*}
|L_p(a)-L_p(b)|&\leq 2^{p-2} (p-1)|a-b|(|a|+|b|)^{p-2}\\
|M_p(a)-M_p(b)|&\leq
\begin{cases}
 (p-1)|a-b|(|a|^{p-3}+|b|^{p-3}) & \text{ when } p\geq 3\\
 |a-b|^{p-2} & \text{ when } 2\leq p <3.
\end{cases}
\end{align*}

\end{lema}
\begin{proof}
The inequality for $p\in (1,2)$ can be found in \cite{D93}.  Given $a,b\in \R$ and $p\geq 2$ we have that
\begin{align*}
|b|^{p-2}-|a|^{p-2}a &= \int_0^1 \frac{d}{dt}|a+t(b-a)|^{p-2}(a+t(b-a))\,dt= (p-1)(b-a)\int_0^1 |a+t(b-a)|^{p-2}\,dt,
\end{align*}
from where it follows that
$$
||b|^{p-2}-|a|^{p-2}a|\leq (p-1)|b-a| (|a|+|b-a|)^{p-2} \leq 2^{p-2}(p-1) |b-a|(|a|+|b|)^{p-2}.
$$
Similarly, when $p\geq 3$
\begin{align*}
||b|^{p-2}-|a|^{p-2}| &= \left|\int_0^1 \frac{d}{dt}|a+t(b-a)|^{p-2}\,dt \right|= (p-2)|b-a| \int_0^1 |a+t(b-a)|^{p-3}\,dt \\
&\leq (p-2)|b-a| 2^{p-3} (|a|+|b)^{p-3}.
\end{align*}
When $p\leq 3$, since $p-2\leq 1$ it holds that
$$
||b|^{p-2}-|a|^{p-2}| \leq ||b|-|a||^{p-2} \leq |b-a|^{p-2}.
$$
This concludes the proof.
\end{proof}

\begin{lema} \label{lemap12}
Let $a,b\in \R$ such that $b>a$. For $1<p\leq 2$ it holds that
$$
L_p(b)-L_p(a) \geq (p-1) \frac{b-a}{(1+|a|^2+|b|^2)^\frac{2-p}{2}}
$$
\end{lema}
\begin{proof}
It is given in page 75 of \cite{Lind}.
\end{proof}

\begin{lema} \label{lema.Lp.gral}
Let $p\geq 2$ and $r\in \N$ such that $r<p-1$. Denote $\mathsf{c}_{r,p}=(p-1)(p-2)\cdots (p-r)$ and  let $a,b\in \R$. When $r$ is even
\begin{align*}
\left|L_p^{(r)}(a) -L_p^{(r)}(b) \right| \leq
\mathsf{\bar c}_{r,p}  |a-b|(|a|+|b|)^{p-r-2} &\text{ if } p-r\geq  2.
\end{align*}
where $\mathsf{\bar c}_{r,p}:=\mathsf{c}_{r,p} (p-r-1) 2^{p-r-2} $,  and when $r$ is odd
\begin{align*}
\left| L_p^{(r)}(a) - L_p^{(r)}(b) \right| \leq
\begin{cases}
 \mathsf{\bar c}_{r,p}|a-b|(|a|^{p-r-2}+ |b|^{p-r-2}) &\text{ if }p-r \geq 2,\\
\mathsf{\bar c}_{r,p} |a-b|^{p-r-1} &\text{ if }1\leq p-r<2
\end{cases}
\end{align*}
where $\mathsf{\bar c}_{r,p}:=\mathsf{c}_{r,p}(p-r+1)$ when $p-r\geq 2$ and $\mathsf{\bar c}_{r,p}:=\mathsf{c}_{r,p}$ when $1\leq p-r<2$.

When $p\in \N$ and $r\ge p-1$ then $\left|L_p^{(r)}(a) -L_p^{(r)}(b) \right|=0$ for any $a,b\in \R$.
\end{lema}

\begin{proof}
An easy computation gives that for $r<p-1$
\begin{align} \label{derivada}
\begin{split}
\frac{d^r}{dt^r} L_p(t):=L_p^{(r)}(t)=
\begin{cases}
\mathsf{c}_{r,p} L_{p-r}(t) & \text{ if $r$ is even}\\
\mathsf{c}_{r,p} M_{p-r+1}(t) & \text{ if $r$ is odd.}
\end{cases}
\end{split}
\end{align}
This together with Lemma \ref{lema.desig.num} gives the lemma.
\end{proof}

\subsection{Higher regularity in time  for all $p\in (1,\infty)$}

Next, we prove higher regularity in time, knowing by the smoothing effects of the previous section that $u_t$ is already bounded.

\begin{thm}[H\"older regularity in time for $u_t$ when $1<p\le 2$] \label{thm.time.reg.p.leq.2}

Let $u$ be the solution of \eqref{eqC} with $p\in (1,2]$ corresponding to the initial datum $u_0\in \LL^\infty(\R^n)$. Then  $u_t(x,\cdot) \in C^{0,p-1}_t([t_0,\infty))$ and there exists $\mathsf{c}_p>0$ depending   on $n$ and $p$ such that for all $x\in \R^n$ and all $0\leq t_0 \leq t_1 <t_2<\infty$ such that
\begin{align*} 
\begin{split}
\frac{|u_t(x,t_2)-u_t(x,t_1)|}{|t_1-t_2|^{p-1}}  &\leq \mathsf{c}_p \|u(t_0)\|_{\LL^\infty(\R^n)}^{(p-1)^2}.
\end{split}
\end{align*}

Analogous results hold for solutions of \eqref{eqD} when $\Omega\subset \R^n$ is open and bounded, and for solutions of \eqref{eq} when $\Omega$ in addition  satisfies \eqref{HJ}.
\end{thm}

\begin{proof}
We prove the result for the Cauchy problem. The proofs in the Neumann and the Dirichlet cases are analogous.

Consider the solution  $u$ of \eqref{eqC} with $p>1$ corresponding to $u_0\in \LL^\infty(\R^n)$. Using Lemma \ref{lema.desig.num}, for fixed $t_2>t_1\geq t_0\geq  0$ and $x\in \R^n$ we have that
\begin{align*}
|u_t(x,t_2)-u_t(x,t_1)| &\leq \int_{\R^n} J(x-z)\big|L_p (u(z,t_2)-u(x,t_2))
-L_p (u(z,t_1)-u(x,t_1)) \big|\,dz  \\
&\leq \mathsf{c}_p\int_{\R^n} J(x-z) | u(z,t_2)-u(z,t_1) -(u(x,t_2)-u(x,t_1))|^{p-1} \,dz.
\end{align*}
Using triangular inequality and mean value theorem, for some $t^*\in (t_1,t_2)$ we have that
\begin{align*}
|u(z,t_2)-u(z,t_1)-(u(x,t_2)-u(x,t_1))|^{p-1} & \leq 2^{p-2}  |t_2-t_1|^{p-1} \left(  |u_t(z,t^*)|^{p-1}  +  |u_t(x,t^*)|^{p-1}  \right) \\
&\leq 2^{(p-1)(2p-1)} |t_2-t_1|^{p-1} \|u(t_0)\|_{\LL^\infty(\R^n)}^{(p-1)^2}
\end{align*}
where in the last inequality we used Theorem \ref{smooth.ut}. This gives that $u_t(x,\cdot) \in C^{0,p-1}([t_0,\infty))$ and
\begin{equation*}
\frac{|u_t(x,t_2)-u_t(x,t_1)|}{|t_2-t_1|^{p-1}} \leq \mathsf{c}_p 2^{(p-1)(2p-1)}    \|u(t_0)\|_{\LL^\infty(\R^n)}^{(p-1)^2}.
\end{equation*}
This concludes the proof
\end{proof}

In the following result we bootstrap the time regularity of solutions to get higher regularity in time.

\begin{thm}[Higher regularity in time] \label{prop.hi.ut}
Let $u$ be solution of \eqref{eqC} with $p\in (2,\infty)$ corresponding to the initial datum $u_0 \in \LL^q(\R^n)$ with $q\in [1,\infty]$. Then, $u(\cdot,t) \in C_t^{[p], p-[p]}([t_0,\infty))$ for all $t\geq t_0>0$. More precisely,  there exists a constant $\mathsf{c}_p>0$ depending only on $ n, p$ and the constant $\mathsf{\bar M}_1$ given in \eqref{bound.M1} such that for all $x\in \R^n$ and all $t_0\leq t_1<t_2<\infty$ we have
$$
\max_{k=0,\dots,[p]-1}\frac{|\partial_t^{k} u(x,t_1) - \partial_t^{k} u(x,t_2) |}{|t_1-t_2|}
+ \frac{|\partial_t^{[p]} u(x,t_1) - \partial_t^{[p]} u(x,t_2) |}{|t_1-t_2|^{p-[p]}} \leq \mathsf{c}_p.
$$
When $u_0\in \LL^\infty(\R^n)$, the above result holds for all $p>1$ and we can also allow $t_0=0$.

When $p\in \N$ we have that $u(\cdot, t)\in C_t^\infty([t_0,\infty))$ for all $t\geq t_0>0$. Moreover, there exists $\mathsf{c}_p>0$ independent of $k$ such that for all $x\in \R^n$ and all $t_0\leq t_1<t_2<\infty$ we have that
$$
\frac{|\partial_t^{k} u(x,t_1) - \partial_t^{k} u(x,t_2) |}{|t_1-t_2|} \leq \mathsf{c}_p \, k!\,
$$
holds for all $k\in \N$. In particular, $u(\cdot,t)$ is analytic with radius of convergence $r_0= 1\wedge t_0$.

When $\Omega$ is open and bounded, analogous results hold for solution of \eqref{eqD}. If in addition $\Omega$ satisfies \eqref{HJ}, the same holds for solutions of \eqref{eq}.
\end{thm}

\begin{rem}
Some comments on the higher regularity in time result:
\begin{enumerate}[leftmargin=15pt, label=(\roman*)]\itemsep2pt \parskip3pt \parsep2pt
\item[$\circ$] When $1<p\leq 2$, Theorem \ref{prop.hi.ut} recovers Theorem \ref{thm.time.reg.p.leq.2}.

\item[$\circ$] Theorem \ref{prop.hi.ut} in particular says that for all $x\in \R^n$
$$
\max_{k=0,\ldots, [p]-1}[\partial_t^k u(x,\cdot)]_{C^{0,1}([t_0,\infty))} + [\partial_t^{[p]} u(x,\cdot)]_{C^{0,p-[p]}([t_0,\infty))}\leq \mathsf{c}_p;
$$
and for all $t_1,t_2\geq t_0$, denoting $\omega(\rho)=\max\{ |\rho|, |\rho|^{p-[p]}\}$
$$
\max_{k=0,\ldots, [p]-1} \|\partial_t^k u(t_1)-\partial_t^k u(t_2)\|_{\LL^\infty(\R^n)} + \|\partial_t^{[p]}u(t_1) - \partial_t^{[p]}u(t_1) \|_{\LL^\infty(\R^n)} \leq \mathsf{c}_p \omega(t_1-t_2).
$$
\end{enumerate}
\end{rem}

\begin{proof}
Let $u$ be a solution of \eqref{eqC} corresponding to the initial datum $u_0\in \LL^q(\R^n)$ with $q\in [1,\infty]$ and let $x\in \R^n$ and $t_1,t_2\geq t_0>0$ ($t_0\geq 0$ when $u_0\in \LL^\infty(\R^n)$). Given $h>0$ and $k\in\N$, by using  \eqref{bound.k1} it is obtained that
\begin{align}\label{claim.xxx.0}
|\partial_t^{k} u(x,t_1) - \partial_t^{k} u(x,t_2) | \leq
 \int_{\R^n} |J(x-y)| |\partial_t^{k-1} A_p(x,y,t_1^*) - \partial_t^{k-1} A_p(x,y,t_2^*)|\,dy
\end{align}
where $t^*_1\in (t_1,t_1+h)$, $t^*_2\in (t_2,t_2+h)$ and $A_p(x,y,t)=L_p(u(y,t)-u(x,t))$.
\\
We observe that for any $k\leq [p]-1$ estimate \eqref{cota.Ap} gives
\begin{equation} \label{ecu.mp}
|\partial_t^k A_p(x,y,t)| \leq \mathsf{M}_p
\end{equation}
where $\mathsf{ M}_p$ is a positive constant depending on $n, p$ and $\mathsf{\bar M}_1$. Then, using the mean value theorem,  for any  $k\leq [p]-1$, we obtain the following estimate:
\begin{equation}\label{claim.xxx.1}
\begin{split}
|\partial_t^{k-1} A_p(x,y,t_1^*) - \partial_t^{k-1} A_p(x,y,t_2^*)| &\leq |t^*_1 - t^*_2| |\partial_t^{k} A_p(x,y,t^*)| \leq \mathsf{  M}_p |t^*_1-t^*_2|
\end{split}
\end{equation}
where $t^* \in (t^*_1,t^*_2)$. Moreover, since $|t^*_1-t^*_2|\leq |t_1 -t_2| + h$, taking $h\leq |t_1-t_2|$ and  gathering expressions \eqref{claim.xxx.0} and \eqref{claim.xxx.1} it is obtained that
\begin{align*}
|\partial_t^{k} u(x,t_1) - \partial_t^{k} u(x,t_2) | \leq 2\mathsf{M}_p |t_1-t_2|.
\end{align*}
When $k=[p]$ and  $p\not\in \N$ we will  show that $\partial_t u^{[p]}(x,\cdot) \in C^{0,p-[p]}([t_0,\infty))$  for each $x\in \R^n$ and $t\geq t_0$, and the following estimate holds
\begin{equation*}
|\partial_t^{[p]-1} A(x,y,t_1) - \partial_t^{[p]-1} A(x,y,t_2) | \leq \mathsf{c}(\mathsf{M}_{[p]}) |t_1-t_2|^{p-[p]}.
\end{equation*}
Then, using expression \eqref{claim.xxx.0} gives that
$$
|\partial_t^{[p]} u(x,t_1) - \partial_t^{[p]} u(x,t_2) | \leq \mathsf{\tilde c}(\mathsf{M}_{[p]}) |t_1-t_2|^{p-[p]}
$$
as desired.

First, we observe that from  \eqref{faa.formula} we can give an expression for $\partial_t^{[p]-1} A_p(x,y,t)$ as follows
\begin{align}
\partial_t^{[p]-1} A_p(x,y,t)= \sum_{k=1}^{[p]-1} L_p^{k}(w(t))v_k(t)
\end{align}
where we have denoted $w(t)=u(y,t)-u(x,t)$ and $v_k(t)$ are functions depending on $w_t(t), w_{tt}(t), \ldots ,\partial^{[p]-1}_t w(t)$ and powers of these functions.

Therefore, given $t_1,t_2\geq t_0>0$ ($t_0\geq 0$ when $u_0\in \LL^\infty(\R^n))$ we have
\begin{align} \label{dif.A}
\begin{split}
|\partial_t^{[p]-1} A_p(x,y,t_1)&- \partial_t^{[p]-1} A_p(x,y,t_2)|  \leq\\
&\leq \sum_{k=1}^{[p]-1} |L_p^{(k)} (w(t_1)) - L_p^{(k)} (w(t_2))| | v_{k}(t_1)| + |L_p^{(k)} (w(t_2)) |v_{k}(t_1)- v_{k}(t_2)|
\end{split}
\end{align}
In order to bound this expression we make some observations. We assume that $k$ is even (the proof is similar when $k$ is odd).  Using \eqref{derivada}, Theorem \ref{smooth.ut.high} and the fact that $v_k(t)$ is a Lipschitz function we get
\begin{align*}
|L_p^{(k)} (w(t_2)) |v_{k}(t_1)- v_{k}(t_2)| &\leq \mathsf{c}_p |L_{p-k}(w(t_2)| \mathsf{c}_1(\mathsf{M}_p) |t_1-t_2|\\
&\leq \mathsf{c}_2(\mathsf{M}_p) |t_1-t_2|.
\end{align*}
From  \ref{derivada} and Theorem \ref{smooth.ut.high} we have that for any $k\in \N$ even
\begin{align*}
|L_p^{(k)} (w(t_1)) - L_p^{(k)} (w(t_2))| &\leq  \mathsf{\bar c}_{k,p} |L_{p-k}(w(t_1))-L_{p-k}(w(t_2))|
\end{align*}
When $k=[p]-1$, using the convexity of $r\mapsto r^{[p]-p}$ (with $[p]-p<1$) and the fact that $u(t)$ is Lipschitz
\begin{align*}
|L_p^{([p]-1)} (w(t_1)) - L_p^{([p]-1)} (w(t_2))| &\leq  \mathsf{\bar c}_{p} |w(t_1)^{p-[p]}-w(t_2)^{p-[p]}| \\
&\leq \mathsf{\bar c}_{p}|w(t_1)-u(t_2)|^{p-[p]}\\
&\leq \mathsf{\bar c}_{p}|u(x,t_1)-u(x,t_2)|^{p-[p]} +\mathsf{\bar c}_{p} |u(y,t_1)-u(y,t_2)|^{p-[p]}\\
&\leq 2\mathsf{\bar c}_{p} \mathsf{M}_p |t_1-t_2|^{p-[p]}.
\end{align*}
Similarly, using Lemma \ref{lema.Lp.gral} we get for $k=1, \ldots [p]-2$
$$
|L_p^{([p]-1)} (w(t_1)) - L_p^{([p]-1)} (w(t_2))| \leq \mathsf{c}_3(\mathsf{M}_p) |t_1-t_2|.
$$
Since assuming $|t_1-t_2|\leq 2$ is no restrictive, these computations lead to bound \eqref{dif.A} as
\begin{align*}
|\partial_t^{[p]-1} A_p(x,y,t_1)- \partial_t^{[p]-1} A_p(x,y,t_2)| &\leq \mathsf{c}_4(\mathsf{M}_p) |t_1-t_2|^{p-[p]}  + \mathsf{c}_5(\mathsf{M}_p) |t_1-t_2|\\
&\leq
\mathsf{c}_6(\mathsf{M}_p) |t_1-t_2|^{p-[p]} .
\end{align*}

\noindent\textbf{Analyticity. }When $p\in \N$,  \eqref{cota.para.todo.k} yields that for all $k\in \N$ there exists $\mathsf{\tilde M}_p$ independent of $k$ such that
\begin{align*}
|\partial_t^{k} A_p(x,y,t)|
\le \mathsf{\tilde M}_p \, k!\,.
\end{align*}
Then, from \eqref{claim.xxx.1}, we get that for all $k\in \N$ it holds that
$$
|\partial_t^{k-1} A_p(x,y,t_1^*) - \partial_t^{k-1} A_p(x,y,t_2^*)| \leq \mathsf{  \tilde M}_p \, k!\, |t^*_1 - t^*_2|.
$$
As a consequence, from the previous computations we obtain that for all $k\in \N$
\begin{align*}
|\partial_t^{k} u(x,t_1) - \partial_t^{k} u(x,t_2) | \leq 2\mathsf{\tilde M}_p  \, k!\,|t_1-t_2|,
\end{align*}
which clearly implies that $u(\cdot, t)$ is analytic with radius of converngence $r_0= 1\wedge t_0$. \\ This concludes the proof of the Cauchy Problem.

The proof for the Dirichlet and Neumann problems is completely analogous. 
\end{proof}

\subsection{H\"older regularity in space} \label{sec.holder}

In this subsection we prove that the modulus of continuity of the initial data and its derivatives is preserved under certain conditions.

We recall some notation which will be used along this paragraph.

A \emph{modulus of continuity} is a function $\omega\colon [0,\infty] \to [0,\infty]$ vanishing at 0 and continuous at 0. A function $v$ admits $\omega$ as a modulus of continuity if and only if
$$
|v(x)-v(y)| \leq \omega(|x-y|)
$$
for any $x$ and $y$ in the domain of $v$.

Given two moduli of continuity $\omega_1$ and $\omega_2$ we say that

\begin{enumerate}[leftmargin=15pt, label=(\roman*)]\itemsep2pt \parskip3pt \parsep2pt
\item[] $\omega_1 \asymp \omega_2$  when $\omega_1= O(\omega_2)$, that is, $\omega_1(t) \leq C \omega_2(t)$ as $t\to 0$, for some $C>0$,
\item[] $\omega_1 \ll \omega_2$ when  $\omega_1=o(\omega_2)$, that is, $\lim_{t\to 0} \frac{\omega_1(t)}{\omega_2(t)}=0$.
\end{enumerate}
Given a modulus of continuity $\omega$ and $\Omega\subseteq \R^n$ we consider the space
$$
C^{0,\omega}(\Omega):=\left\{  f\in C^0(\Omega) \colon [f]_{C^{0,\omega}(\Omega)} <\infty \right\},
$$
where
$$
[f]_{C^{0,\omega}(\Omega)}:=\sup \left\{ \frac{|f(x)-f(y)|}{\omega(|x-y|)}\colon x,y\in \Omega, x\neq y \right\}.
$$

\begin{exam}
The space $C^{0,\omega}$ includes \emph{H\"older continuous function} when $\omega(t)=t^\alpha$, $\alpha\in (0,1)$; \emph{Lipschitz functions} when   $\omega(t)=t$; \emph{almost Lipschitz function} when $\omega(t)=t(1+|\log t|)$, etc.

Observe that a function $v\in C^0$ always belongs to $C^{0,\omega}_{loc}$, where $\omega$ is the modulus of continuity of $v$.
\end{exam}

The following  characterization for functions in $C^{0,\omega}$ will be useful.

\begin{lema}
Let $\omega$ be a modulus of continuity and $\Omega\subseteq \R^n$ be an open set. The following statements are equivalent:
\begin{enumerate}[leftmargin=15pt, label=(\roman*)]\itemsep2pt \parskip3pt \parsep2pt
\rm \item \it
$f\in C^{0,\omega}(\Omega)$;
\rm \item \it
there is a positive constant $M$ such that $
[f]_{C^{0,\omega}(\Omega)} \leq M$;
\rm \item \it
for all $x_0\in \Omega$ and $r_0>0$ such that $B_{r_0}(x_0)\subset \Omega$ it holds that
$$
\osc_{B_r (x_0)} f:= \sup_{B_r(x_0)} f - \inf_{B_r(x_0)} f  \leq C \omega(r) \quad \forall 0\leq r\leq r_0
$$
where $C$ is a positive constant.
\end{enumerate}
\end{lema}

Given $k\in \N_0$ and $\alpha\in (0,1)$, the H\"older space $C^{k,\alpha}(\Omega)$ has assigned the norm
$$
\|f\|_{C^{k,\alpha}(\Omega)} := \max_{|\beta|\leq k} \sup_{x\in\Omega} |D^\beta f(x)| + \max_{|\beta|=k} |D^\beta f|_{C^{0,\alpha}(\Omega)}.
$$

We also recall the standard notation for multi-indexes.

Let $\alpha=(\alpha_1,\ldots, \alpha_n)\in \Z^n_+$ and  $\beta=(\beta_1,\ldots, \beta_n)\in \Z^n_+$ be two \emph{multi-indexes.} We denote

\begin{enumerate}[leftmargin=15pt, label=(\roman*)]\itemsep2pt \parskip3pt \parsep2pt

\item[] $|\alpha|= \alpha_1 + \cdots + \alpha_n$,
\item[] $\alpha\leq \beta$ means that $\alpha_i \leq \beta_i$ for all $1\leq i \leq n$,
\item[] $D^\alpha u(x,t)$ stands for $D^\alpha u(x,t):=D_x^\alpha v(x)=\partial_{x_1}^{\alpha_1}\cdots \partial_{x_n}^{\alpha_n} u(x,t)$,
\item[] $\alpha! = \alpha_1! \alpha_2! \cdots \alpha_n!$,
\item[] $\begin{pmatrix} \alpha\\ \beta \end{pmatrix}$ means the product $\begin{pmatrix} \alpha_1\\ \beta_1 \end{pmatrix} \cdots \begin{pmatrix} \alpha_n\\ \beta_n \end{pmatrix}$.
\end{enumerate}

We recall that for $p\geq 2$, we denote $L_p(t)=|t|^{p-2}t$ and $M_p(t)=|t|^{p-2}$.

For our purposes we use the \emph{multivariate Fa\'a di Bruno's formula} to compute derivatives of a composition of functions. Let $u(x,t)\colon \R^n\times \R^+_0\to \R$ be a smooth enough function and let $\beta=(\beta_1, \beta_2, \cdots , \beta_n)$ be a multi-index with $|\beta|=k$, then
\begin{equation} \label{faa}
D^\beta (L_p(u(x,t))) =
\sum_{r=1}^k L_p^{(r)}u(x,t) \, d_{r}(u(x,t))
\end{equation}
where $d_r$ depends on the product of the different combination of derivatives of order $r=1,\dots, k$, and whose precise formula can be found, for instance, in \cite{CS96}[Corollary 2.10]. Namely, we have
$$
d_r(u(x,t))=\sum_{p(\beta,r)} \beta! \prod_{j=1}^k \frac{[D^{\ell_j}u(x,t)]^{\kappa_j}}{\kappa_j! (\ell_j!)^{k_j}}
$$
where $p(\beta,r)$ is the set of $(\kappa_1,\ldots, \kappa_k;\ell_1,\ldots, \ell_k)$ (with $\kappa_i\in \Z^+_0$ and $\ell_i$ are multi-indexes, $1\leq i \leq k$)   such that for some $1\leq s \leq k$, $\kappa_i=0$ and $\ell_i=0$ for $1\leq i \leq k-s$; $\kappa_1>0$ for $k-s+1\leq i\leq k$; and $0\leq \ell_{k-s+1}\leq \cdots \leq \ell_k$ are such that
$$
\sum_{i=1}^k \kappa_i = r, \quad \sum_{i=1}^k \kappa_i \ell_i =\beta.
$$
For instance, \eqref{faa} when $u=u(x_1,x_2,x_3)$ becomes
\begin{align*}
\partial_{x_1} \left(L_p u \right) &= u_{x_1} L_p'v , \\
\partial_{x_1} \partial_{x_2} \left( L_p u  \right)&= v_{x_1} u_{x_2}  L_p''u +v_{x_{1} x_{2}}  L_p' v,\\
\partial_{x_1} \partial^2_{x_3} \left(L_p u \right)&= u_{x_1} u_{x_3} u_{x_3} L_p'''u +(u_{x_1} u_{x_3x_3}  +
 2 u_{x_1} u_{x_1x_3}
  )  L_p'' u + u_{x_1 x_3x_3}  L_p' u,
\end{align*}
and when $u=u(x_1,x_2)$, \eqref{faa} becomes
$$
\partial_{x_1}\partial_{x_2}^2 (L_p u) = u_{x_2}^2 u_{x_1} L_p'''u + L_p'' (u_{x_2x_2} u_{x_1} + 2 u_{x_2} u_{x_1x_2}) + L_p' u_{x_1 x_2 x_2}.
$$
From the formula of $d_r$ it can be seen that
\begin{equation} \label{faa1}
D^\beta (L_p(u(x,t))) =
L_p' u \, D^\beta u(x,t) + R_{k-1}(u(x,t))
\end{equation}
where $R_{k-1}(u)$ contains the terms with ``lower order derivatives" or order up $k-1$, and it is such that
$$
R_{k-1}(u(x,t)) \leq k! \,\mathsf{\bar m}_{p}(t) \sum_{j=2}^k L_p^{(j)}u(x,t)
$$
with $\mathsf{\bar m}_{p}(t)$ denoting a  function such that
$$
\prod_{\ell=1}^k \sum_{0\leq |\alpha|\leq k-1}\|D^\alpha u(t)\|_{\LL^\infty(\R^n)}^\ell \leq \mathsf{\bar m}_{p}(t).
$$

When $u$ is solution of \eqref{eqC}, notice that in view of the smoothing effect, we do need to require a priori boundedness of $u_0$, indeed
\[
\|u( t)\|_{\LL^\infty(\R^n)} \leq
\frac{\mathsf{\tilde K}_p}{t_0^{\frac{1}{p-2}}}   + \mathsf{K}_{p,q,J}\|u_0\|_{\LL^q(\R^n)} :=\mathsf{m}_0 \quad \text{ for all } 0< t_0\leq  t\leq T.
\]
Notice also that when $u_0\in \LL^\infty(\R^n)$ we have that $\|u( t)\|_{\LL^\infty(\R^n)}\le \|u_0\|_{\LL^\infty(\R^n)}:=\mathsf{m}_0$ and all the following results will extend up to $t=0$. Summing up, we define
\begin{equation}\label{m0.def}
\|u( t)\|_{\LL^\infty(\R^n)} \le \mathsf{m}_0:=\begin{cases}
  \frac{\mathsf{\tilde K}_p}{t_0^{\frac{1}{p-2}}}   + \mathsf{K}_{p,q,J}\|u_0\|_{\LL^q(\R^n)} & \mbox{when $q\in (1,\infty)$ and $t\geq t_0>0$}\\
  \|u_0\|_{\LL^\infty(\R^n)}& \mbox{when $q=\infty$ and $t\ge 0$}.
\end{cases}
\end{equation}
In fact, we can take in this case $\mathsf{\bar m}_0=\mathsf{m}_0$.

An analogous expression holds for solutions of \eqref{eq} and \eqref{eqD} by replacing $\R^n$ with $\Omega$.

\medskip

\begin{thm}[H\"older regularity in space] \label{teo.holder.cont}

Let $p> 2$ and let $u$ be a solution of \eqref{eqC} starting from the initial datum $u_0\in \LL^q(\R^n)$ with $q\in [1,\infty]$ and let $\mathsf{m}_0$ be as in \eqref{m0.def}. Moreover, assume that
\begin{enumerate}[leftmargin=15pt, label=(\roman*)]\itemsep2pt \parskip3pt \parsep2pt
\rm \item \it
$D^\alpha u_0\in C^{0,\omega}(\R^n)\cap \LL^\infty(\R^n)$ for any $1\le |\alpha|\leq [p]-1$,
and define $\mathsf{m}_p$ as
\begin{equation} \label{mp.def}
\mathsf{m}_p:=\mathsf{m}_0 + \sum_{1\leq |\alpha|\leq [p]-1}\|D^\alpha  u_0\|_{\LL^\infty(\R^n)},
\end{equation}

\rm \item \it
there exists a modulus of continuity $\omega_{J,p}$ such that  for \text{a.e.} $ x,y \in \R^n$
\begin{equation} \label{mod.J'}
 \sum_{0\leq |\alpha|\leq [p]-1}\int_{\R^n} \left|D^\alpha J(y-z)-D^\alpha J(x-z) \right|\,dz \leq \omega_{J,p}(|x-y|).
\end{equation}
\end{enumerate}

Then $D^\alpha u(\cdot,t) \in C^{0,\bar \omega}(\R^n)$ for any $|\alpha|\leq [p]-1$ and  $t\ge t_0>0$, where the modulus of continuity $\bar\omega$ is given by $\bar\omega(\rho)=\max\{ \rho, \rho^{p-2},   \omega_{J,p}(\rho), \omega(\rho)\}$.  Moreover, the following estimates hold true for all $|\alpha|\leq [p]-1$,  all $x_1,x_2\in \R^n$ and all $t\ge t_0>0$
\begin{equation} \label{estimate.holder}
|D^\alpha u(x_1,t)- D^\alpha u(x_2,t)|
\leq {\mathsf K}(t)\,\bar\omega(|x_1-x_2|)
\end{equation}
where ${\mathsf K}(t)$ is a function that depends only on $p$, $n$, $J$, and $\mathsf{m}_p$.

When $u_0\in \LL^\infty(\R^n)$ the result holds for all $t\geq 0$.

When $2\leq p\in \N$,  estimate \eqref{estimate.holder} holds for $|\alpha|=k$ for  any $k\in \N$ provided that moreover $u_0\in C^\infty(\R^n)$ and $J\in C^\infty(\R^n)$. In particular $\mathsf{K}(t)$ is a function that depends only on $p$, $n$, $J$, and $\mathsf{m}_p$ but not on $k$ nor on $k-$th derivatives of $u_0$.

When $\Omega$ is open and bounded, analogous results hold for solution of \eqref{eqD}. If in addition $\Omega$ satisfies \eqref{HJ}, the same holds for solutions of \eqref{eq}.

\end{thm}
\begin{rem} \rm ${\mathsf K}(t)$ is defined for all $t\ge t_0>0$ and has an (almost) explicit expression given in the proof. The form of ${\mathsf K}(t)$ reveals that when $u_0\in \LL^\infty(\R^n)$ we can extend the result up to $t_0=0$, that is when we can take $q=\infty$.  This information is encoded in the expression of $m_0$ as in \eqref{m0.def}. When $q<\infty$, we cannot extend the estimate for $t=0$ since ${\mathsf K}(t)\to \infty$ as $t_0\to 0$.  Also we remark that ${\mathsf K}(t)\to \infty$ as $t\to \infty$.
\end{rem}

\begin{proof}
We prove the result for solutions of the Cauchy problem. For the Dirichlet and Neumann case the proof is analogous.

\noindent$\circ~$\textsc{Case $p\in (1,2)$}.
Let $u$ be a solution of \eqref{eq} corresponding to the initial datum $u_0\in C^{0,\omega}(\R^n)\cap \LL^\infty(\R^n)$. In this case we have that $|\alpha|=0$ and therefore we need to  prove that for any $x_1,x_2\in \R^n$ and $t\geq 0$ there exists a positive function $\mathsf{K}(t)$
 depending on $p$, $n$ and $\mathsf{m}_0$ such that
\begin{equation} \label{ineq.p.1.2}
|u(x_1,t)-u(x_2,t)| \leq \mathsf{K}(t) \bar\omega(|x_1-x_2|),
\end{equation}
where $\bar \omega(\rho):=\max\{\omega_{J,p}(\rho),\omega(\rho)\}$.

\medskip

\noindent$\circ~$\textsc{Step 1}. Let $x_1,x_2\in \R^n$ be fixed and let $t\geq 0$. Assume first that $u(x_1,t)\geq u(x_2,t)$. Using the fact that $u$ solves equation \eqref{eq} we can write
\begin{align*}
u_t(x_1,t)-u_t(x_2,t)&=\int_{\R^n} \left( J(x_1-y)L_p (u(y,t)-u(x_1,t))- J(x_2-y)L_p (u(y,t)-u(x_2,t) \right)\,dy\\
&=-\int_{\R^n} J(x_1-y)\left( L_p(u(y,t)-u(x_2,t)) - L_p(u(y,t)-u(x_1,t))\right)\,dy\\
&\quad + \int_{\R^n} L_p(u(y,t)-u(x_2,t)) \left( J(x_1-y)-J(x_2-y) \right)\,dy:=(i)+(ii).
\end{align*}
From Lemma \ref{lemap12} and \eqref{mod.J'} we get that
\begin{align*}
(i) \leq -(p-1)  \frac{u(x_1,t)-u(x_2,t)}{(1+2\|u_0\|_{\LL^\infty(\R^n)})^\frac{2-p}{2}},\qquad
(ii) \leq 2\|u_0\|_{\LL^\infty(\R^n)}^{p-1} \omega_J(|x_1-x_2|),
\end{align*}
and, denoting $y(t)=u(x_1,t)-u(x_2,t)$ and
$$
A:=
\frac{p-1}{(1+2\|u(t)\|_{\LL^\infty(\R^n)})^\frac{2-p}{2}}, \qquad
B:=2\|u_0\|_{\LL^\infty(\R^n)}^{p-1} \omega_{J,p}(|x_1-x_2|),
$$
we get the following differential inequality
\begin{equation} \label{edop12}
y'(t) \leq -A y(t) + B, \qquad t\geq 0.
\end{equation}

\noindent$\circ~$\textsc{Step 2}.  Observe that the function  $\bar y(t)= y(0) e^{-A t} + Bt$ is a supersolution to \eqref{edop12} since
 $-A \bar y(t) + B = \bar y(t) - ABt \leq \bar y(t)$ for all $t\geq 0$. Then, $y(t)\leq \bar y(t)$ for all $t\geq 0$ and we get
\begin{align} \label{ineq.p.1.2.1}
\begin{split}
u(x_1,t)-u(x_2,t) &\leq (u_0(x_1)-u_0(x_2))e^{-At} + Bt\\
&\leq \omega(|x_1-x_2|)e^{-At} + 2t\|u_0\|_{\LL^\infty(\R^n)}^{p-1} \omega_{J,p}(|x_1-x_2|)\\
&\leq \max\{\omega(|x_1-x_2|), \omega_{J,p}(|x_1-x_2|)\} \left( e^{-At} + 2t\|u_0\|_{\LL^\infty(\R^n)}^{p-1} \right)\\
&:= \bar \omega(|x_1-x_2|) \mathsf{K}(t)
\end{split}
\end{align}
where we have used that $u_0\in C^{0,\omega}(\R^n)$, and $\mathsf{K}(t)$  depends on $p$, $n$ $\mathsf{m}_0$ and $t$.

\noindent$\circ~$\textsc{Step 3}. When $u(x_1,t)\leq u(x_2,t)$, proceeding similarly we get
$$
\frac{d}{dt}(u(x_2,t)-u(x_1,t)) \leq -(u(x_2,t)-u(x_1,t)) A + B, \qquad t\geq 0
$$
and  then $u(x_2,t)-u(x_1,t) \leq \bar\omega(|x_1-x_2|) \mathsf{K}(t)$. This inequality together with \eqref{ineq.p.1.2.1} gives \eqref{ineq.p.1.2}.

\medskip

\noindent$\circ~$\textsc{Case $p>2$}.  From now on $p>2$ and $u$ denotes a solution of \eqref{eqC} corresponding to the initial datum $u_0\in \LL^q(\R^n)$ with $q\in [1,\infty]$. Moreover, $x_1,x_2\in \R^n$ and  $t\geq t_0>0$ with  $t_0=0$ when $u_0\in \LL^\infty(\R^n))$.

We split the proof in several steps.

\noindent$\circ~$\textsc{Step 1}. Let us prove that $D^\beta u_0$ bounded for any $|\beta|\leq [p]-1$ implies that $D^\beta u(t)$ is bounded for $|\beta|\leq [p]-1$ and all $t\ge t_0$.

More precisely, we prove the following statement:

\emph{Assume that $\|D^\beta u(t)\|_{\LL^\infty(\R^n)} \leq \mathsf{\tilde m}_p(t)$ for $|\beta|\leq [p]-2$ and $\|D^\beta u_0\|_{\LL^\infty(\R^n)} \leq \mathsf{m}_p$ for $|\beta|\leq [p]-1$. Then there is a function $\mathsf{\bar m}(t)$ such that $\|D^\beta u(t)\|_{\LL^\infty(\R^n)} \leq \mathsf{\bar m}_p(t)$ for $|\beta|\leq [p]-1$.
}

Here $\mathsf{\bar m}_p(t)$ and $\mu_p(t)$ depend on $p$, $n$, $J$ and the constant $\mathsf{m}_0$ given in \eqref{m0.def}.

We prove it by induction on the order of $\beta$.

\noindent\textsc{Case  $|\beta|=1$}.  By hypothesis $\|(u_0)_{x_i}\|_{\LL^\infty(\R^n)} \leq \mathsf{m}_p$ for $i=1,\ldots, n$, and in light of \eqref{m0.def} we have
$$
\|u_0\|_{\LL^\infty(\R^n)} \leq \mathsf{m}_0, \qquad \|u(t)\|_{\LL^\infty(\R^n)}\leq  \mathsf{m}_0=:\mathsf{\tilde m}_p(t).
$$
In this case, for any $i\in \{1,\ldots, n\}$ we have that
\begin{align*}
\frac{d}{dt}&|u_{x_i}(x,t)| = \sign(u_{x_i}(x,t)) \partial_{x_i} u_t(x,t)=
\sign(u_{x_i}(x,t)) \left( \int_{\R^n}  J_{x_i}(x-y) L_p(u(y,t)-u(x,t))\,dy \right.\\
&\quad \left. -(p-1)u_{x_i}(x,t)\int_{\R^n} J(x-y) M_p(u(y,t)-u(x,t))\,dy \right)\\
& \leq
2^p\| u(t)\|_{\LL^\infty(\R^n)}^{p-1} \int_{\R^n} |\nabla J(x-y)|\,dy + |u_{x_i}(x,t)| 2^{p-1}(p-1) \|u(t)\|_{\LL^\infty(\R^n)}^{p-2} \int_{\R^n} J(x-y) \,dy\\
&\leq  A |u_{x_i}(x,t)| + B
\end{align*}
where we have denoted
$$
A:=(p-1) 2^{p-1}  \max\{\mathsf{m}_0,\mathsf{m}_0^{p-2} \} ,\qquad B:=2^p \max\{\mathsf{m}_0,\mathsf{m}_0^{p-1}\} \int_{\R^n} |\nabla J(x-y)|\,dy .
$$
Solving the differential inequality above gives a bound for $|u_{x_i}(x,t)|$ for all $x\in \R^n$ and all $t\geq t_0$:
\begin{equation*}
\|u_{x_i}(t)\|_{\LL^\infty(\R^n)} \leq e^{A t} \mathsf{m}_1 + \frac{B}{A}(e^{At}-1):=\mu_p (t).
\end{equation*}

\noindent \textsc{Inductive step.} Assume that
$$
\|D^\beta u(t)\|_{\LL^\infty(\R^n)} \leq  \mathsf{\tilde m}_p(t) \text{ for }|\beta|\leq [p]-2  \text{ and }t\geq t_0, \qquad \|D^\beta u_0\|_{\LL^\infty(\R^n)} \leq \mathsf{m}_p \text{ for }|\beta|\leq [p]-1,
$$
for some $\mathsf{\tilde m}_p(t)$ depending on $p$, $n$, $J$ and $\mathsf{m}_0$.

Let us see that for $|\beta|= [p]-1$ there is a  function $\mathsf{\bar m}_p(t)$ depending on $p$, $n$, $J$, $\mathsf{\tilde m}_p$ and $\mathsf{m}_0$ such that
\begin{equation} \label{cota.hyp}
\|D^\beta u(t)\|_{\LL^\infty(\R^n)} \leq \mathsf{\bar m}_p(t) \qquad \text{ for } t\geq t_0.
\end{equation}
By using the Leibniz formula for derivatives
\begin{align*}
\frac{d}{dt} D^\beta u(x,t)&=D^\beta u_t(x,t) =
D^\beta \int_{\R^n} J(x-z) L_p(u(x,t)-u(z,t))\\
&= \sum_{0\leq \alpha \leq \beta} \begin{pmatrix} \beta \\ \alpha  \end{pmatrix} \int_{\R^n}D^\alpha J(x-z) D^{\beta-\alpha} (L_p(u(x,t)-u(z,t)))  \,dz.
\end{align*}
Then we can write
\begin{align*}
\frac{d}{dt} |D^\beta u(x,t)| &=
\sign(D^\beta u(x,t)) D^\beta u_t(x,t)\\
&\leq
\sum_{0\leq \alpha \leq \beta} \begin{pmatrix} \beta \\ \alpha  \end{pmatrix} \int_{\R^n} |D^\alpha J(x-z)| |D^{\beta-\alpha}  L_p(u(x,t)-u(z,t))|  \,dz\\
&=
\sum_{1\leq \alpha \leq \beta} \begin{pmatrix} \beta \\ \alpha  \end{pmatrix} \int_{\R^n} |D^\alpha J(x-z)| |D^{\beta-\alpha} L_p(u(x,t)-u(z,t))|  \,dz
\\
&\quad +
 \int_{\R^n} |J(x-z)| |D^{\beta}  L_p(u(x,t)-u(z,t))|  \,dz:=(i)+(ii).
\end{align*}
Let us  estimate $(i)$. From \eqref{derivada} and \eqref{m0.def}, for any $1\leq r\leq [p]-1$ we get
\begin{align} \label{rema1a}
\left| L_p^{(r)} u(x,t) \right| \leq \mathsf{c}_{r,p}\|u(t)\|_{\LL^\infty(\R^n)}^{p-r-1} \leq  (p-1)^{p-1} \max\{\mathsf{m}_0^{p-2}, \mathsf{m}_0^{p-[p]} \}.
\end{align}
In this case $|\beta-\alpha|\leq [p]-2$, then using \eqref{faa1} and and the inductive hypothesis  we get
\begin{align} \label{cota.orden.k}
\begin{split}
D^{\beta-\alpha}&L_p(u(x,t)-u(z,t)) =
L_p' u(x,t) \, D^{\beta-\alpha} u(x,t) + ([p]-2)! \mathsf{c}_0(p,\mathsf{m}_0, \mathsf{\bar m}_{p})(t) \sum_{r=2}^{[p]-2} L_p^{(r)}u(x,t) \\
&\leq
\|L_p' u(t)\|_{\LL^\infty(\R^n)} \, \|D^{\beta-\alpha} u(t)\|_{\LL^\infty(\R^n)} + ([p]-2)!\, \mathsf{c}_0(p,\mathsf{m}_0,\mathsf{\tilde m}_{p}(t)) \sum_{r=2}^{[p]-2} \|L_p^{(r)}u(t)\|_{\LL^\infty(\R^n)} \\
&\leq \mathsf{c}_1(p,\mathsf{m_0},\mathsf{\tilde  m}_{p}(t)).
\end{split}
\end{align}
This gives that
$$
(i)\leq \mathsf{c}_1(p,\mathsf{m}_0,\mathsf{\tilde  m}_{p}(t)) \sum_{1\leq \alpha \leq \beta} \begin{pmatrix} \beta \\ \alpha  \end{pmatrix} \int_{\R^n} |D^\alpha J(x-z)|   \,dz \leq    \mathsf{c}_2(p,n,J,\mathsf{m}_0,\mathsf{\tilde  m}_{p}(t)).
$$
To bound $(ii)$ we use again \eqref{faa1}, \eqref{rema1a} and the fact that $|\beta|=[p]-1$, similarly as before
\begin{align*}
D^{\beta}L_p(u(x,t)-u(z,t)) &=
L_p' u(x,t) \, D^{\beta} u(x,t) + ([p]-1)! \mathsf{c}_3(p,\mathsf{m}_0,\mathsf{\tilde m}_p(t)) \sum_{r=2}^{[p]-1} L_p^{(r)}u(x,t) \\
&\leq (p-1)^{p-1} \max\{\mathsf{m}_0^{p-2},\mathsf{m}_0^{p-[p]}\} |D^\beta u(x,t)| + \mathsf{c}_4(p,\mathsf{m}_0,\mathsf{\tilde  m}_{p}(t)),
\end{align*}
giving that
$$
(ii) \leq
\mathsf{c}_5(p,J,\mathsf{m}_0) |D^\beta u(x,t)| + \mathsf{c}_6(p,J,\mathsf{m}_0, \mathsf{\tilde  m}_{p}(t)).
$$
Combining these expressions leads to
\begin{align} \label{edo1}
\frac{d}{dt} |D^\beta u(x,t)| \leq A(p,J,\mathsf{m}_0) |D^\beta u(x,t)| + B(t),
\end{align}
where $B(t)$ depends on $p$, $J$ , $\mathsf{m}_0$ and $\mathsf{\tilde m}_{p}(t)$. Solving this differential inequality and using the inductive hypothesis gives the desired bound:
\begin{align} \label{cota.dbut}
\begin{split}
|D^\beta u(x,t)| &\leq |D^\beta u_0(x)| e^{A t} + \int_{t_0}^t B(s) e^{(t-s)A}\,ds\\
&\leq
\mathsf{m}_p e^{A t} + \int_{t_0}^t B(s) e^{(t-s)A}\,ds:=\mathsf{\bar m}_p(t)\qquad\mbox{for all $t\ge t_0$}\,.
\end{split}
\end{align}

\noindent$\circ~$\textsc{Step 2}. We compute some  estimates for the difference of derivatives. Observe that  by using the mean value theorem and \eqref{cota.hyp} we get that
$$
|u(x_1,t)-u(x_2,t)| \leq \mathsf{\bar m}_p(t) |x_1-y_1|,
$$
then using Lemma \ref{lema.Lp.gral} we obtain that for $0\leq r  \leq [p]-1$
\begin{align} \label{rema1}
\begin{split}
| &L_p^{(r)}  (u(x_1,t)-u(z,t)) - L_p^{(r)} (u(x_2,t) -u(z,t))|\leq \\
&\leq 2^p   \mathsf{\bar c}_{r,p}  \max\{|u(x_1,t)-u(x_2,t)|, |u(x_1,t)-u(x_2,t)|^{p-r-1}\}  \max\{1,(2\|u(t)\|_{\LL^\infty(\R^n)})^{p-r-2} \}\\
&\leq \mathsf{K}_1(t)  \max\{|x_1-x_2|, |x_1-x_2|^{p-r-1}\}
\end{split}
\end{align}
where $\mathsf{K}_1(t)$ depends on $p$, $\mathsf{m}_0$ and $ \mathsf{\bar m}_p$.

\noindent$\circ~$\textsc{Step 3}.
Let $|\beta|=k\leq [p]-1$. By using the Leibniz formula for derivatives we have that
\begin{align*}
D^\beta &u_t(x_1,t) - D^\beta u_t(x_2,t) = \sum_{0\leq \alpha \leq \beta} \begin{pmatrix} \beta \\ \alpha  \end{pmatrix} \int_{\R^n} D^\alpha J(x_1-z) D^{\beta-\alpha} (L_p(u(x_1,t)-u(z,t)))  \,dz\\
&\quad -\sum_{0\leq \alpha \leq \beta} \begin{pmatrix} \beta \\ \alpha  \end{pmatrix} \int_{\R^n} D^\alpha J(x_2-z) D^{\beta-\alpha} (L_p(u(x_2,t)-u(z,t)))  \,dz\\
&\quad \pm
\sum_{0\leq \alpha \leq \beta} \begin{pmatrix}  \beta \\ \alpha  \end{pmatrix} \int_{\R^n} D^\alpha J(x_1-z) D^{\beta-\alpha} (L_p(u(x_2,t)-u(z,t)))  \,dz\\
&= \sum_{0\leq \alpha \leq \beta} \begin{pmatrix} \beta \\ \alpha  \end{pmatrix} \int_{\R^n} D^\alpha J(x_1-z) \left(
D^{\beta-\alpha} (L_p(u(x_1,t)-u(z,t)) ) -
D^{\beta-\alpha} (L_p(u(x_2,t)-u(z,t)))
\right)
\,dz\\
&\quad +
\sum_{0\leq \alpha \leq \beta} \begin{pmatrix} \beta \\ \alpha  \end{pmatrix} \int_{\R^n} \left( D^\alpha J(x_1-z) - D^\alpha J(x_2-z) \right)
D^{\beta-\alpha} (L_p(u(x_2,t)-u(z,t)) ) \,dz.
\end{align*}
Now, since
$$
\frac{d}{dt}|D^\beta u(x_1,t)- D^\beta u(x_2,t)| = \sign(D^\beta u(x_1,t)- D^\beta u(x_2,t)) (D^\beta u_t(x_1,t)-D^\beta u_t(x_2,t)),
$$
in light of the previous expression we obtain that
\begin{align}\label{eq.final}
\begin{split}
\frac{d}{dt}&|D^\beta u(x_1,t)- D^\beta u(x_2,t)| \leq |D^\beta u_t(x_1,t)- D^\beta u_t(x_2,t)| \\
&\leq
\int_{\R^n}
|J(x_1-z)| \left|
D^{\beta} (L_p(u(x_1,t)-u(z,t)) ) -
D^{\beta} (L_p(u(x_2,t)-u(z,t)))
\right|
\,dz\\
&\quad +
\sum_{1\leq \alpha \leq \beta} \begin{pmatrix} \beta \\ \alpha  \end{pmatrix} \int_{\R^n}
|D^\alpha J(x_1-z)| \left|
D^{\beta-\alpha} (L_p(u(x_1,t)-u(z,t)) ) -
D^{\beta-\alpha} (L_p(u(x_2,t)-u(z,t)))
\right|
\,dz\\
&\quad +
\sum_{0\leq \alpha \leq \beta} \begin{pmatrix} \beta \\ \alpha  \end{pmatrix} \int_{\R^n}
\left| D^\alpha J(x_1-z) - D^\alpha J(x_2-z) \right|
|D^{\beta-\alpha} (L_p(u(x_2,t)-u(z,t)) )| \,dz\\
&:=(I_1)+(I_2)+(I_3).
\end{split}
\end{align}

In order to bound $(I_1)$ we use  expression \eqref{faa1} to write $D^{\beta} (L_p(u))$ as $L_p'u(x,t) \, D^{\beta}  u + R_{k-1}u$, where
$$
R_{k-1}u(x,t) \leq k! \mathsf{\bar m}_p(t) \sum_{j=2}^{k} L_p^{(j)}u(x,t)
$$
since $k\leq [p]-1$. This allows to write:
\begin{align*}
| &D^{\beta} (L_p(u(x_1,t)-u(z,t)) ) -
D^{\beta} (L_p(u(x_2,t)-u(z,t))) | = \\
&=\left|
L_p'u(x_1,t) \, D^{\beta}  u(x_1,t) + R_{k-1}(u(x_1,t)) -
L_p'u(x_2,t) \, D^{\beta}  u(x_2,t) - R_{k-1}(u(x_2,t)) \right.\\
&\quad \pm \left.
L_p'u(x_2,t) \, D^{\beta}  u(x_1,t)  \right|\\
&\leq
\left| L_p'u(x_1,t)-L_p'u(x_2,t)  \right|\,\left| D^{\beta}  u(x_1,t)
\right| + \left|L_p'u(x_2,t) \right| \left|D^{\beta} u(x_1,t) - D^{\beta} u(x_2,t) \right|\\
&\quad + \left| R_{k-1}(u(x_1,t)) - R_{k-1}(u(x_2,t)) \right|\\
&:=(i_1)+(i_2)+(i_3).
\end{align*}
To bound $(i_1)$ we use expression \eqref{cota.hyp} from Step 1 and \eqref{rema1} to get
$$
(i_1)\leq \mathsf{K}_1(t) \mathsf{\bar m}_p(t) \hat \omega(|x_1-x_2|),
$$
where we have denoted $\hat \omega(\rho):=\max\{ \rho ,  \rho^{p-2}\}$. The term $(i_2)$ can be bounded  using \eqref{rema1a} as
$$
(i_2)\leq c_{1,p} \max\{\mathsf{m}_0,\mathsf{m}_0^{p-2} \} \left|D^{\beta} u(x_1,t) - D^{\beta} u(x_2,t) \right|.
$$
Finally, the term $(i_3)$ can be bounded using the mean value theorem, \eqref{cota.hyp} and  \eqref{rema1a} as
$$
(i_3)= \mathsf{c}_6 (p,\mathsf{m}_0,\mathsf{\bar m}_p(t))|x_1-x_2|.
$$
These estimates allow to bound  $(I_1)$ as
\begin{align*}
(I_1)&\leq ( (i_1)+ (i_2)+(i_3))
\int_{\R^n}
|J(x_1-z)| \,dz\\
&\leq
\mathsf{c}_7(p,J, \mathsf{m}_0)  \left|D^{\beta} u(x_1,t) - D^{\beta} u(x_2,t) \right| + \mathsf{c}_8(p,J, \mathsf{m}_0, \mathsf{\bar m}_p(t) )\, \hat \omega(|x_1-x_2|).
\end{align*}
To bound $(I_2)$ we proceed similarly as for  $(I_1)$. In this case $|\beta-\alpha|\leq k-1 \leq [p]-2$ and from \eqref{faa1}
\begin{equation*}
D^{\beta-\alpha} (L_p(u(x,t))) =
(L_p')u(x,t) \, D^{\beta-\alpha}  u(x,t) + R_{k-2}(u(x,t))
\end{equation*}
where $R_{k-2}(u(x,t))$ can be bounded using \eqref{cota.hyp} as
$$
R_{k-2}(u(x,t)) \leq (k-1)! \mathsf{\bar m}_{p}(t) \sum_{j=2}^{k-1} L_p^{(j)}u(x,t).
$$
This allows to write
\begin{align*}
| &D^{\beta-\alpha} (L_p(u(x_1,t)-u(z,t)) ) -
D^{\beta-\alpha} (L_p(u(x_2,t)-u(z,t))) |  \\
&\leq
\left| (L_p')u(x_1,t)-(L_p')u(x_2,t)  \right|\,\left| D^{\beta-\alpha}  u(x_1,t)
\right| + \left|(L_p')u(x_2,t) \right| \left|D^{\beta-\alpha} u(x_1,t) - D^{\beta-\alpha} u(x_2,t) \right|\\
&\quad + \left| R_{k-2}(u(x_1,t)) - R_{k-2}(u(x_2,t)) \right|\\
&:=(i_1)+(i_2)+(i_3).
\end{align*}
To bound $(i_1)$ we use expression \eqref{cota.hyp}  and \eqref{rema1} to get that $(i_1)\leq  \mathsf{K}_1(t) \mathsf{\bar m}_p(t) \hat \omega(|x_1-x_2|)$. By using \eqref{rema1a} the term $(i_2)$ can be bounded  as
\begin{align*}
(i_2)&\leq c_{1,p} \max\{\mathsf{m}_0,\mathsf{m}_0^{p-2} \} \left|D^{\beta-\alpha} u(x_1,t) - D^{\beta-\alpha} u(x_2,t) \right|\leq c_{1,p} \max\{\mathsf{m}_0,\mathsf{m}_0^{p-2}\} \mathsf{\bar m}_p(t) |x_1-x_2|
\end{align*}
where in the last inequality we have used the mean value theorem and \eqref{cota.hyp}.
\\
Finally, using again the mean vale theorem and  \eqref{rema1a} we can bound  $(i_3)$  as
$$
(i_3)\leq  \mathsf{c}_9(p,\mathsf{m}_0,\mathsf{\bar m}_p(t))|x_1-x_2|.
$$
With these estimates, $(I_2)$  can be upper bounded as follows:
\begin{align*}
(I_2)&\leq ( (i_1)+ (i_2)+(i_3))
\sum_{1\leq \alpha \leq \beta} \begin{pmatrix} \beta \\ \alpha  \end{pmatrix} \int_{\R^n}
|D^\alpha J(x_1-z)| \,dz\leq
\mathsf{c}_{10}(p,J, \bar m_p(t), m_0)\, \hat \omega(|x_1-x_2|).
\end{align*}
To bound $(I_3)$ observe that expression \eqref{cota.orden.k} gives
\begin{equation*}
|D^{\beta-\alpha} (L_p(u(x,t) -u(z,t)))| \leq
\mathsf{c}_{11}(p,\mathsf{m}_0,\mathsf{\bar m}_{p}(t)).
\end{equation*}
Then, from   hypothesis \eqref{mod.J'} we get
\begin{align*}
(I_3) &\leq \mathsf{c}_{11}(p,\mathsf{m}_0,\mathsf{\bar m}_{p}(t))
\sum_{0\leq \alpha \leq \beta} \begin{pmatrix} \beta \\ \alpha  \end{pmatrix} \int_{\R^n} \left( D^\alpha J(x_1-z) - D^\alpha J(x_2-z) \right)
 \,dz\\
&\leq \mathsf{c}_{12}(p,\mathsf{m}_0,\mathsf{\bar m}_{p}(t)) \, \omega_{J,p}(|x_1-x_2|).
\end{align*}
Finally, inserting the bounds of $(I_1)$, $(I_2)$ and $(I_3)$ into \eqref{eq.final} give  the differential inequality
\begin{align} \label{EDO}
\begin{split}
\frac{d}{dt}|D^\beta u(x_1,t)&- D^\beta u(x_2,t)|\leq |D^\beta u_t(x_1,t)- D^\beta u_t(x_2,t)|\leq \\
& \leq A(p,J,\mathsf{m}_0) \left|D^{\beta} u(x_1,t) - D^{\beta} u(x_2,t) \right| + B(t) \, \tilde \omega(|x_1-x_2|) ,
\end{split}
\end{align}
where  $\tilde \omega(\rho):= \max\{\rho^{p-2}, \rho, \omega_{J,p}(\rho)\}$, $B(t)$ depends on $p$, $J$, $\mathsf{m}_0$ and $\mathsf{\bar m}_{p}(t)$, and $A$ in independent on $t$. Solving \eqref{EDO} and using the hypothesis on $D^\beta u_0$ we get
\begin{align*}
|D^\beta u(x_1,t)- D^\beta u(x_2,t)| &\leq |D^\beta u_0(x_1)- D^\beta u_0(x_2)| e^{At} + \bar \omega(|x_1-x_2|)\int_0^t B(s) e^{(t-s)A}\,ds\\
&\leq \bar\omega(|x_2-x_1|) \left(e^{At} + \int_0^t B(s) e^{(t-s)A} \,ds\right)
\end{align*}
where we have denoted
$$
\bar \omega(\rho):= \max\{\rho, \rho^{p-2}, \omega_{J,p}(\rho),\omega(\rho) \}.
$$
Therefore, we conclude that
$$
\sup\left\{ \frac{|D^\beta u(x_1,t)- D^\beta u(x_2,t)| }{\bar\omega(|x_1-x_2|)} \colon x_1,x_2\in \R^n, x\neq y\right\}
\leq e^{At} + \int_0^t B(s) e^{A(t-s)}\,ds:=\overline{\mathsf K}(t)
$$
and then  $D^\beta u(\cdot ,t)\in C^{0,\bar\omega}(\R^n)$ for any $|\beta|\leq [p]-1$ and $t\geq t_0$. This concludes the proof.
\end{proof}

\subsection{H\"older regularity in space for time derivatives}

Under the considerations of Theorem \ref{teo.holder.cont}, we can state an analogous result for the time derivatives of solutions.

Recall that  the smoothing effect and the norm decreasing in time property of solutions gives
$$
\|u( t)\|_{\LL^\infty(\R^n)} \leq \mathsf{m}_0.
$$
where $\mathsf{m}_0$ is defined in \eqref{m0.def}. Furthermore, in light of Theorem \ref{smooth.ut.high},  there exists a function $\mathsf{M}_p>0$ depending on $n$, $p$ and the constant $\mathsf{\bar M}_1$ given in \eqref{bound.M1} such that for all $j=0,\ldots, [p]-1$
$$
\|\partial_t^j u(t)\|_{\LL^\infty(\R^n)} \leq \mathsf{M}_p.
$$
These estimates holds for $t\geq t_0>0$ when $u_0\in \LL^q(\R^n)$ with $q\in [1,\infty]$ and we can let $t_0=0$ when $u_0\in \LL^\infty(\R^n)$.

In the following theorem we prove that the modulus of continuity of the initial data and its derivatives is preserved by the  time derivatives of order up to $[p-1]$.

\begin{thm}[H\"older regularity in space for time derivatives] \label{teo.holder.cont.1}
Given $p> 2$, let $u$ be a solution of \eqref{eqC} starting from the initial datum $u_0\in \LL^q(\R^n)$ with $q\in [1,\infty]$.

Assume moreover conditions \eqref{mp.def}  and  \eqref{mod.J'}.

Then $\partial_t^j u(\cdot,t) \in C_x^{[p]-1,\bar \omega}(\R^n)$, for any $j\in \N_0$ such that $j\leq [p]-1$, where $\bar\omega$ is the modulus of continuity given by $\bar \omega(\rho)=\max\{ \rho, \rho^{p-2},  \omega_{J,p}(\rho), \omega(\rho)\}$.

More precisely, it holds that $D^\alpha \partial_t^j u(\cdot,t) \in C^{0,\bar \omega}(\R^n)$  for any $t\geq t_0>0$, where $\alpha$ is a multi-index such that $|\alpha|\leq [p]-1$ and   $j\in \N_0$ is such that $j\leq [p]-1$, and  there exists a positive constant $\mathsf{c}_p$ depending on $p$, $n$, $J$, $\mathsf{M}_p$, $\mathsf{m}_0$ and $t$ such that
\begin{align}\label{AAAA-1}
  \frac{|D^\alpha \partial_t^{j} u(x_1,t)- D^\alpha \partial_t^{j} u(x_2,t)| }{\bar\omega(|x_1-x_2|)}
\leq \mathsf c_p(t)
\end{align}
hold true for all $x_1,x_2\in \R^n$ and all $t\ge t_0$. We can let $t_0=0$ when $u_0 \in \LL^\infty(\R^n)$.

\noindent Moreover, when $p\in\N$, estimate \eqref{AAAA-1} holds for any $j\in \N_0$ and any $\alpha\in \N_0^n$, hence each $D^{\alpha}u(x,\cdot)$ is analytic in time with radius $r_0=1\wedge t_0$.
 
When $\Omega$ is open and bounded, analogous results hold for solution of \eqref{eqD}. If in addition $\Omega$ satisfies \eqref{HJ}, the same holds for solutions of \eqref{eq}.
\end{thm}

\begin{proof}
We prove the result for solutions of the Cauchy problem. For the Dirichlet and Neumann case the proof is analogous.

Let $u$ be a solution of \eqref{eqC} with $p>2$ corresponding to the initial datum $u_0\in C^0(\R^n)\cap \LL^q(\R^n)$ with $q\in [1,\infty]$ and let $t\geq t_0>0$ (we can let $t_0=0$ when $q=\infty$).

We split the proof in several steps.

\noindent$\circ~$\textsc{Step 1.} Let us see that there exists a positive function $\mathsf{c}(t)$ depending  on $p$, $n$, $J$, $\mathsf{m}_0$, $\mathsf{M}_p$ and $\mathsf{\bar m}_p(t)$ such that
\begin{equation} \label{el.step1}
\| D^\beta \partial_t^{j} u(t)\|_{\LL^\infty(\R^n)}  \leq \mathsf{c}(t),
\end{equation}
holds for all $j\in \{0,\ldots, [p]-1\}$ and all multi-index  $\beta$ such that $|\beta|\leq  [p]-1$.

We prove it by induction on $j$.

\noindent\textsc{Case $j=0$.} Let us see that $\| D^\beta  u_t(t)\|_{\LL^\infty(\R^n)}  \leq \mathsf{c}(t)$.

From \eqref{edo1} and  \eqref{cota.dbut} we have that
\begin{align*}
|D^\beta u_t(x,t)| &\leq A(p,J,\mathsf{m}_0) \||D^\beta u(t)\|_{\LL^\infty(\R^n)} + B(p,J,\mathsf{m}_0,\mathsf{\bar m}_{p}(t))\\
&\leq A(p,J,\mathsf{m}_0) \mathsf{\bar m}_p(t) + B(p,J,\mathsf{m}_0,\mathsf{\bar m}_{p}(t)):=\mathsf{c}(t).
\end{align*}
where $\mathsf{m}_0$ is given in \eqref{m0.def} and $\mathsf{\bar m}_p(t)$ is given in \eqref{cota.dbut}.

\noindent\textsc{Inductive step:} assume that $\| D^\beta \partial_t^{j} u(t)\|_{\LL^\infty(\R^n)} \leq \mathsf{c}(t)$ holds for all $j\in \{0,\ldots, [p]-2\}$ and $\beta$ such that $|\beta|\leq [p]-1$. Let us see that $\| D^\beta \partial_t^{[p]-1} u(t)\|_{\LL^\infty(\R^n)} \leq \mathsf{\tilde c}(t)$ holds for any $|\beta|\leq [p]-1$ for a suitable function $\mathsf{\tilde c}(t)$ depending on $\mathsf{c}(t)$.

Let $h>0$.  We use expression \eqref{bound.k1}  to write
$$
\partial_t^{[p]-1} u(x,t)  =
 \int_{\R^n} J(x-y) \partial_t^{[p]-2} A_p(x,y,t^*) \,dy,
$$
where $t^* \in (t,t+h)$ and $A_p(x,y,t):=L_p(u(y,t)-u(x,t))$.

Given $\beta$ such that $|\beta|\leq [p]-1$, by using the Leibniz formula for derivatives we get
\begin{align} \label{deriv.b.t}
\begin{split}
|D^\beta \partial_t^{[p]-1}  u(x_1,t)| &\leq \sum_{0\leq \alpha \leq \beta} \begin{pmatrix} \beta \\ \alpha  \end{pmatrix} \int_{\R^n} \left| D^\alpha J(x_1-z) \right| \left| D^{\beta-\alpha} \partial_t^{[p]-2} (A_p(x_1,z,t^*)) \right|  \,dz\\
&=
\int_{\R^n} \left| J(x_1-z) \right| \left| D^{\beta} \partial_t^{[p]-2} (A_p(x_1,z,t^*)) \right|  \,dz
\\&+
\sum_{1\leq  \alpha \leq \beta} \begin{pmatrix} \beta \\ \alpha  \end{pmatrix} \int_{\R^n} \left| D^\alpha J(x_1-z) \right| \left| D^{\beta-\alpha} \partial_t^{[p]-2} (A_p(x_1,z,t^*)) \right|  \,dz.
\end{split}
\end{align}
To bound \eqref{deriv.b.t} observe that the  multivariate Fa\'a di Bruno's formula \eqref{faa}  in this case reads as
\begin{equation} \label{form.der}
D^\beta \partial_t^{[p]-2} (A_p u(x,z,t^*)) = L_p'(u) D^\beta \partial_t^{[p]-2} u(x,t^*) + R_{[p]-3,[p]-1}(u(y,t^*)-u(x,t^*))
\end{equation}
where  $R_{[p]-3,[p]-1}u(x,t^*)$ involves derivatives $D^\alpha \partial_t^m u(x,t)$ with $0\leq m\leq [p]-3$ and $|\alpha|\leq [p]-1$ and by inductive hypothesis can be bounded as
$$
R_{[p]-3,[p]-1} \leq \mathsf{K}(p,\mathsf{m}_0, \mathsf{\bar m}_p(t^*),\mathsf{M}_p).
$$
Therefore, we can use  $L_p'(u(t^*))\leq \|u(t^*)\|_{\LL^\infty(\R^n)}$ and the inductive hypothesis to bound \eqref{form.der} as
\begin{align*}
|D^\beta \partial_t^{[p]-2} (A_p u(x,z,t^*))| &\leq (p-1)\max\{\mathsf{m}_0,\mathsf{m}_0^{p-2}\} | D^\beta \partial_t^{[p]-2} u(x,t^*)| + \mathsf{K}\\
&\leq (p-1)\max\{\mathsf{m}_0,\mathsf{m}_0^{p-2}\} \mathsf{c}(t^*) + \mathsf{K}
\end{align*}
for all $\beta$ such that $|\beta|\leq [p]-1$  which  used in \eqref{deriv.b.t} gives
$$
|D^\beta \partial_t^{[p]-1} u(x_1,t)|  \leq \mathsf{C}_1(p,n,J,\mathsf{m}_0,\mathsf{\bar m}_p(t^*), \mathsf{M}_p, \mathsf{c}(t^*)):=\mathsf{\tilde c}_p(t).
$$

\noindent$\circ~$\textsc{Step 2.}  Let  $|\beta|\leq [p]-1$ and $j\in \{0,\ldots, [p]-1\}$. By using the Leibniz formula for derivatives and adding and subtracting the term $\sum_{0\leq \alpha \leq \beta} \begin{pmatrix}  \beta \\ \alpha  \end{pmatrix} \int_{\R^n} D^\alpha J(x_1-z) D^{\beta-\alpha} \partial_t^j (A_p(x_2,z,t^*) ) \,dz$ we get
\begin{align*}
\begin{split}
|D^\beta \partial_t^{j} &u(x_1,t) - D^\beta \partial_t^{j}u(x_2,t)| \\
&\leq
\int_{\R^n} | J(x_1-z)|  \left|
D^{\beta} \partial_t^{j-1} (A_p(x_1,z,t^*))  ) -
D^{\beta} \partial_t^{j-1} (A_p(x_2,z,t^*) )
\right|
\,dz\\
&\quad +\sum_{1\leq \alpha \leq \beta} \begin{pmatrix} \beta \\ \alpha  \end{pmatrix} \int_{\R^n} | D^\alpha J(x_1-z)|  \left|
D^{\beta-\alpha} \partial_t^{j-1} (A_p(x_1,z,t^*))  ) -
D^{\beta-\alpha} \partial_t^{j-1} (A_p(x_2,z,t^*) )
\right|
\,dz\\
&\quad +
\sum_{0\leq \alpha \leq \beta} \begin{pmatrix} \beta \\ \alpha  \end{pmatrix} \int_{\R^n} \left| D^\alpha J(x_1-z) - D^\alpha J(x_2-z) \right|
|D^{\beta-\alpha} \partial_t^{j-1} (A_p(x_2,z,t^*))|  \,dz.
\end{split}
\end{align*}
From the above expression and \eqref{el.step1} we can  proceed analogously to the proof of  Theorem \ref{teo.holder.cont} to obtain that
\begin{align} \label{edoo}
\begin{split}
\left|D^\beta \partial_t^{j}u(x_1,t)- D^\beta \partial_t^{j} u(x_2,t)\right| \leq A_1 \left|D^{\beta} \partial_t^{j-1} u(x_1,t^*) - D^{\beta} \partial_t^{j-1} u(x_2,t^*) \right| + B_1 \, \tilde \omega(|x_1-x_2|) ,
\end{split}
\end{align}
where $A_1=A_1(m_0,p,J)$,  $B_1=B_1((p,n,J, \mathsf{m}_0,\mathsf{M}_p,\mathsf{\bar m}_{p}(t)))$, and  $\tilde \omega(\rho):= \max\{\rho, \rho^{p-2}, \omega_{J,p}(\rho)\}$.

\noindent$\circ~$\textsc{Step 3.} We can iterate \eqref{edoo} to get
\begin{align*}
|D^\beta \partial_t^{j}u(x_1,t)&- D^\beta \partial_t^{j} u(x_2,t)|\leq \\
&\leq A_1 \left|D^{\beta} \partial_t^{j-1} u(x_1,t^*_1) - D^{\beta} \partial_t^{j-1} u(x_2,t^*_1) \right| + B_1 \tilde \omega(|x_1-x_2|)\\
&\leq A_1 \left( A_2 \left|D^{\beta} \partial_t^{j-2} u(x_1,t^*_2) - D^{\beta} \partial_t^{j-2} u(x_2,t^*_2) \right|  + B_2 \tilde \omega(|x_1-x_2|) \right)+ B_1 \tilde \omega(|x_1-x_2|)\\
&\qquad \vdots\\
&\leq A_1\cdots A_j |D^\beta u(x_1,t^*_j) - D^\beta(x_2,t^*_j)|\\
&\quad + \left( A_1\cdots A_{j-1} B_j + A_1\cdots A_{j-2} B_{j-1} + A_1\cdots A_{j-3} B_{j-2} + \cdots + B_1\right) \tilde \omega(|x_1-x_2|)\\
&\leq \mathsf{A}  |D^\beta u(x_1,t^*_j) - D^\beta(x_2,t^*_j)| + \mathsf{B} \tilde \omega(|x_1-x_2|)\\
&\leq
(\mathsf{A} \, \mathsf{K}(t) + \mathsf{B} )\, \bar \omega(|x_1-x_2|)
\end{align*}
where $\mathsf{A}=A_1\cdots A_j$ and $\mathsf{B}$ depends on $A_i, B_i$, $i=1,\ldots, j$, and in the last inequality we have used  \eqref{estimate.holder}   and denoted $\bar \omega(\rho):= \max\{\rho, \rho^{p-2},  \omega_{J,p}(\rho),\omega(\rho) \}$. This gives that
$$
\sup\left\{ \frac{|D^\beta \partial_t^{j} u(x_1,t)- D^\beta \partial_t^{j} u(x_2,t)| }{\bar\omega(|x_1-x_2|)} \colon x_1,x_2\in \R^n, x\neq y\right\}
\leq
(\mathsf{A} \, \mathsf{K}(t) + \mathsf{B} ):={\mathsf c}(t)
$$
and therefore $D^\beta \partial_t^{j}u(\cdot,t)\in C^{0,\bar \omega}(\R^n)$ for all $|\beta| \leq  [p]-1$ and $j=0,\ldots, [p]-1$.
\end{proof}

We prove now that spatial derivatives of order up to $[p]-1$ of solutions are in the class $C_t^{[p]-1}$.
\begin{thm} \label{teo.holder.cont.2}
Given $p> 2$, let $u$ be a solution of \eqref{eqC} starting from the initial datum $u_0\in \LL^q(\R^n)$ with $q\in [1,\infty]$. Assume moreover conditions \eqref{mp.def}  and  \eqref{mod.J'}.

It holds that $D^\beta u(x,\cdot) \in C_t^{[p],p-[p]}([t_0,\infty))$ for any multi-index $\beta$ such that $|\beta|\leq [p]-1$ and any $t_0>0$. More precisely, there exists a positive constant $\mathsf{c}_p$ depending on $p$, $n$, $\|J\|_{\LL^\infty(\R^n)}$, $\mathsf{m}_0$ and $t_0\leq t_1<t_2<\infty$ such that
\begin{align*}
\max_{j=0,\ldots,[p]-1}  \frac{|D^\beta \partial_t^{j} u(x,t_1)- D^\beta \partial_t^{j} u(x,t_2)| }{|t_1-t_2|} +  \frac{|D^\beta \partial_t^{[p]} u(x,t_1)- D^\beta \partial_t^{[p]} u(x,t_2)| }{|t_1-t_2|^{p-[p]}}
\leq \mathsf c_p
\end{align*}
holds for any $x\in \R^n$ and $t_0\leq t_1<t_2<\infty$.
\end{thm}

\begin{proof}
We prove the result for solutions of the Cauchy problem. For the Dirichlet and Neumann case the proof is analogous.

Let $u$ be a solution of \eqref{eqC} with $p>2$ corresponding to the initial datum $u_0\in C^0(\R^n)\cap \LL^q(\R^n)$ with $q\in [1,\infty]$ and let $t\geq t_0>0$ (we can let $t_0=0$ when $q=\infty$).

Let $\beta$ such that $|\beta|\leq [p]-1$ and $j\in \{0,\ldots, [p]-1\}$. By using \eqref{deriv.b.t} we have the
\begin{align*}
|D^\beta \partial_t^{j} &u(x,t_1) - D^\beta \partial_t^{j}u(x,t_2)|\\
&=\sum_{0\leq \alpha \leq \beta} \begin{pmatrix} \beta \\ \alpha  \end{pmatrix} \int_{\R^n} | D^\alpha J(x-z)|   |
D^{\beta-\alpha} \partial_t^{j-1} A_p(x,z,t_1^*)   -
D^{\beta-\alpha} \partial_t^{j-1} A_p(x,z,t_2^*)
|
\,dz,
\end{align*}
where $t_1^*\in (t_1,t_1+h)$ and $t_2^*\in (t_2,t_2+h)$ for some $h>0$, and $A_p(x,z,t)=L_p(u(z,t)-u(x,t))$.

We proceed as in the proof of  Theorem \ref{prop.hi.ut}: since $|t_1^*-t_2^*|\leq 2|t_1-t_2|$  we choose $h\leq |t_1-t_2|$ and using \eqref{el.step1}, we get
\begin{align*}
|D^{\beta-\alpha} \partial_t^{j-1} A_p(x,z,t_1^* )  -
D^{\beta-\alpha} \partial_t^{j-1}  A_p(x,z,t_2^*) | &\leq
|D^{\beta-\alpha} \partial_t^{j} A_p(x,z,t)| |t_1^* -t_2^*|\\
&\leq
\begin{cases}
\mathsf{c}(t^*) |t_1-t_2| &\text{ when } j\leq [p]-1\\
\mathsf{c}(t^*) |t_1-t_2|^{p-[p]} &\text{ when } j= [p],
\end{cases}
\end{align*}
where $t_* \in (t_1^*,t_2^*)$. This concludes the proof.
\end{proof}

As a consequence of Theorems \ref{teo.holder.cont.1} and \ref{teo.holder.cont.2} we finally obtain:
\begin{thm} \label{teo.holder.cont.final}
Given $p> 2$ let $u$ be a solution of \eqref{eqC} starting from the initial datum $u_0\in \LL^q(\R^n)$ with $q\in [1,\infty]$. Assume moreover conditions \eqref{mp.def}  and  \eqref{mod.J'}.

Then, for any $t_0>0$ $(t_0=0$ when $u_0\in \LL^\infty(\R^n)$ it holds that
$$
u\in C_x^{[p]-1,\bar \omega}(\R^n)\cap C_t^{[p],p-[p]}([t_0,\infty)).
$$
More precisely, there exists a positive constant $\mathsf{c}_p$ depending on $p$, $n$, $\|J\|_{\LL^\infty(\R^n)}$, $\mathsf{m}_0$, $\mathsf{M}_p$ and $t_0\leq t_1<t_2<\infty$ such that
\begin{align*}
\max_{\stackrel{|\beta|< [p]-1}{j=0,\ldots, [p]-1}} \frac{|D^\beta_x \partial_t^j u(x_1,t_1) - D^\beta_x \partial_t^j u(x_2,t_2)|}{\bar\omega(|x_1-x_2|)+ |t_1-t_2|}+ \max_{|\beta|< p-1}\frac{|D_x^\beta \partial_t^{[p]}u(x_1,t_1)-D_x^\beta \partial_t^{[p]}u(x_2,t_2) |}{\bar\omega(|x_1-x_2|) + |t_1-t_2|^{p-[p]}} \le \mathsf{c}_p
\end{align*}
holds for any $x_1,x_2\in\R^n$ and $t_1>t_2\geq t_0$.

In particular, when $p\in \N$ we have that $u\in C^\infty_x(\R^n)\cap C_t^\infty([t_0,\infty))$. Hence, each $D^\beta u(x, \cdot)$ is analytic in time, more precisely, there exists $c_p>0$ independent of $k$
and $t_0\leq t_1<t_2<\infty$ such that
$$
\frac{|D^\beta_x \partial_t^k u(x_1,t_1) - D^\beta_x \partial_t^k u(x_2,t_2)|}{\bar\omega(|x_1-x_2|)+ |t_1-t_2|} \leq \mathsf{c}_p \, k!
$$
holds for any $\beta\in \N_0^n$ and $k\in \N$, and the analyticity radius is $r_0=1\wedge t_0$.

\end{thm}

\section{Some consequences} \label{section7}

In this section, we present several consequences derived from Theorems \ref{main} and \ref{main2}. Specifically, we explore ultra-contractivity estimates, examine the asymptotic behavior of solutions for the Dirichlet and Neumann problems, and establish certain functional inequalities.

\subsection{Ultra-contractivity estimates}

\begin{prop}[Smoothing for the difference of solutions]
Let   $u$ and $v$ be solutions \eqref{eqC} corresponding to the initial data $u_0,v_0\in \LL^q(\R^n)$ with $q\in [1,\infty)$, respectively. Then, for any $r\in [1,\infty]$  and $0<t_0\leq  t$ it holds that
$$
\|u(t)-v(t)\|_{\LL^r(\R^n)}  \leq
\mathsf{c}_{p,q,J}^\frac{1}{r} \|u(t_0)-v(t_0)\|_{\LL^q(\R^n)}^\frac{q}{r} \left(  t^{-\frac{1}{p-2}} +  \|u(t_0)\|_{\LL^q(\R^n)} + \|v(t_0)\|_{\LL^q(\R^n)}\right)^{1-\frac{q}{r}}.
$$
In particular, when $u_0,v_0\in \LL^\infty(\R^n)$ we have that for $0\leq t_0\leq t$
$$
\|u(t)-v(t)\|_{\LL^\infty(\R^n)}  \leq
 t^{-\frac{1}{p-2}} +  \|u(t_0)\|_{\LL^\infty(\R^n)} + \|v(t_0)\|_{\LL^\infty(\R^n)}.
$$

A similar result holds for solutions of \eqref{eqD} when $\Omega$ is open and bounded, and for solutions of \eqref{eq} when $\Omega$ in addition  satisfies \eqref{HJ}.
\end{prop}

\begin{proof}
Let $u$ and $v$ be solutions of \eqref{eqC} corresponding to the initial data $u_0,v_0\in \LL^q(\R^n)$ with $q\in [1,\infty)$. Then,  for any $r\in [1,\infty)$ it holds that
\begin{align*}
\|u(t)-v(t)\|_{\LL^r(\R^n)}  &= \left(\int_{\R^n} |u(x,t)-v(x,t)|^{r-q} |u(x,t)-v(x,t)|^q\,dx\right)^\frac1r\\
&\leq \|u(t)-v(t)\|_{\LL^q(\R^n)}^\frac{q}{r} \|u(t)-v(t)\|_{\LL^\infty(\R^n)}^{1-\frac{q}{r}} \\
&\leq
\|u(t)-v(t)\|_{\LL^q(\R^n)}^\frac{q}{r} \left(\|u(t)\|_{\LL^\infty(\R^n)} + \|v(t)\|_{\LL^\infty(\R^n)} \right)^{1-\frac{q}{r}}.
\end{align*}
The result follows just by applying Proposition \ref{norm.decreasing.difference} and Theorem \ref{main}. Observe that the estimate does not degenerate as $r\to\infty$.

The proof for the Dirichlet and Neumann case follows similarly by using Theorem \ref{main2}. This concludes the proof.
\end{proof}

\subsection{Asymptotic behavior on domains}
In this section we analyze the asymptotic behavior of solutions of the Dirichlet and Neumann problem. For that end, we recall the following Poincar\'e  type inequalities  established in Proposition 6.19 and 6.25 in \cite{AMRT}.
\begin{prop}[Poincar\'e type inequality] \label{poincare}
Given $p\geq 1$ and $\Omega\subset \R^n$ open and bounded, for every $u\in \LL^p(\Omega)$ it holds that
$$
\int_\Omega |u-\overline u|^p\,dx \leq \mathsf{c}_1 \iint_{\Omega\times\Omega} J(x-y)|u(x)-u(y)|^p\,dxdy
$$
where $\overline u=\frac{1}{|\Omega|}\int_\Omega u(x)\,dx$ and $\mathsf{c}_1=\mathsf{c}_1(J,p,\Omega)$.

When $u=0$ in $\R^n \setminus \Omega$, it holds that
$$
\int_\Omega |u|^p\,dx \leq \mathsf{c}_2 \iint_{\Omega\times\Omega} J(x-y)|u(x)-u(y)|^p\,dxdy
$$
with $\mathsf{c}_2=\mathsf{c}_2(J,p,\Omega)$.
\end{prop}
The following result holds as a consequence of Proposition \ref{poincare} (see Theorems 6.22 and 6.35 in \cite{AMRT}).

\begin{prop}[Asymptotic behavior of solutions] \label{decay2p}
Let $1<p<\infty$ and $\Omega\subset \R^n$ be open and bounded. Given a solution $u$ of \eqref{eq} corresponding to the  initial datum $u_0\in \LL^\infty(\Omega)$, then it holds that
$$
\|u(t)-\overline u_0 \|_{\LL^p(\Omega)}^p \leq \frac{\mathsf{c}_1}{t} \|u_0\|_{\LL^2(\Omega)}^2, \qquad \forall t>0,
$$
where $\overline u_0=\frac{1}{|\Omega|}\int_\Omega u_0(x)\,dx$.  Moreover, given a solution of \eqref{eqD} with initial datum $u_0\in \LL^\infty(\Omega)$, then
$$
\|u(t)\|_{\LL^p(\Omega)}^p \leq \frac{\mathsf{c}}{t} \|u_0\|_{\LL^2(\Omega)}^2, \qquad \forall t>0,
$$
where $\mathsf{c}$ depends of $\mathsf{c}_2$, being  $\mathsf{c}_1$ and $\mathsf{c}_2$ are the constants given in Proposition \ref{poincare}.
\end{prop}

Using the smoothing effect of solutions stated in Theorem \ref{main2}, we extend the asymptotic behavior of solutions to integrable initial data.

\begin{thm}[Asymptotic behavior of solutions, Neumann case] \label{decay.N}
Let $p>2$ and $\Omega\subset \R^n$ be open and bounded,  and assume \eqref{HJ}. Given a solution $u$  of \eqref{eq} corresponding to the   initial datum $u_0\in \LL^1(\Omega)$, then it holds that
$$
\|u(t)-\overline u_0\|_{\LL^\infty(\Omega)} \leq \mathsf{c} t^{-\frac1p} \quad   \forall t\gg 1,
$$
where $\overline u_0=\frac{1}{|\Omega|}\int_\Omega u_0(x)\,dx$  and $\mathsf{c}$ is a positive constant depending on $J$, $n$, $p$, $\kappa_{J,\Omega}$, $|\Omega|$ and $\|u_0\|_{\LL^1(\Omega)}$.
\end{thm}

\begin{proof}
Let $\Omega$ be open and bounded  satisfying \eqref{HJ} and  let $u$ be solution of \eqref{eq} corresponding to the  initial datum $u(t)\in \LL^1(\Omega)$. By Proposition \ref{decay2p}, for any $t\geq>0$,
$$
\| u(t)-\overline{u}_0\|_{\LL^p(\Omega)}^p \leq \frac{\mathsf{c}_1}{t}\|u_0\|_{\LL^2(\Omega)}^2
$$
where $\overline{u}_{0}=\frac{1}{|\Omega|}\int_\Omega u_0(x)\,dx$. Them, using  Theorem \ref{main2} we get that for $t>0$,
\begin{equation} \label{poincare.mejorado}
\| u(t)-\overline{u}_{0}\|_{\LL^p(\Omega)}^p
\leq
\frac{\mathsf{c}_1}{t} |\Omega|\|u_0\|_{\LL^\infty(\Omega)}^2
\leq
\frac{\mathsf{c}_1}{t} |\Omega| \left(t^{-\frac{1}{p-2}} + \|u_0\|_{\LL^1(\Omega)} \right)
\end{equation}
where $\mathsf{c}$ depends of $p$, $n$, $J$ and $\kappa_{J,\Omega}$.

Observe that   $w:= u -\overline u_0$ solves the same equation than $u$ but with initial datum $w_0=u_0-\overline u_0$. Then, from Theorem \ref{main2}, Proposition \ref{norm.decreasing.difference} and H\"older's inequality, we get that for any $t>0$
\begin{align*}
\|w(t)\|_{\LL^\infty(\Omega)} &\leq
\mathsf{c} \left(    t^{-\frac{1}{p-2}}  + \|w(\tfrac{t}{2})\|_{\LL^1(\Omega)} \right) \\
&\leq
\mathsf{c} \left(t^{-\frac{1}{p-2}}  + |\Omega|^\frac{p-1}{p}\|u(\tfrac{t}{2})-\overline u_0\|_{\LL^p(\Omega)} \right),
\end{align*}
where $\mathsf{c}$ depends of $p$, $n$, $J$ and $\kappa_{J,\Omega}$.  Finally, this together with  \eqref{poincare.mejorado} yields that for $t>2t_0$,
\begin{align*}
\|w(t)\|_{\LL^\infty(\Omega)} &\leq
\mathsf{\bar c} \left(t^{-\frac{1}{p-2}}  +  t^{-\frac{1}{p}}\left(t^{-\frac{1}{p-2}} + \|u_0\|_{\LL^1(\Omega)} \right)^\frac1p \right)\\
&\leq
\mathsf{\bar c} t^{-\frac{1}{p}} ( 2 + \|u_0\|_{\LL^1(\Omega)}^\frac1p ).
\end{align*}
where $\mathsf{\bar c}$ depends of $\mathsf{c}_1$, $\mathsf{c}$, and $|\Omega|$, since $t^{-\frac{1}{p-2}} \leq t^{-\frac1p}$ for $t\gg 1$. This completes the proof.
\end{proof}

\begin{rem}
Expression \eqref{poincare.mejorado} can be viewed as a refined version of the Poincar\'e-type inequality presented in Proposition \ref{decay2p}, now applicable to $\LL^1$ data.
\end{rem}

With the same arguments an analogous result can be obtain for solutions of \eqref{eqD}.
\begin{thm}[Asymptotic behavior of solutions,  Dirichlet case] \label{decay.D}
Let $p>2$ and $\Omega\subset \R^n$ be open and bounded. Given a solution $u$  of \eqref{eqD} corresponding to the initial datum $u_0\in \LL^1(\Omega)$, then it holds that
$$
\|u(t)\|_{\LL^\infty(\Omega)} \leq \mathsf{c} t^{-\frac1p} \quad   \forall t\gg 1,
$$
where $\mathsf{c}$ is a positive constant   depending on $J$, $n$, $p$,  $|\Omega|$ and $\|u_0\|_{\LL^1(\Omega)}$.
\end{thm}

\subsection{Functional inequalities}

As a consequence of the smoothing effect, the following functional inequality can be derived.
\begin{prop}
Let $\Omega\subseteq \R^n$. Then for any $u\in \LL^q(\R^n)$, $1\leq q <\infty$. Then
$$
\|u\|_{\LL^2(\R^n)}^2 \leq (\mathsf{\tilde K}_p + 2)\max\{ \mathcal{E}_p(u,u)^\frac{1}{p-1}, \mathcal{E}_p(u,u)^{-\frac{p-2}{p-1}} \}  + \mathsf{K}_{p,q,J}\|u\|_{\LL^q(\R^n)},
$$
where $\mathsf{\tilde K}_p$ and $\mathsf{K}_{p,q,J}$ are given in Theorem \ref{main}.
\end{prop}

\begin{proof}
Let $u$ be a solution of \eqref{eqC} with initial datum $u_0\in \LL^q(\R^n)$, $1\leq q< \infty$.  By Proposition \ref{parts}, for \text{a.e.} $x\in \R^n$, $t>0$, and $v\in \LL^p(\Omega)$ we have that
$$
- \int_{\R^n} u_t(x,t) v(x)\,dx = -\int_{\R^n} v(x) \J_p u(x,t) \,dx = \mathcal{E}_p (u,v).
$$
Since  $u_t(t) \in \LL^p(\R^n)$ for any $t>0$ due to  Theorem \ref{smooth.ut}, we can use the previous relation with $v(x)=u_t(x,t)$ to get
\begin{align*}
\frac{d}{dt} \mathcal{E}_p(u(t),u(t)) &= p \mathcal{E}_p (u(t),u_t(t)) = -p\int_{\R^n} |u_t(x,t)|^2\,dx < 0.
\end{align*}
Therefore $\mathcal{E}_p (u(t),u(t))$ is decreasing in time. In particular, $\mathcal{E}_p (u(t),u(t)) \leq \mathcal{E}_p (u_0,u_0)$, which gives
\begin{align*}
\frac{d}{dt}\|u(t)\|_{\LL^2(\R^n)}^2 &=
2\int_{\R^n} u(x,t) u_t(x,t)\,dx =-2  \mathcal{E}_p (u(t),u(t))
\geq -2 \mathcal{E}_p (u_0,u_0).
\end{align*}
Integrating the inequality above and using Theorem \ref{main} gives
\begin{align*}
-2t \mathcal{E}_p (u_0,u_0) &\leq \|u(t)\|_{\LL^2(\R^n)}^2 - \|u_0\|_{\LL^2(\R^n)}^2\\
&\leq \mathsf{\tilde K}_p   t^{-\frac{1}{p-2}}   + \mathsf{K}_{p,q,J}\|u_0\|_{\LL^q(\R^n)}  - \|u_0\|_{\LL^2(\R^n)}^2,
\end{align*}
from where, choosing $t=(\mathcal{E}_p (u_0,u_0))^{-\frac{p-2}{p-1}}$ we get
$$
\|u_0\|_{\LL^2(\R^n)}^2 \leq (\mathsf{\tilde K}_p + 2)\max\{ \mathcal{E}_p(u_0,u_0)^\frac{1}{p-1}, \mathcal{E}_p(u_0,u_0)^{-\frac{p-2}{p-1}} \}  + \mathsf{K}_{p,q,J}\|u_0\|_{\LL^q(\R^n)},
$$
and this concludes the proof.
\end{proof}

\appendix

\section{Gradient flow and EVI solutions} \label{appendix.1}
The Hille-Yosida theorem is a powerful tool when studying the linear semigroups theory and its connection with linear evolution problems. The study of nonlinear physical models led to the development of nonlinear functional analysis. In particular, the nonlinear semigroup theory was dealt and developed by  Br\'ezis, Benilan, Crandall, Kato, Komura, Pazy, among others. In this context we will consider \emph{gradient flow} or \emph{mild solutions}. For a nice introduction to this subject we refer to \cite{ABS21}.
	
\begin{defn}[Gradient flow]
We say that $u\colon(0,\infty)\to D(f)$ is a \emph{gradient flow} of $f$ if $u\in AC_{loc}((0,\infty),H)$ and
$$
u'(t)\in  -\partial  f(u(t)) \quad a.e.\; t\in (0,\infty)
$$
We say that $u(t)$ starts from $u_0$ if $u(t)\to u_0$ as $t\to 0^+$.
\end{defn}

The following result provides for  gradient flow solution to   nonlinear problem. For a proof we refer for instance to \cite{ABS21}.

\begin{thm}[Brezis-Komura] \label{BK}
Assume that $f$ is convex and lower semicontinuous. For every $u_0 \in \overline{D(f)}$ there exists a unique gradient flow $u(t)=S_t u_0$ starting from $u_0$. The family of operators $S_t\colon \overline{D(f)}\to D(f)$, $t>0$, satisfies the semigroup property $S_{t+s}=S_t\circ S_s$ and the contractivity property
$$
\| S_t u_0 - S_t y_0 \| \leq \|u_0-y_0\| \quad u_0,y_0 \in \overline{D(f)}.
$$
\end{thm}

The \emph{evolution variational inequality} (EVI) formulation of gradient flows provides for a useful mechanism to find solutions.
\begin{defn}
A curve $u\in AC_{loc}((0,\infty),H)$ is called an EVI solution if for any $v\in Dom(f)$ one has
\begin{equation} \label{es.evi}
\frac12\frac{d}{dt}\left(\|u(t)-v\|^2 \right) \leq f(v)-f(u(t)) \quad a.e.\; t\in (0,\infty).
\end{equation}
We say that $u(t)$ starts from $u_0$ if $u(t)\to u_0$ as $t\to 0^+$.
\end{defn}
We recall that gradient flow solutions are equivalent to EVI solutions.

\begin{prop} \label{equivalence}
For a convex and lower semicontinuous function $f$ with $D(f)\neq\emptyset$, a locally absolutely continuous curve $u$ is a gradient flow if and only if it is an $EVI$ solution.
\end{prop}

\section*{Acknowledgments}

M. Bonforte was partially supported by the Projects MTM2017-85757-P and PID2020-113596GB-I00 (Spanish Ministry of Science and Innovation). M. Bonforte moreover acknowledges financial support from the Spanish Ministry of Science and Innovation, through the ``Severo Ochoa Programme
for Centres of Excellence in R\&D'' (CEX2019-000904-S) and by the European Union’s Horizon 2020 research and innovation programme under the Marie Sk\l odowska-Curie grant agreement no. 777822.
Part of this work was done while M.B. was visiting A. Figalli at ETH Z\"urich (CH) in the year 2023. M.B. would like to thank the FIM (Institute for Mathematical Research) at ETH Z\"urich for the kind hospitality and for the financial support. We are deeply grateful to A. Figalli for stimulating discussion that originated many results of this work, especially the higher regularity part.

A. Salort was partially supported by CONICET under grant PIP 11220150100032CO, by ANPCyT under grants PICT 2016-1022 and PICT 2019-3530 and by the University of Buenos Aires under grant 20020170100445BA. This article was developed during the stay at Universidad Aut\'onoma de Madrid supported by CONICET under the program ``Becas externas postdoctorales para j\'ovenes investigadores".

This work has been supported by the Madrid Government (Comunidad de Madrid – Spain) under the multiannual Agreement with UAM in the line for the Excellence of the University Research Staff in the context of the V PRICIT (Regional Programme of Research and Technological Innovation).

Finally, the authors are grateful to F. Del Teso for the stimulating discussions, that originated the Figures \ref{fig.1} and \ref{fig.2}, obtained as a modification of the MatLab code developed in the paper   \cite{DTL3}.

\noindent{\bf Conflict of interest statement. }On behalf of all authors, the corresponding author states that there is no conflict of interest.

\noindent{\bf Data Availability Statements. }All data generated or analysed during this study are included in this published article.

\end{document}